\documentclass[11pt, a4paper, twoside]{article}

\RequirePackage{amssymb,amsthm,amsmath}
\usepackage[colorlinks=true,linkcolor=blue, citecolor=red]{hyperref}
\usepackage[nameinlink,capitalize]{cleveref}
\usepackage[dvipsnames]{xcolor}
\usepackage{algorithm, algorithmicx, algpseudocode}
\usepackage{graphicx}
\usepackage{subfig}
\usepackage{amsfonts,amsmath,amssymb,amsthm,enumerate,dsfont,bbm}
\crefname{equation}{}{}
\numberwithin{equation}{section}

\usepackage[mathscr]{euscript}
\usepackage[font=footnotesize,labelfont=bf]{caption}
\usepackage{a4wide}
\usepackage{authblk}
\usepackage{comment}

\theoremstyle{plain}
\newtheorem{lemma}{Lemma}[section]
\newtheorem{prop}[lemma]{Proposition}
\newtheorem{thm}[lemma]{Theorem}
\newtheorem{cor}[lemma]{Corollary}
\newtheorem{hyp}[lemma]{Assumption}
\newtheorem{althyp}[lemma]{Assumption}
\theoremstyle{definition}
\newtheorem{df}[lemma]{Definition}
\newtheorem{ex}[lemma]{Example}
\newtheorem{remark}[lemma]{Remark}
\newtheorem{convention}[lemma]{Convention}
\newcommand{\od}{\mathrm{d}}
\newcommand{\E}{{\mathbb{E}}}
\renewcommand{\P}{{\mathbb{P}}}
\newcommand{\R}{{\mathbb{R}}}
\newcommand{\N}{{\mathbb{N}}}
\newcommand{\cB}{\mathcal{B}}
\newcommand{\cF}{\mathcal{F}}
\newcommand{\cG}{\mathcal{G}}
\newcommand{\cP}{\mathcal{P}}
\newcommand{\cT}{\mathcal{T}}
\newcommand{\cY}{\mathcal{Y}}
\newcommand{\cZ}{\mathcal{Z}}
\newcommand{\cN}{\mathcal{N}}
\newcommand{\cW}{\mathcal{W}}
\newcommand{\rwB}{{B}^n}
\newcommand{\tY}{{\sf Y}}
\newcommand{\tZ}{{\sf Z}}
\newcommand{\bx}{{\bf x}}
\newcommand{\sptext}[3]{\hspace{#1 em}\mbox{#2}\hspace{#3 em}}
\newcommand{\rset}{\mathbb{R}}
\newcommand{\nset}{\mathbb{N}}
\newcommand{\ep}{\varepsilon}
\newcommand{\ind}{\mathbf{1}}
\newcommand{\e}{\mathbb{E}}
\newcommand{\p}{\mathbb{P}}
\newcommand{\m}{\mathcal} 
\newcommand{\ot}{\overline{t}}
\newcommand{\ut}{\underline{t}}
\newcommand{\uto}{\underline{t_o}}
\newcommand{\os}{\overline{s}}
\newcommand{\us}{\underline{s}}
\newcommand{\ur}{\underline{r}}
\newcommand{\lip}{\text{\rm Lip}}
\newcommand{\gho}{|g|_\ep}   
\newcommand{\fhot}{|f|^{\scriptscriptstyle loc}_{\alpha;t}}  
\newcommand{\fhox}{|f|_{\ep;x}} 
\newcommand{\fhoxt}{|{\sf f}|_{\ep;x}} 
\newcommand{\fn}{{{\sf f}_n}}
\newcommand{\f}{{\sf f}}
\newcommand{\ps}[1]{\left |#1\right |_{(\ep,\beta)}}
\newcommand{\Ps}[2]{|#1|_{(#2,\beta)}}
\newcommand{\Psg}[2]{|#1|_{(#2,\gamma)}}
\newcommand{\Psgb}[2]{|#1|_{(#2,\overline \gamma)}}

\allowdisplaybreaks
\numberwithin{equation}{section}

\title{Convergence rate for random walk approximations of mean field BSDEs\\}

\author[1]{Boualem Djehiche}
\author[2,4]{Hannah Geiss}
\author[2,5,7]{Stefan Geiss}
\author[3]{C\'eline Labart}
\author[2,6]{Jani Nykänen}
\affil[1]{Department of Mathematics, KTH, SE-100 44 Stockholm, boualem@math.kth.se}
\affil[2]{Department of Mathematics and Statistics, P.O.Box 35, FI-40014 University of Jyvaskyla, Finland}
\affil[3]{Univ. Grenoble Alpes, Univ. Savoie Mont Blanc, CNRS, LAMA, 73000 Chambéry, France, celine.labart@univ-smb.fr}
\affil[4]{hannah.r.geiss@jyu.fi}
\affil[5]{stefan.geiss@jyu.fi}
\affil[6]{jani.m.nykanen@jyu.fi  \medskip}

\affil[7]{corresponding author}

\begin{document}

\maketitle
 
\begin{abstract}
We study the rate of convergence w.r.t.~a Wasserstein type distance for
random walk approximations of mean field BSDEs.
Our method does not use the particle method but instead a freezing
technique. We extend results by Briand,
Ch. Geiss, S. Geiss, and Labart [Bernoulli, 27(2) 2021] about  the rate
of convergence of a Donsker-type theorem for  BSDEs
from the classical setting to the mean field setting. In this connection
the mean field setting leads to new phenomena and requires
new techniques that should be of independent interest:
The H\"older continuous terminal condition causes a singularity in time
of the generator when seen as a generator in the
non-mean field setting. To handle this singularity we introduce  a
concept of modified Hölder continuity
by which we  are able to achieve, up to a logarithmic term, the same
polynomial approximation rates as in the classical
non-mean field setting (in fact, already when approximating the Brownian motion itself a logarithmic
term is necessary). Moreover, the exploited freezing technique of the mean field terms
yields to the problem to handle the quantitative behavior of several
different generators. Using BMO-techniques we obtain the rate of convergence for the
integrated gradient process in the scale of Lorentz (type) spaces of exponential type.
\end{abstract}


 \bigskip

\tableofcontents


\section{Introduction and main results}

Given $0\le t_o < T$ and a stochastic basis $(\Omega,\cF,\P,(\cF_s)_{s\in [t_o,T]})$ satisfying the usual conditions,
carrying a 1d-standard Brownian motion $B=(B_s)_{s\in [t_o,T]}$ with $B_0\equiv 0$,
a random variable  $\xi \in L^{2 \ep}( \Omega, \cF, \P)$ for some $\ep \in (0,1]$ that is independent from $B$,
and given a generator
\[ f:  [t_o,T[ \times \rset \times \rset \times\rset \times \m P_2(\rset) \times \m P_2(\rset) \to \rset, \]
we consider  under \cref{H:A1} the discretization of solutions to mean field BSDEs of form
\begin{equation} \label{eq:mainBSDE}
\left .
\begin{array}{r}
   	Y^{t,x}_s  = g\left(B^{t,x}_T\right) + \int_s^T  f\left(r,B^{t,x}_r,Y^{t,x}_r, Z^{t,x}_r, [Y_r^{t_o,\xi}], [Z_r^{t_o,\xi}] \right) dr
    - \int_s^T Z^{t,x}_r\, dB_r \\
    t_o\le t\leq s\leq T
\end{array}
\right \},
\end{equation}
where $g$ is $\ep$-H\"older continuous, 
\[ B^{t,x}_r := x+B_r-B_t, \] 
and $(Y^{t_o, \xi}, Z^{t_o, \xi})$ solves \eqref{eq:mainBSDE} with  $(t, x)$ replaced by $(t_o, \xi)$.
Here  $[\eta]$ stands for the law of the random variable $\eta$ and 
as filtration we choose $\cF_s:= \sigma(\xi) \vee \sigma( B_r-B_{t_0}: r\in [t_o,s] )\vee \cN$ with 
$\cN$ being the $\P$-null sets.
\medskip

We recall, that for $p\ge 1$ the set $\m P_p(\rset)$ consists of probability measures with a finite $p$-th moment, endowed with the Wasserstein distance
\[ \m W_p(\mu,\mu'):=\inf_{\pi}\bigg(\int_{\rset\times\rset}|x-x'|^pd\pi(x,x')\bigg)^{1/p}
   \sptext{1}{for}{1} (\mu,\mu')\in \m P_p(\rset)\times \m P_p(\rset), \]
the infimum being taken over the probability distributions $\pi$ on $\rset\times\rset$ 
whose marginals on $\rset$ are $\mu$ and $\mu'$, respectively.
Notice that if $X$ and $X'$ are $\R$-valued random variables defined on the same probability space with a $p$-th moment, then by definition we have
\begin{equation}
\label{Wasserstein distance_maj}
\m W_p([X],[X'])\leq \Big[\E|X-X'|^p\Big]^{1/p}.
\end{equation}

Mean field BSDEs (also McKean-Vlasov BSDEs in the literature) were introduced in \cite{BLP07,BDLP07}.
In \cite{BDLP07} the Markovian setting is used, in which existence
and uniqueness of the solution  for  terminal conditions of type
$\eta=\E\big[g(x,X_T)\big]_{|x=X_T}$ is verified,  where the forward process $X$ is a stochastic differential equation of mean field type  and
the generator is given by $\E\big[f(s,\Lambda,(X_s,Y_s,Z_s))\big]_{|\Lambda=(X_s,Y_s,Z_s)}$. In \cite{BLP07}, the
 existence and uniqueness  result is extended  to a more general setting and
mean field BSDEs are related to nonlocal partial differential equations. 
\smallskip

Mean field forward BSDEs (mean field FBSDEs in the following) arise naturally in the probabilistic analysis of mean field games or mean field control problems. Mean field games are introduced by \cite{LL06a} and \cite{LL06b}, and independently, by \cite{huang06} to model games with interactions between many similar players. In this theory, each player’s dynamics and cost take into account the empirical distribution of all agents. For example, in \cite{AD19}, the authors model pedestrian crowd by using a mean field game approach and solve the problem with a mean field FBSDE. Optimal control problems of mean field type can be interpreted as large population limits of
cooperative games, where the players collaborate to optimize a joint objective \cite{Lac17}. For example, in \cite{AC21}, the authors solve an optimal portfolio with mean field risk minimization problem. It is therefore a relevant challenge to solve such equations numerically.
\smallskip

Numerical methods for mean field FBSDEs are of natural interest already for some time 
(see \cite{A15}, \cite{CT15}, \cite{CCD19}).  To solve forward mean field  SDEs typically particle methods are exploited \cite{AKH02,TV03,Bos05}, where the empirical measure of a large number of
interacting particles approximates the mean field term.
However, for backward  equations the numerical implementation yields to a high complexity
when a larger number of particles is used.
The above articles on numerical methods for mean field FBSDEs are not based on particle algorithms:
\cite{CT15} investigates decoupled mean field FBSDE by a cubature method and
\cite{CCD19} develops a numerical method for fully coupled mean field FBSDEs applying recursive Picard iterations
on small time intervals. The scheme is based on a variation of the method of continuation.
In \cite{GMW21}, the authors propose several algorithms to solve mean field FBSDEs with the help of neural
networks and estimate the solution or the gradient of it via minimization problems.
We also refer to \cite{CCSJ21}, \cite{FZ20} and \cite{HHL23} for other papers based on deep learning methods.
In \cite{ABGL22} the authors solve non Markovian mean field BSDEs by using the Wiener chaos expansion and a particle system approximation.
In \cite{ZM23}, the authors propose first order and Crank-Nicholson numerical schemes to solve certain kinds of mean field BSDEs.
They provide $L^p$ error estimates for the proposed schemes.
\bigskip

\paragraph{The random walk BSDE}\ \smallskip

We propose to approximate the solution $(Y^{t, x}, Z^{t, x})$ to \eqref{eq:mainBSDE} using a scaled random walk.
For $n\in\nset:=\{1,2,3,...\}$ we set $h:=\frac{T}{n}$ and consider the time-net $(t_k)_{k=0}^n$ 
\[ t_k:= kh.\]
For $t\in [0,T]$ we set
\[ \ut := \max \left\{ t_k \mid k \in \left\{ 0, ..., n \right\}  \textrm{ and } t \geq t_k  \right\} \sptext{1}{and}{1}
   \ot := \inf \left\{ t_k \mid k \in \left\{ 0, ..., n \right\} \textrm{ and } t \leq t_k  \right\}\]
so that $0\le \ut \le t \le \ot \le T$.
For $n\ge n_0 > \frac{2T}{T-t_o}$ we let $j\in \{0,1,\ldots,n-3\}$ such that $\underline{t_o}=t_j$
and use as scaled random walk
\begin{equation}\label{eqn:definition:Btn}
	\rwB_{t_j}\equiv 0 \sptext{1}{and}{1} 	\rwB_{t_k} := \sqrt{h}\, \sum_{m=j+1}^k \zeta_m,
\end{equation}
where $(\zeta_k)_{k=j+1}^n$ is an i.i.d. sequence of Rademacher random
variables (taking values $-1$ and $1$ with probability $1/2$). 
The quadratic variation  of  $(\rwB_t)_{t \in [0,T]}$  is deterministic and it holds
\[ \langle \rwB \rangle_s = h \sum_{k=j+1}^n \delta_{t_k}((0,s]). \]
We let
\[ \rwB_t := \rwB_{\ut} \sptext{1}{for}{1} t \in [\uto,T]
   \sptext{1}{and}{1}
   B_s^{n,t,x} := x+\rwB_s- \rwB_t \sptext{1}{for}{1} x\in\rset
                                   \sptext{1}{and}{1} \uto \leq t\leq s\leq T. \]
For $t\in [t_o,T[$ the random walk scheme associated to \eqref{eq:mainBSDE} is
\begin{equation}\label{eq:mainBSDEn-hat}
\left . \begin{array}{r}
    Y^{n,\ut,x}_s
  = g( B^{n,\ut,x}_T)+\int\limits_{]s,T]} f(r\wedge t_{n-1},  B^{n,\ut,x}_{r-}, Y^{n,\ut,x}_{r-},  Z^{n,\ut,x}_r,\mu_r,\nu_r ) d\langle \rwB \rangle_r
   - \int\limits_{]s,T]}  Z^{n,\ut,x}_r d\rwB_r \\
\hspace*{-6em} t_o\le t \le s \le T
\end{array} \right \}
\end{equation}
\smallskip

where $\mu_s:= [Y^{n,t_o, \xi}_r]$ and $\nu_r =[Z^{n,t_o, \xi}_r]$
are the laws of the global solution to the equation with the terminal condition $g( B^{n,\ut,\xi}_T)$.
In the generator $f$ we use $r\wedge t_{n-1}$ as $f$ is  defined on
$[t_o,T[\times\rset\times\rset\times\rset  \times \cP_{2}(\R) \times \cP_2(\R) \to\rset $ only.
Both schemes, \eqref{eq:mainBSDE} and \eqref{eq:mainBSDEn-hat}, can be seen as functional maps
\begin{align}
\cT^n:(B_s^{n,\ut,x})_{s\in [\ut,T]} & \longmapsto  \Big ( (Y^{n,\ut,x}_s)_{s\in [\ut,T]},(Z^{n,\ut,x}_s)_{s\in ]\ut,T]} \Big ), \label{eqn:map_Tn}\\
  \cT:(B_s^{t,x})_{s\in [t,T]} \hspace{.8em}      & \longmapsto  \Big ( (Y^{t,x}_s)_{s\in [t,T]},(Z^{t,x}_s)_{s\in [t,T[}\Big ). \label{eqn:map_T}
\end{align}
The operators $\cT,\cT^n$ depend on the data $(g,f,\xi)$ of the BSDE \eqref{eq:mainBSDE}, and
one has that
\begin{align*}
   \Big ( (Y^{t,\xi}_s)_{s\in [t,T]},(Z^{t,\xi}_s)_{s\in [t,T[}\Big )
&   = \Big ( (Y^{t,x}_s)_{s\in [t,T]},(Z^{t,x}_s)_{s\in [t,T[}\Big )_{x=\xi},\\
   \Big ( (Y^{n,\ut,\xi}_s)_{s\in [\ut,T]},(Z^{n,\ut,\xi}_s)_{s\in [\ut,T[}\Big )
&   = \Big ( (Y^{n,\ut,x}_s)_{s\in [\ut,T]},(Z^{n,t,x}_s)_{s\in [\ut,T[}\Big )_{x=\xi}.
\end{align*}

\paragraph{Previous results on random walk approximations}\ \smallskip

One of the first studies on the approximation of standard BSDEs (i.e. $f$ may depend on $Y$ and $Z$, but not on their laws)
by a random walk scheme is due to Briand, Delyon and Mémin in \cite{MR1817885}. They proposed an approximation based
on Donsker's theorem. They showed that the solution $(Y,Z)$ to a standard BSDE can be approximated by the solution $(Y^n,Z^n)$
to the BSDE
\begin{equation*}
  Y^n_t = G(B^n) + \int_{]t,T]} f(B^n_{s-},Y^n_{s-},Z^n_{s})\, d\langle B^n\rangle_s - \int_{]t,T]} Z^{n}_s\, dB^n_s, \quad 0\leq t\leq T.
\end{equation*}
They proved, in full generality (meaning that $G(B)$ is only
required to be a square integrable random variable) that $(Y^n,Z^n)$ converges weakly
to $(Y,Z)$. However, the question of the rate of convergence was left
open. For general path-dependent terminal conditions $G(B)$ this question is open, the
Markovian case $G(B)=g(B_T)$ has been studied in  \cite{BGGL21}:
Assuming \cref{A1_alt} {\bf without the singularity $(T-t\vee t')^{-\frac{1}{2}}$}
it was shown in \cite[Theorem 12]{BGGL21} that
\begin{align}
\label{eq:was_Y}
\cW_p\left([Y^{t,x}_s], [Y^{n,t,x}_s]\right) & \leq c_\eqref{eq:was_Y}\, (1+|x|)^\ep \, n^{-\alpha \wedge\frac{\ep}{2}},\\
\label{eq:was_Z}
\cW_p\left([Z^{t,x}_s], [Z^{n,t,x}_s]\right)  & \leq  c_\eqref{eq:was_Z}\, \frac{(1+|x|)^\ep}{\sqrt{T-s}} \, n^{- \alpha \wedge\frac{\ep}{2}}
\end{align}
for any $\, p\in [1,\infty [$ and constants $  c_\eqref{eq:was_Y}, c_\eqref{eq:was_Z}  >0$
depending at most on  $(T,\ep,\alpha,f,g,p)$.
This result exploits the PDE structure behind the Markovian case combined with a result of Rio~\cite{Rio09} which directly implies
that there exists a constant $c_p>0$ such that $\cW_p(B^n_1,  \mathcal{G}) \leq c_p\, n^{-1/2}$
for $n\geq 1$ and  $p\geq 1$, when $\mathcal{G}$ is a standard Gaussian random variable.
The rate in \eqref{eq:was_Y} and  \eqref{eq:was_Z} improve  \cite{GLL20a} and \cite{GLL20b}, where the
$L^2$ distance was studied which resulted in an order of $n^{-1/4}$ for smooth coefficients.

\paragraph{Main results}\ \smallskip

We come back \eqref{eqn:map_Tn} and \eqref{eqn:map_T}. For fixed $\xi:\Omega\to \R$ we define
the functions $U^n:[\underline{t_o},T]\times \R\to \R$ and $u:[t_o,T]\times \R\to \R$ by
\[ U^n(t,x) := Y_t^{n,\ut,x}
   \sptext{1}{and}{1}
   u(t,x) := Y_t^{t,x}. \]
Moreover, we use   
\[    \Delta^n: [\underline{t_o},T[ \times \R\to \R   \sptext{1}{by}{1} \Delta^n(t,x)
   := \frac{1}{2 \sqrt{h} } \left [ U^n (\ut+h,x+\sqrt{h}) - U^n (\ut+h,x-\sqrt{h}) \right ].
\]
Exploiting \eqref{eqn:def:mainPDEnf_Deltanf} for the relations
\eqref{eqn:fn:U->Y},
\eqref{eqn:fn:Delta->Z-I},
and  \eqref{YandZfrom-u}
we can take,
as versions of the solution processes,
\begin{align*}
Y_s^{n,t,x} & = U^n(s,B_s^{n,t,x}) \sptext{1}{for}{1} s\in [\ut,T],   \\
Z_s^{n,t,x} & = \Delta^n(s,B_s^{n,t,x}) \sptext{1}{for}{1} s\in ]\ut,T]
                \sptext{1}{and}{1} s \not = \us,   \\
Y_s^{t,x}   & = u(s,B_s^{t,x}) \sptext{2.5}{for}{1} s\in [t,T], \\
Z_s^{t,x}   & = \nabla u(s,B_s^{t,x}) \sptext{1.7}{for}{1} s\in [t,T[.
\end{align*}
So we get explicit path-dependent functionals $\cT^n$ and $\cT$ in  \eqref{eqn:map_Tn} and \eqref{eqn:map_T}.
Let  $\Gamma(B,B^n)$ denote the set of all the processes $(\cB, \cB^n)$  with $\cB$  continuous and $\cB^n$ being
c\`adl\`ag that have the same finite dimensional distributions as  $B$ and $B^n$, respectively.
\medskip

Our first main result is a general coupling result for mean field BSDEs:
\medskip

\begin{thm}\label{statement:main_result_local_version}
Suppose \cref{H:A1}, $\beta>\frac{1}{2}$, $n\ge n_0$, and that
$\|\cdot\|$ is a monotone semi-norm on $(\Omega, \cF, \p)$ with $\|\cdot\| \le \kappa \|\cdot \|_{L^{\exp(1)}}$
for some $\kappa>0$.
Then, for $t \in [t_o,T[$ and $x,y\in \R$,  one has
\begin{equation}\label{eqn:1:statement:main_result_local_version}
     |u(t,x) - U^n(t,y)|
\le c_\eqref{eqn:1:statement:main_result_local_version} \left [ |x-y|^\ep + (1+|x|)^\ep \frac{|\log (n+1)|^\beta}{n^{\alpha \wedge \frac{\ep}{2}}} \right ]
\end{equation}
and, for each coupling  $(\cB, \cB^n)\in \Gamma(B,B^n)$,
\begin{align}
&   \left \| \int_t^T  \Big| \nabla u(s, \cB_s^{t,x}) - \Delta^n(s, \cB^{n,\underline{t},y}_s)\Big|^2  \od s \right \|^\frac{1}{2}
      \notag \\
& \le c_\eqref{eqn:2:statement:main_result_local_version} \left [   \left \| \sup_{t\le s\le T}   \ps  {\cB_s^{t,x}-\cB^{n,\underline{t},y}_s}^2  \right \|^\frac{1}{2}
       + \left ( 1 + \sqrt{\kappa} \, (1+|y|)^\ep \right ) \frac{|\log (n+1)|^{\beta}}{n^{\alpha \wedge \frac{\ep}{2}}} \right ],
       \label{eqn:2:statement:main_result_local_version}
\end{align}
where $c_\eqref{eqn:2:statement:main_result_local_version}, c_\eqref{eqn:1:statement:main_result_local_version}>0$
depend at most on $(\Theta,f_0,M)$.
\end{thm}
\bigskip

The meaning of $n_0$ and of the parameters $(\Theta,f_0,M)$, the constants depend on, are explained in \cref{section_notation} below.
The function $[0,\infty[ \ni r\mapsto \ps r \in [0,\infty[$ is a modification of the function
$r\mapsto r^\ep$ where the behaviour around zero is slightly changed to fit our purposes. Moreover,
we call $\| \cdot \|:L^0(\Omega,\cF,\P)\to [0,\infty]$ a monotone semi-norm provided that
$\|\eta\|=0$ if $\eta=0$ a.s., $\|\eta_1+\eta_2\|\le \|\eta_1\|+\|\eta_2\|$, $\|\lambda \eta\| = |\lambda| \|\eta\|$ if
$\lambda \not =0$, and $\|\eta_1\| \le \|\eta_2\|$ if $|\eta_1| \le |\eta_2|$ a.s.
Furthermore, for a random variable $f:\Omega \to \R$ and $\gamma\in ]0,\infty[$ we use the quasi--norm
\[ \|f\|_{L^{\exp(\gamma)}}
   := \inf \left \{ c > 0 : \P(|f| > \lambda) \le e^{1-\sqrt[\gamma]{\frac{\lambda}{c}}} \mbox{ for }
       \lambda \ge c \right \}.\]
The spaces $L^{\exp(\gamma)}$ measure an exponential tail behaviour of a random variable.
In particular, for all $\gamma,p>0$ one has that
\[  \|f\|_{L^p} \le c_{\gamma,p} \|f\|_{L^{\exp(\gamma)}}.\]
\smallskip

\cref{statement:main_result_local_version} can be applied to any coupling
$(\cB,\cB^n)$ between the Brownian motion and the random walk.
Here there are different options: One can use the Skorohod coupling, which only gives approximately the rate
$n^{-1/4}$ but has the advantage that the coupling can be realized by stopping times
which directly enables the application of stochastic analysis methods.
On the other hand, exploiting the coupling due to J.~Koml\'os, P.~Major and G.~Tusn\'ady
gives approximately the rate $n^{-1/2}$ but there is no joint stochastic basis for
the Brownian motion and the random walk.
Using this coupling we will deduce the following theorem:
\medskip

\begin{thm} \label{coro_final_pathwise_estimate}
Suppose \cref{H:A1}, $\beta > \frac{1}{2}$, and $n \geq n_0$.
Then there is a coupling  $(\cB, \cB^n)   \in \Gamma(B, B^n)$ with 
\begin{align}
     \left \|   \sup_{s \in [t, T]} |\cY^{t, x}_s-{\cY}^{n, t, y}_s|
     \right \|_{L^{\exp(\ep)}}
&\le c_\eqref{eqn:coro_final_pathwise_estimate:Y} \left [ |x-y|^\ep + (1+ |x|)^\ep \frac{|\log (n+1)|^{\ep \vee \beta} } {n^{\alpha \wedge\frac{\ep}{2}}}
     \right ],
     \label{eqn:coro_final_pathwise_estimate:Y} \\
     \left \|  \sqrt{\int_t^T \left| \cZ^{t, x}_s - {\cZ}^{n,t,y}_s \right|^2 \od s} \right \|_{L^{\exp(\beta+\ep)}}
&\le c_\eqref{eqn:coro_final_pathwise_estimate:Z}
     \left [   D_n^\ep \left | 1\vee \log \frac{1}{D_n}\right |^\beta
            + (1 + |y|)^\ep \frac{|\log (n+1)|^{\beta} } {n^{\alpha \wedge\frac{\ep}{2}}} \right ]
     \label{eqn:coro_final_pathwise_estimate:Z} \\
\mbox{for } D_n & := |x-y| + c_0 \frac{\log(n+1)}{\sqrt{n}} \notag
\end{align}
and some $c_0=c_0(T)>0$,
where the processes $({\cY}^{t, x}, {\cZ}^{t, x})$ and $({\cY}^{n, t, y}, {\cZ}^{n, t, y})$
solve \cref{eq:mainBSDE} and \cref{eq:mainBSDEn-hat}
with  driving processes $\cB$ and $\cB^n$, respectively,
and
$c_\eqref{eqn:coro_final_pathwise_estimate:Y},
 c_\eqref{eqn:coro_final_pathwise_estimate:Z}>0$
depend mostly on the parameters $(\Theta,f_0, M)$.
\end{thm}
\bigskip

\cref{coro_final_pathwise_estimate} enables us to approximate the mean field equation
\[  \left .	\begin{array}{r}
    Y^{t_o,\xi}_s  = g\left(B^{t_o,\xi}_T\right) + \int_s^T  f\left(r,B^{t_o,\xi}_r,Y^{t_o,\xi}_r, Z^{t_o,\xi}_r, [Y_r^{t_o,\xi}], [Z_r^{t_o,\xi}] \right) dr
    - \int_s^T Z^{t_o,\xi}_r\, dB_r \\
     t_o \leq s\leq T
    \end{array} \right \} \]
by
\[ \left . \begin{array}{r}
     Y^{n,\uto,\xi^n}_s
   \!\!\!= g(B^{n,\uto,\xi^n}_T)\!+\!\!\int\limits_{]s,T]} f(r\wedge t_{n-1},  B^{n,\uto,\xi^n}_{r^-}, Y^{n,\uto,\xi^n}_{r^-}\!\!\!,
                            Z^{n,\uto,\xi^n}_r\!,[Y^{n,\uto, \xi^n}_r],[Z^{n,\uto, \xi^n}_r]) d\langle \rwB \rangle_r \\
   - \int\limits_{]s,T]}  Z^{n,\uto,\xi^n}_r d\rwB_r, \quad t_o\le s \le T
   \end{array} \right \}, \]
where $\xi^n$ does not necessarily take only finitely many values, although this is the typical application. We also
remark that the random walk BSDE has always a solution by \eqref{eqn:fn:U->Y} and \eqref{eqn:fn:nablaU->Z}
with $x$ replaced by $\xi^n$. We let
\[ u(t_o,x,\mu)  := 	Y_{t_o}^{t_o,x} \sptext{1.8}{if}{1} \mu=[\xi]
   \sptext{1}{and}{1}
   U^n(\uto,x,\mu_n) := 	Y_{t_o}^{n,\uto,x} \sptext{1}{if}{1} \mu_n= [\xi^n]. \]
\medskip
\begin{cor}\label{statement:meanfield_PDE}
For $x,y\in \R$ and probability measures $\mu,\nu$ on $\cB(\R)$ with
$\int_\R |z|^{2\ep} \mu(\od z) \le M$ and $\int_\R |z|^{2\ep} \nu(\od z) \le M$
one has
\begin{equation} \hspace{-.75em}
     \left |u(t_o,x,\mu) - U^n(\uto,y,\nu) \right |
 \le c_\eqref{eqn:1:statement:meanfield_PDE}
     \left [ W^\ep_{2\ep}(\mu,\nu) + |x-y|^\ep +  (1+ |y|)^\ep \frac{|\log (n+1)|^{\ep \vee \beta} } {n^{\alpha \wedge\frac{\ep}{2}}}
     \right ], \label{eqn:1:statement:meanfield_PDE}
\end{equation}
and with the coupling from \cref{coro_final_pathwise_estimate},
\begin{multline}
      \left \| \sqrt{\int_{t_o}^T  \Big| \nabla u(s, B_s^{t_o,x},\mu) - \Delta^n(s, B^{n,\uto,y}_s,\nu)\Big|^2  \od s}  \right \|_{L^2} \\
 \le c_\eqref{eqn:2:statement:meanfield_PDE}
\left [ W^\ep_{2\ep}(\mu,\nu) + |x-y|^\ep + (1+ |y)^\ep \frac{|\log (n+1)|^{\ep+\beta} } {n^{\alpha \wedge\frac{\ep}{2}}}
     \right ],  \label{eqn:2:statement:meanfield_PDE}
\end{multline}
where
$c_\eqref{eqn:1:statement:meanfield_PDE}=c_\eqref{eqn:1:statement:meanfield_PDE}(\Theta,f_0,M)>0$ and
$c_\eqref{eqn:2:statement:meanfield_PDE}=c_\eqref{eqn:2:statement:meanfield_PDE}(\Theta,f_0,M)>0$.
\end{cor}
\smallskip

The proof of Theorems
\ref{statement:main_result_local_version}/\ref{coro_final_pathwise_estimate} and
\cref{statement:meanfield_PDE}
will be given in \cref{subsection_pathwise}.  

\begin{remark}
Regarding \cref{coro_final_pathwise_estimate} we also  have the following lower bound:
Assume $t_o=0$, $f\equiv 0$, $\ep=1$ and $g(x)=x$.
Then we have that ${\cY}^{0,0}=\cB$ and  ${\cY}^{n,0,0}=\cB^n$, so that
\cref{th_random_walk_uniform_convergence} below implies that
\[ \left \| \sup_{s \in [0, T]} |\cY^{0,0}_s-{\cY}^{n, 0, 0}_s|   \right \|_{L^1} \ge \frac{1}{c} \frac{\sqrt{\log(n+1)}}{\sqrt{n}}. \] 
\end{remark}
\medskip

The benefit of the  approach described by Theorems
\ref{statement:main_result_local_version}/\ref{coro_final_pathwise_estimate} and  \cref{statement:meanfield_PDE}
is that it works without the particle method.
Here the  recombining scheme allows to determine the  conditional
expectation in an easy way while usually Monte Carlo methods or other sophisticated procedures are necessary.
This results in a  simple and fast algorithm (see  \cref{algorithm}, in particular \cref{rem_num_scheme}).
If we denote the solution to  \cref{eq:mainBSDE}  for a given $\xi$ by  $Y^{t,x,[\xi]}_t$, the map
$(t,x,[\xi])  \mapsto  Y^{t,x,[\xi]}_t$ can be interpreted as a viscosity solution to the associated  non-local PDE (provided  suitable assumptions such that a solution exists).
 This  associated  non-local PDE was  established  in a more general setting with jumps and
an SDE as a forward process  in \cite{Li17}.  First results in this direction can be found already in \cite{BLP07}.
Hence our approach  provides a method which permits to approximate the PDE solution by a random walk procedure  and an estimate of the approximation rate.
\medskip

The proof of \cref{statement:main_result_local_version} and \cref{coro_final_pathwise_estimate}
relies on a freezing method and the extension of results from \cite{BGGL21}, where we have three steps to
passage from the continuous setting to the discrete one:

\paragraph{The three steps from continuous to discrete} \ \medskip

By freezing the distributions  $[Y^{t_o,\xi}_s]$ and $[Z^{t_o,\xi}_{s\wedge t_{n-1}}]$ we introduce 
\begin{align} \label{f-zero}
  {\sf f}_n(r,x,y,z) :=f(r\wedge t_{n-1} ,x,y,z,[Y^{t_o,\xi}_r],  [Z^{t_o,\xi}_{r\wedge t_{n-1}}])
\end{align}
and consider two more BSDEs: One driven by  a Brownian motion
  \begin{align}  \label{eq:mainBSDEn-tilde}
  Y^{t,x}_{{\sf f}_n,s}  = g\left(B^{t,x}_T\right) + \int_s^T {\sf f}_n\left(r,B^{t,x}_r, Y^{t,x}_{{\sf f}_n,r},  Z^{t,x}_{{\sf f}_n,r} \right) dr  - \int_s^T Z^{t,x}_{{\sf f}_n,r} \, dB_r \sptext{1.5}{for}{.5} s \in [t,T] 
  \end{align}
and another one driven by a random walk
  \begin{align}\label{eq:mainBSDEn}
      Y^{n,\ut,x}_{{\sf f}_n,s}
   &= g(B^{n,t,x}_T) +  \int_{]s,T]} {\sf f}_n(r, B^{n,t,x}_{r-},Y^{n,\ut,x}_{{\sf f}_n,r-},  Z^{n,\ut,x}_{{\sf f}_n,r}) \, d\langle \rwB \rangle_r - \int_{]s,T]} Z^{n,\ut,x}_{{\sf f}_n,r}\, d \rwB_r  \sptext{1.5}{for}{.5} s \in [\ut,T].
  \end{align}

Then we split our problem into three parts by
\begin{align}\label{eq:error_terms}
\mathrm{error}((Y^{t,x}, Z^{t,x}),(Y^{n,t,x}, Z^{n,t,x})) \leq \,\, & \mathrm{error}((Y^{t,x}, Z^{t,x}),( Y^{t,x}_{{\sf f}_n},  Z^{t,x}_{{\sf f}_n}))\notag\\
+& \mathrm{error}(( Y^{t,x}_{{\sf f}_n},  Z^{t,x}_{{\sf f}_n}),( Y^{n,\underline{t},x}_{{\sf f}_n}, Z^{n,\underline{t},x}_{{\sf f}_n})) \\
+& \mathrm{error}((Y^{n,\underline t,x}_{{\sf f}_n}, Z^{n,\underline t,x}_{{\sf f}_n}),(Y^{n, \underline t,x}, Z^{n, \underline t,x})).\notag
\end{align}

\paragraph{New techniques and obstacles caused by the mean field setting}\ \smallskip

\begin{enumerate}[(a)]
\item {\it Comparison of different generators:}
      Freezing the distributions in the generators of \eqref{eq:mainBSDE} and \eqref{eq:mainBSDEn-hat} yields to two generators in the {\it non mean field
      setting}, where one has to compare the corresponding solutions. This is in contrast to  \cite{BGGL21}, where only one generator was present.
      This comparison is done in \cref{statement:disrcrete-rate} and yields to the additional log factor
      $|\log(n+1)|^\beta$, even for the $Y$-process.

\item {\it Singularity in time of the generator:} When freezing the distributions in the generators of \eqref{eq:mainBSDE} and \eqref{eq:mainBSDEn-hat}
      one needs to handle a singularity in time of order
      $|T-t\vee t'|^{-\frac{1}{2}} \left|t-t'\right|^{\alpha{ \wedge \frac{\ep}{2}}}$, see \cref{eq:prop_f_new}
      of \cref{A1_alt} and also \cref{example:singularity:Z}.
      As we consider a discrete time approximation, this singularity needs to be understood.
      In order to estimate the corresponding
      $\mathrm{error}(( Y^{t,x}_{{\sf f}_n},  Z^{t,x}_{{\sf f}_n}),( Y^{n,\underline{t},x}_{{\sf f}_n}, Z^{n,\underline{t},x}_{{\sf f}_n}))$
      we cannot use the estimate \eqref{eq:was_Z} because the singularity in time (if squared)  is not integrable.
      We will therefore prove a new version of   \cite[Theorem 12]{BGGL21} (especially a new version  of \eqref{eq:was_Z})
      in \cref{rate-singularity} where the convergence rate and the singularity are connected. In addition to this,  we are in the
      mean field setting and have to handle the approximation of the mean field terms in parallel to the classical terms.

\item {\it Tail estimates of exponential type:} The coupling of Koml\'os, Major and Tusn\'ady between the Brownian motion and the random walk
      gives an exponential tail estimate. To deduce also tail estimates of exponential type for the final result
      we use two techniques: Firstly, the John-Nirenberg theorem in the continuous and the discrete setting,  see \cref{prop_tilde_difference}
      and \cref{statement:disrcrete-rate}. Secondly, we encode the   tail estimates of exponential type
      into $L^p$-estimates by a control of the behaviour of the constants in dependence on $p$
      (see \eqref{eqn:statement:LPhi_vs_Lp})
      and trace back these behaviour of the constants throughout the proofs.
      
\item {\it Modified H\"older continuity:} We introduce 
      the function  $[0,\infty[\ni r\mapsto \ps r \in [0,\infty[$ which behaves for large $r$ like the  function $r^\ep$ but has a modified behaviour  around zero 
      with the  aim  to get  the same  polynomial factor $n^{-\alpha \wedge \frac{\ep}{2}}$ as in  the pointwise estimates  \eqref{eq:was_Y}, \eqref{eq:was_Z} also   in \cref{statement:main_result_local_version} and
      \cref{coro_final_pathwise_estimate}  (up to a logarithm).
      We need that $[0,\infty[\ni r\mapsto \Ps{r}{\frac{1}{2}}^2 \in [0,\infty)$ is equivalent to a concave function up to a multiplicative constant
      and that, for example, 
      \[ \int_0^T \frac{\od s}{\Ps{s}{\frac{1}{2}}^2}<\infty.\]
      The latter allows for handling the $Z$-processes of the BSDEs.
      A typical application of $\ps r$ one finds in \cref{lemma36} in which we use the replacement of $\ep$-H\"older continuity,
      \begin{align*} 
      |\varphi(x)-\varphi(x')| \le  \Ps {x-x'} \ep  
      \sptext{1}{for}{1} x,x'\in \R.
      \end{align*}

\end{enumerate}

\paragraph{Organization of the article}\ \smallskip

Section \ref{section_notation} introduces the notation and assumptions. 
In \cref{sec:mphi-spaces} we recall facts about function spaces describing exponential tail behaviour.
Section \ref{sect:rw_approx} states results on the random walk approximation of the Brownian motion. 
Regularity properties of $u$ and $\nabla u$ (where $u$ represents the solution of the PDE associated to the standard BSDE)
are proven in \cref{sect:regularity_u_grad_u} under \cref{A1_alt} where we allow a singularity of the 
H\"older regularity in time of the generator. The incorporation of this singularity is essential for the proof
of \cref{statement:main_result_local_version} and \cref{coro_final_pathwise_estimate}.
After recalling the discrete scheme in \cref{sec:discrete_scheme}, we examine the three 
approximation steps from the continuous problem to the discrete problem  in  \cref{sect:conv_rate}.
\cref{subsection_pathwise}  contains the proof of \cref{statement:main_result_local_version}, \cref{coro_final_pathwise_estimate} and  \cref{statement:meanfield_PDE}.
\newpage
\paragraph{Open problems}\ \smallskip

So far we did not succeed to treat general mean field FBSDEs  with an  SDE as forward process since  a  counterpart to the upper
estimate in  \cref{th_random_walk_uniform_convergence}  for SDEs  instead of the Brownian motion is missing. The only upper
estimate which is known to us  (\cite[Proposition 3.1]{GLL20b}) has rate $\frac{1}{4}$. Also the question of treating higher dimensions
is an open problem.


\section{Some notation and assumptions}
\label{section_notation}
We list and discuss here  the assumptions and  some needed notation.
\begin{enumerate}[(1)]
\item $(\Omega,\cF,\P)$ is a complete probability space.
\item $\xi:\Omega \to \R$ is a random variable.
\item We fix $t_o\in [0,T[$ and let $B=(B_t)_{t\in [t_o,T]}$ be a standard Brownian motion,
      where all paths are continuous and $B_{t_o}\equiv 0$, which is
      independent from $\xi$.
\end{enumerate}
If we define  
$\cF_t:= \sigma (\xi, B_s: s\in [t_o,t])\vee \cN$ with 
$\cN:= \{ A\in \cF: \P(A)=0\}$, then the stochastic basis 
$(\Omega,\cF,\P,(\cF_t)_{t\in [t_o,T]})$ satisfies the usual assumptions
by Jacod's absolute continuity condition, see \cite[Lemma 4.20]{Aksamit:Jeanblank:2017}.

\begin{hyp}\label{H:A1}\label{H:A2}
	\renewcommand{\labelenumi}{(\roman{enumi})}
For $0<\ep\leq 1$, $0<\alpha \leq 1$, and $M>0$ the following is satisfied:
	\begin{enumerate}[{\rm (i)}]
    \item The function $g:\rset\to\rset$ is $\ep$-H\"older continuous: for all
      $(x,x')\in\rset^2$ one has
		\begin{equation*}
			 \left|g(x)-g\left(x'\right)\right|\leq \gho \, \left|x-x'\right|^{\ep}, \quad \text{where} \quad \gho := \sup_{x\neq y} \frac{|g(x)-g(y)|}{|x-y|^{\ep}} <\infty. 
		\end{equation*}
      \item The function $f: [t_o,T[\times\rset\times\rset\times\rset  \times \cP_{2}(\R) \times \cP_2(\R) \to \rset  $
        is  locally $\alpha$-H\"older continuous in time, $\ep$-H\"older continuous in
        space and Lipschitz continuous with respect to $(y,z,\mu ,\nu)$: for all
        $(t,x,y,z, \mu ,\nu)$ and $(t',x',y',z',\mu' ,\nu')$ in
        $[t_o,T[\times\rset\times\rset\times\rset \times \cP_{2}(\R) \times \cP_2(\R)$ one has
        \begin{align} \label{eq:prop_f}
		&	 \hspace*{-5em}\left|f(t,x,y,z,\mu ,\nu)-f\left(t',x',y',z',\mu' ,\nu'\right)\right| \nonumber \\
		 \hspace*{-5em} &   \leq    \,   \fhot  \,\frac{ \left| t-t'\right|^{\alpha}}{(T-t\vee t')^\frac{1}{2}} + \fhox \, \left|x-x'\right|^\ep  \nonumber \\ 
       & \quad+ L_f \Big (  \left|y-y'\right|+ \left|z-z'\right|+ \cW_{2}(\mu,\mu') +  \cW_2(\nu,\nu') \Big ).
		    \end{align}
       \item One has $\int_{t_o}^T|f(t,0,0,0,\delta_0,\delta_0)|^2 d t <\infty$.
       \item $\xi \in L^{2\ep}(\cF_{t_o})$ with $\| \xi\|_{L^{2\ep}}\le M$.
		\end{enumerate}
\end{hyp}
\smallskip

Under  \cref{H:A1} there exists a unique solution $(Y^{t_o,\xi},Z^{t_o,\xi})$ to
\begin{align} \label{eq:mainBSDE-new}
	Y^{t_o,\xi}_s &= g\left(B^{t_o, \xi}_T\right) + \int_s^T  f\left(r,B^{t_o,\xi}_r,Y^{t_o,\xi}_r, Z^{t_o,\xi}_r, [Y^{t_o,\xi}_r], [Z^{t_o,\xi}_r] \right) dr - \int_s^T Z^{t_o,\xi}_r\, dB_r,  \notag\\
	&  \hspace{20em} \quad t_o\leq s\leq T,
\end{align}
see \cite[Theorem A.1]{Li17}. 
In \cref{loc-Holder-cont} we consider generators with frozen measure components, that is, generators of type
${\sf f}: [t_o, T[ \times \R \times \R \times \R \to \R$,
and show that  \cref{H:A1} implies 
\cref{A1_alt} below:

\begin{althyp} \label{A1_alt}
\renewcommand{\labelenumi}{(\roman{enumi})}
 There exist $0<\ep\leq 1$  and  $0<\alpha \leq 1$ such that   \cref{H:A1}{\rm((i),(iv))} hold and
 the generator ${\sf f}:[t_o,T[\times\rset\times\rset\times\rset   \to \rset  $ satisfies
 	\begin{enumerate}
	\item[{\rm (ii')}]   
	For all  $(t,x,y,z), (t',x',y',z') \in  [t_o,T[\times\rset\times\rset\times\rset $

\begin{multline} \label{eq:prop_f_new}
     \left|{\sf f}(t,x,y,z)- {\sf f}(t',x',y',z')\right|
\leq |{\sf f}|^{\scriptscriptstyle loc}_{\alpha \wedge \frac{\ep}{2};t} \,\frac{ \left|t-t'\right|^{\alpha{ \wedge \frac{\ep}{2}}}}{(T-t\vee t')^\frac{1}{2}}
      + \fhoxt  \left|x-x'\right|^\ep \\
      + L_{{\sf f}}\Big (  \left|y-y'\right|+ \left|z-z'\right|\Big ).
\end{multline}
\item[{\rm (iii')}] One has $\int_{t_o}^T|\f(t,0,0,0)|^2 d t <\infty$.
\end{enumerate}
\end{althyp}
\smallskip

\begin{ex}
\label{example:singularity:Z}
A singularity in $T$  like  in \eqref{eq:prop_f_new}  can originate from the frozen measure component $[Z_t]$  if $g$ is not Lipschitz continuous, but only H\"older continuous: Let us assume  that $g(x)= \max \{x,0\}^\varepsilon$  
for some $\varepsilon \in ]0,\tfrac{1}{2}[$, $T:=1$, and consider the  BSDE
\[ Y_s = g\left(B_1\right) + \int_s^1  \|Z_r\|_{L^2} dr - \int_s^1 Z_r\, dB_r,
	   \quad 0\leq s\leq 1, \]
i.e. we have the generator $f(t,x,y,z,[Y_{t}],[Z_{t}] )= f([Z_t]):= \| Z_t\|_{L^2}$, so that
\[ | f(t',x',y',z',[Y_{t'}],[Z_{t'}] )  -f(t,x,y,z,[Y_{t}],[Z_{t}] )|
    = |\|Z_{t'}\|_{L^2} - \|Z_t\|_{L^2}|
   \le \cW_2([Z_{t'}], [Z_t]). \]
For $v(t,x) := \E  g ( x+ B_{1-t})$ it holds $  Z_t =   \frac{\partial}{\partial x} v(t, B_t),$
where $Z$ is a martingale and
\begin{equation}\label{eqn:integtal_Z_curvature}
  Z_{t'} =   Z_t +  \int_t^{t'}  \frac{\partial^2}{\partial x^2} v(s, B_s) dB_s \mbox{ a.s.}
  \sptext{1}{for}{1}  0\le t<t'<1.
\end{equation}
In \cite[Remark 2(i)]{GGG:2012} it was shown   that
$\lim_{t \to 1}  (1-t)^{\frac{3}{2} - \varepsilon}  \E  \left  | \frac{\partial^2}{\partial x^2} v(t, B_t)   \right |^2 \in \, ]0,\infty[$.
Therefore there are $c_0\ge 1$ and $T_0\in [0,1[$ such that
\begin{equation}\label{eqn:singularity:Z}
\frac{1}{c_0} (1-t)^{\ep - \frac{3}{2}}  \le \E  \left  | \frac{\partial^2}{\partial x^2} v(t, B_t)   \right |^2 \le c_0 (1-t)^{\ep - \frac{3}{2}}
    \sptext{1}{for}{1}
    t\in [T_0,1[.
\end{equation}
Combining \eqref{eqn:integtal_Z_curvature} and \eqref{eqn:singularity:Z} we find 
$c_1\ge 1$ and $T_1\in [T_0,1[$ such that
\[  \frac{1}{c_1} (1-t)^{\ep - \frac{1}{2}}  \le \E  | Z_t |^2 \le c_1 (1-t)^{\ep - \frac{1}{2}}
    \sptext{1}{for}{1}
    t\in [T_1,1[.\]
Let $A>1$ and define the time net $t_k:= 1 - \frac{1}{A^k}$ so that $0=t_0<t_1<t_2 < \cdots  \uparrow 1$.
We choose $A>1$ such that $\frac{1}{c_1} A^{(\frac{1}{4}-\frac{\ep}{2})} -  c_1=2$.
Then, for $T_1 < t_k < t_{k+1} < 1$ one has
\begin{multline*}
    |f(t_{k+1} ,x,y,z,[Y_{t_{k+1}}],[Z_{t_{k+1}}]) - f(t_k,x',y',z',[Y_{t_k}],[Z_{t_k}])|
  = \Big | \| Z_{t_{k+1}} \|_{L^2} - \| Z_{t_k}\|_{L^2} \Big | \\
\ge A^{k(\frac{1}{4}-\frac{\ep}{2})} \left [ \frac{1}{c_1} A^{(\frac{1}{4}-\frac{\ep}{2})} -  c_1 \right ]
 = 2  A^{k(\frac{1}{4}-\frac{\ep}{2})}
 = \left [ \frac{2}{ \sqrt{A} \left ( 1 - \frac{1}{A} \right ) ^{\frac{1}{4} + \frac{\ep}{2}}}\right ] \frac{|t_{k+1}-t_k|^{\frac{1}{4} + \frac{\ep}{2}}}{|1-t_{k+1}|^\frac{1}{2}}.
\end{multline*}
\end{ex}

\smallskip

\begin{remark}
Although {\it global} H\"older continuity in time for the generator was assumed in  \cite{BGGL21},  the results can still be proved
under local  H\"older continuity given in \eqref{eq:prop_f_new}. We will discuss the necessary changes in \cref{sec:middle_part}.
\end{remark} 
\smallskip

We collect the notation we use in the sequel:
\begin{align*}
       T > 0 & : \quad \mbox{terminal time} \\
t_o\in [0,T[ & : \quad \mbox{time at that the forward process starts} \\
\ep\in ]0,1] & : \quad \mbox{H\"older continuity in the space variable } x \\
M>0  & : \quad \|\xi\|_{L^{2\ep}} \le M \\
\beta \in \{0\} \cup \left ]\frac{1}{2},\infty \right [ & : \quad \mbox{exponent for logarithmic weight } \\
\gho  &:= \sup_{x\neq y}                  \left [\frac{|g(x)-g(y)|}{|x-y|^{\ep}} \right ] \\
\fhot &:= \sup_{t\neq t', x,y,z,\mu ,\nu} \left [ \frac{|f(t,x,y,z,\mu ,\nu)-f\left(t',x,y,z,\mu ,\nu\right)| }{| t-t'|^{\alpha}}{(T-t\vee t')^\frac{1}{2}} \right ]  \\
\fhox &:= \sup_{x\neq x',t, y,z,\mu ,\nu} \left [ \frac{|f(t,x,y,z,\mu ,\nu)-f\left(t,x',y,z,\mu ,\nu\right)| }{| x-x'|^{\ep}} \right ] \\
L_f \sptext{.5}{and}{0.5} L_{\sf f}   & : \quad \mbox{Lipschitz constants related to the generators } f \sptext{.5}{and}{.5} {\sf f} \\
c^{(2)}_{f_0} & := \| f(\cdot, 0, 0, 0,\delta_0,\delta_0) \|_{L^2([t_o, T])}\\
c_{{\sf f}_0}^{(2)} & := \| {\sf f}(\cdot, 0, 0, 0) \|_{L^2([t_o, T])}\\
f_0           & := f(0, 0, 0, 0, \delta_0, \delta_0) \\
n_0           & = n_0(T,L_f,t_o)  >   T\max\left \{72 L_f^2+1, \frac{2}{T-t_o} \right \} \\
h   &:= \frac{T}{n} \mbox{ for } n\ge n_o \sptext{1}{step-size}{0} \\
t_k &:= kh \\
\ut & := \max \left\{ t_k \mid k \in \left\{ 0, ..., n \right\}  \textrm{ and } t \geq t_k  \right\} \\
\ot &:= \inf \left\{ t_k \mid k \in \left\{ 0, ..., n \right\} \textrm{ and } t \leq t_k  \right\}\\
t_j &:= \uto.
\end{align*} 
To shorten the notation we also use
\[ \Theta:= (T,\ep,\alpha,\beta,\gho,g(0),\fhot,\fhox,L_f,c_{f_0}^{(2)}).\]

\begin{remark}
We will use $n\ge n_0(T,L_f,t_o)$ in the sequel. In particular, this ensures
$n_0 > T L_f$, so that we have a unique solution to \eqref{eq:mainBSDEn-hat} and \eqref{eq:mainBSDEn} (see \cite{MR1817885}).
Our assumption on $n_0$ also guarantees the requirements for \cref{statement:a-priori-discrete} (that generalizes \cite[Lemma 14]{BGGL21})
and that $t_{n-2}> t_o$.
\end{remark} 


\section{The function spaces \texorpdfstring{$M_\varphi^0$}{Mphi}}
\label{sec:mphi-spaces}

For the convenience of the reader we provide a list of inequalities and facts from function space theory which will be used later on.
We assume an increasing bijection $\psi:[1,\infty[ \to [1,\infty[$
and define for a random variable $f:\Omega \to \R$ on some probability space $(\Omega,\cF,\P)$
the quantity
\[ |f|_{M_\varphi^0}
   := \inf \left \{ c > 0 : \P(|f| > \lambda) \le e^{1-\psi(\lambda/c)} \mbox{ for }
       \lambda \ge c \right \}
   \sptext{1}{for}{1}
   \varphi(t):= \frac{1}{\psi^{-1}(1+\log(1/t))}.\]       
As particular functions we consider
\[ \psi_\gamma (x) := x^\gamma \sptext{1}{and}{1}
   \varphi_\gamma(t):= \frac{1}{\sqrt[\gamma]{(1+\log(1/t))}}
   \sptext{1}{for}{1} \gamma \in ]0,\infty[.\]
In particular, for the spaces used in \cref{statement:main_result_local_version} and \cref{coro_final_pathwise_estimate}
we have $\|\cdot\|_{L^{\exp(\gamma)}}=|\cdot|_{M_{\varphi_{\frac{1}{\gamma}}}^0}$.
\smallskip

We will use the following relations:
\smallskip

\begin{enumerate}[(1)]
\item \label{item:1:Mphi}
      {\cite[Lemma 5.2]{Geiss97}}: If there is an $a>1$ such that
      $\inf_{\lambda \ge 1}\frac{\psi(a\lambda)}{\psi(\lambda)}>1$,
      then one has
      \begin{equation}\label{eqn:statement:Mvarpi_supremum}
      \left | \sup_{1\le k \le n} \frac{|f_k|}{\psi^{-1}(1+\log k)}\right |_{M_\varphi^0}
      \le c_\eqref{eqn:statement:Mvarpi_supremum}   \sup_{1\le k \le n} |f_k|_{M_\varphi^0}
      \end{equation}
       for all random variables $f_1,\ldots,f_n:\Omega \to \R$,
       where $c_\eqref{eqn:statement:Mvarpi_supremum}  =c_\eqref{eqn:statement:Mvarpi_supremum}(\varphi)>0$.

\item \label{item:2:Mphi}
      For $\gamma\in ]0,\infty[$ there is a constant
      $c_\eqref{eqn:statement:LPhi_vs_Lp}=c_\eqref{eqn:statement:LPhi_vs_Lp}(\gamma) \ge 1$
      such that
      \begin{equation}\label{eqn:statement:LPhi_vs_Lp}
      \frac{1}{c_\eqref{eqn:statement:LPhi_vs_Lp}} |f|_{M^0_{\varphi_\gamma}}
      \le \sup_{p\in [1,\infty[} \frac{\| f\|_{L^p}}{\sqrt[\gamma]{p}}
      \le c_\eqref{eqn:statement:LPhi_vs_Lp} |f|_{M^0_{\varphi_\gamma}},
      \end{equation}
      for all random variables $f:\Omega \to \R$,
      see for example \cite[Proposition 19.2.7]{GG:book:MPFA}.

\item Combining items \eqref{item:1:Mphi} and \eqref{item:2:Mphi} yields:
      If $\gamma \in ]0,\infty[$ and if $f_1,\ldots,f_n:\Omega \to \R$ are random variables with
      $\|f_k \|_{L^q} \le \sqrt[\gamma]{q}$ for $q\in [1,\infty[$ and $k=1,\ldots,n$, then for $p\in [1,\infty[$ one has
      \begin{equation}\label{eqn:2:statement:moments_of_supremum_of_iid}
      \left \| \sup_{1\le k\le n} |f_k| \right \|_{L^p} \le c_\eqref{eqn:2:statement:moments_of_supremum_of_iid} \sqrt[\gamma]{p\,\log(n+1)}
      \end{equation}
      for some $c_\eqref{eqn:2:statement:moments_of_supremum_of_iid}=c_\eqref{eqn:2:statement:moments_of_supremum_of_iid}(\gamma)>0$.

\item  Hoeffding's inequality \cite{Hoeffding:63}:
       If $(\zeta_k)_{k\geq 1}$ is an i.i.d. sequence of Rademacher random variables, then
       \begin{equation}\label{eqn:hoeffding}
       c_\eqref{eqn:hoeffding} :=\sup_{n\in \N} \left | \frac{\zeta_1+\cdots +\zeta_n}{\sqrt{n}} \right |_{M^0_{\varphi_2}} < \infty,
       \end{equation}
\item  If $(g_k)_{k\ge 1}$ are i.i.d. standard Gaussian random variables, then
       \begin{equation}\label{eqn:1:statement:moments_of_supremum_of_iid}
 \frac{1}{ c_{\eqref{eqn:1:statement:moments_of_supremum_of_iid}}} \sqrt{\log(n+1)}     \le    \left \| \sup_{1\le k\le n} |g_k| \right \|_{L^1}  \le   c_{\eqref{eqn:1:statement:moments_of_supremum_of_iid}} \sqrt{\log(n+1)},
       \end{equation}
       where $c_\eqref{eqn:1:statement:moments_of_supremum_of_iid} \ge  1$ is an absolute constant. This relation
       is folklore, see for example \cite[Theorem 19.1.12, Example 19.1.15]{GG:book:MPFA}.

\end{enumerate}


\section{Random walk approximation of the Brownian motion}\label{sect:rw_approx}

The following \cref{th_random_walk_uniform_convergence} gives the pathwise convergence rate for the random walk approximation of the Brownian motion,
where the lower estimate given in item \eqref{item:1:th_random_walk_uniform_convergence} should be folklore.
Item \eqref{item:2:th_random_walk_uniform_convergence} is based on a result of
J.~Koml\'os, P.~Major and G.~Tusn\'ady \cite[Theorem 1]{KMT75} which states (for the special case  of  Rademacher  
random variables  we consider here)  that if  $(g_i)_{i=1}^{\infty}$ are   i.i.d. standard Gaussian random variables, 
then there is a sequence of Borel functions $f_n:\R^{2n}\to \R$, $n\in \N$, such that
\[ (r_n)_{n\in \N} := (f_n(g_1, g_2,..., g_{2n}))_{n\in \N} \]
are independent Rademacher variables, and it holds  for  $S_k := \sum_{m=1}^k r_m$ and $T_k := \sum_{m=1}^k g_m,$  that
\begin{equation} \label{exp-estimate}
  \p\left (\sup_{1\le k\le n}  |S_k -T_k|> C_\eqref{exp-estimate} \log n+ x\right ) \le K_\eqref{exp-estimate} e^{-\lambda_\eqref{exp-estimate} x}
  \sptext{1}{for}{1} x >0 \sptext{.5}{and}{.5} n\ge 1,
\end{equation}
where    $C_\eqref{exp-estimate}, K_\eqref{exp-estimate},\lambda_\eqref{exp-estimate}>0$  are constants.
\medskip

\begin{thm}
\label{th_random_walk_uniform_convergence}
\begin{enumerate}[{\rm (1)}]
\item \label{item:1:th_random_walk_uniform_convergence}
      For $n\in \N$ and $t\in [0,T]$ assume random variables $A_t^n:\Omega\to \R$, $t\in [0,T]$, such that $t\mapsto A_t^n(\omega)$ is
      constant on all $\left [{(k-1)\frac{T}{n}},{k\frac{T}{n}}\right [$, and a standard Brownian motion $W_t:\Omega\to \R$
      with $W_0\equiv 0$ and all paths being continuous. Then
      one has
      \[ \sup_{t\in [0,T]} |W_t-A_t^n| \ge \frac{1}{2} \sup_{k=1,\ldots,n}\left |W_{k\frac{T}{n}}-W_{(k-1)\frac{T}{n}}\right | \]
      so that
      \[  \frac{ \sqrt{T}}{2 c_\eqref{eqn:1:statement:moments_of_supremum_of_iid}}
          \sqrt{ \frac{\log (n+1)}{n}} \le \left \| \sup_{t\in [0,T]} |W_t-A_t^n| \right \|_{L_1}.\]

\item \label{item:2:th_random_walk_uniform_convergence}
      There is a $c_\eqref{eqn:th_random_walk_uniform_convergence}=c_\eqref{eqn:th_random_walk_uniform_convergence}(T)>0$ such that for all
      $n \in \N$ there is a coupling of the Brownian motion from item \eqref{item:1:th_random_walk_uniform_convergence}
      with a random walk $(W_t^n)_{t\in [0,T]}$ with $W_t^n = W_{\ut}^n = \sqrt{\frac{T}{n}} \sum_{k=1}^{n\ut/T} r_k$, where
      $r_1,\ldots,r_n:\Omega \to \{-1,1\}$ are independent Rademacher variables, such that
      \begin{equation}\label{eqn:th_random_walk_uniform_convergence}
          \sup_{p\in [1,\infty[} \frac{1}{p}  \left \| \sup_{t \in [0, T]} \left| W_t - W_t^n \right | \right \|_{L^p}
      \le c_\eqref{eqn:th_random_walk_uniform_convergence} \, \frac{\log (n+1)}{\sqrt{n}}.
      \end{equation}
\end{enumerate}
\end{thm}
\begin{proof}
\eqref{item:1:th_random_walk_uniform_convergence}
We use the relation \eqref{eqn:1:statement:moments_of_supremum_of_iid} for the last term in
\[ \sup_{t\in [0,T]} |W_t-A_t^n|
   \ge \sup_{k=1,\ldots,n} \sup_{t\in \left [(k-1)\frac{T}{n},k\frac{T}{n}\right [} \left |W_t-A_{(k-1)\frac{T}{n}}\right |
   \ge \frac{1}{2} \sup_{k=1,\ldots,n}\left |W_{k\frac{T}{n}}-W_{(k-1)\frac{T}{n}}\right |. \]
\eqref{item:2:th_random_walk_uniform_convergence}
Let $(\Omega, \m{F},\p)$ be a probability space carrying a Brownian motion $(W_t)_{t \in [0,2T]}$.
We choose $n \in \nset$ and put  $t_k := k \tfrac{T}{n}$ for $k=0,1,...,2n.$ This provides us with the $2n$ i.i.d.~Gaussian random variables
$(W_{t_k}-W_{t_{k-1}})_{k=1}^{2n}$
from which  the $r_1,...,r_n$ can be  constructed following  \cite[Theorem 1]{KMT75}.
Then for $(W^n_t)_{t \in [0,T]}$ as defined in assertion \eqref{item:2:th_random_walk_uniform_convergence}
we have
\begin{align*}
    \sup_{t \in [0,T]} |W_t -W^n_t|^p
& = \left ( \max_{1\le k\le n}  \sup_{t \in [t_{k-1},t_k[}  |W_t -W^n_{t_{k-1}}|^p \right ) \vee |W_{t_n} -W^n_{t_n}|^p \\
&\hspace*{-5em} \le \left [ 2^{p-1} \left  (\max_{1\le k\le n}  \sup_{t \in [t_{k-1},t_k[} (|W_t -W_{t_{k-1}}|^p + |W_{t_{k-1}} -W^n_{t_{k-1}}|^p) \right ) \right ]
     \vee |W_{t_n} -W^n_{t_n}|^p  \\
&\hspace*{-5em} \le 2^{p-1} \left ( \max_{1\le k\le n}  \sup_{t \in [t_{k-1},t_k[} |W_t -W_{t_{k-1}}|^p +   \max_{1\le k\le n}   |W_{t_k} -W^n_{t_k}|^p \right ).
 \end{align*}
\underline{First term:} Here we prove that
\begin{equation}\label{eqn:term:1:rw_approximation_BM}
    \left (\e \max_{1\le k\le n}  \sup_{t \in [t_{k-1},t_k[} |W_t -W_{t_{k-1}}|^p \right)^{1/p}
\le c_\eqref{eqn:term:1:rw_approximation_BM} \, p\, \sqrt{\frac{\log (n+1)}{n}}.
\end{equation}
To verify this, we define
\[ f_k :=  \sup_{t \in [t_{k-1},t_k[} |W_t -W_{t_{k-1}}|.\]
The sequence is an i.i.d. sequence of random variables. By definition,
\[   \left (\e \max_{1\le k\le n}  \sup_{t \in [t_{k-1},t_k[} |W_t -W_{t_{k-1}}|^p \right)^{1/p}
   = \left \| \max_{1\le k\le n} |f_k| \right \|_{L^p}.\]
By Doob's maximal inequality we get for
$q\in [2,\infty[$ that
\[ \| f_1 \|_{L^q} \le \frac{q}{q-1} \| W_{t_1} \|_{L^q}
                    = \frac{q}{q-1} \sqrt{\frac{T}{n}} \|  W_1 \|_{L^q}
                    \le  \frac{q}{q-1} \sqrt{\frac{T}{n}} \kappa \sqrt{q}
                    \le (2 \sqrt{T} \kappa) \,\,  \sqrt{\frac{q}{n}},\]
where we used that  $\|  W_1 \|_{L^q} \le \kappa  \sqrt{q} $ for an  absolute constant   $\kappa>0$.
For $q\in [1,2[$ we derive
\[      \|\sqrt{n} f_1 \|_{L^q} \le     \|\sqrt{n} f_1 \|_{L^2}  \le    (2 \sqrt{T} \kappa) \,\,  \sqrt{2} \le    (2 \sqrt{2T} \kappa) \,\,  \sqrt{q}.\]
Using \eqref{eqn:2:statement:moments_of_supremum_of_iid} implies

\[      \left \| \max_{1\le k\le n} \sqrt{n} |f_k| \right \|_{L^p}
    \le \Big (c_\eqref{eqn:2:statement:moments_of_supremum_of_iid} 2 \sqrt{2T} \kappa \Big)\, \, \sqrt{p} \, \sqrt{\log(n+1)}.\]
\smallskip
\underline{Second term:}
By  the inequality \eqref{exp-estimate}, for $x \ge1$ and $n\ge 2$, we have for $S_k := \sum_{m=1}^k r_m$ and
$T_k := \sum_{m=1}^k \frac{W_{t_m} -W_{t_{m-1}}}{\sqrt{h}}$
\begin{align*}
     \p\left (\sup_{1\le k\le n}  |S_k -T_k| > (C_\eqref{exp-estimate}+1) x \log (n+1) \right )
&\le \p\left (\sup_{1\le k\le n}  |S_k -T_k| > C_\eqref{exp-estimate} \log n + x    \right ) \\
&\le K_\eqref{exp-estimate} e^{-\lambda_\eqref{exp-estimate} x}.
\end{align*}
This is the same as
\[      \p\left (\sup_{1\le k\le n} \left [ \sqrt{\frac{n}{T}}|W_{t_k} - W^n_{t_k}| \right ] > (C_\eqref{exp-estimate} +1) x \log(n+1) \right )
   \le  K_\eqref{exp-estimate} e^{-\lambda_\eqref{exp-estimate}  x}. \]
And  this implies that
\begin{equation}\label{eqn:proof:th_random_walk_uniform_convergence}
 \left \| \sup_{1\le k\le n} \left [\sqrt{\frac{n}{T}} |W_{t_k} - W^n_{t_k}| \right] \right \|_{L^p}
    \le c_\eqref{eqn:proof:th_random_walk_uniform_convergence} \, p \, \log(n+1)
\end{equation}
with
$  c_\eqref{eqn:proof:th_random_walk_uniform_convergence}
 = c_\eqref{eqn:proof:th_random_walk_uniform_convergence}  (C_\eqref{exp-estimate}, K_\eqref{exp-estimate},\lambda_\eqref{exp-estimate})>0$.
\end{proof}


\section{Regularity properties of \texorpdfstring{$u$}{u} and \texorpdfstring{$\nabla u$}{nabla-u} under  \texorpdfstring{\cref{A1_alt}}{A1} }\label{sect:regularity_u_grad_u}

In this section we  generalize some regularity results obtained in \cite{BGGL21}. 
If ${\sf f}:[t_o,T[\times\rset\times\rset\times\rset  \to \rset  $ is a  generator satisfying
\cref{A1_alt}, then we define
\[ u(t,x):= Y^{t,x}_t. \]
Later we will use that, a.s., 
\begin{align} \label{YandZfrom-u}
 Y^{t,x}_s =u(s, B^{t,x}_s)
 \sptext{.75}{on}{.75} [t_o,T]
 \sptext{1}{and}{1}
 Z_s^{t,x}=\nabla u(s, B^{t,x}_s)
 \sptext{.75}{on}{.75} [t_o,T[.
\end{align}
The first relation follows from the  regularity of $u$ in \cref{en:regu},
\cite[Theorem 4.1]{ElKPQ:97}.  Since in \cite[Theorem 4.1]{ElKPQ:97} 
 the generator  is assumed to be bounded in $t$ one has to  truncate the generator and use    a standard a-priori argument.
The second relation is treated in  \cref{lemma37}\eqref{ma_zhang_formula} below.
\medskip

\begin{lemma}[{\cite[Lemma 6]{BGGL21}}]
\label{en:regu}

Under \cref{A1_alt} there exists a constant
$c_{\eqref{u-estimates}}=c_{\eqref{u-estimates}}(T,\ep,\gho,g(0),\fhoxt,L_{\sf f},c^{(2)}_{{\sf f}_0})>0$
such that for $u(t,x):= Y^{t,x}_t$ and for all $(t,x)\in[t_o,T]\times\rset$,
\begin{equation} \label{u-estimates}
		|u(t,x)| \leq c_{\eqref{u-estimates}}\, (1+|x|)^\ep,  \qquad | u(t,\cdot)|_{\ep;x} \leq
        c_{\eqref{u-estimates}},  \qquad  |u(\cdot,x)|_{\ep/2;t} \leq c_{\eqref{u-estimates}}\, (1+|x|)^\ep.
\end{equation}
\end{lemma}
\bigskip
The proof follows the approach from \cite[Lemma 6]{BGGL21} and an a-priori estimate in the form of
\cite[Lemma 5.26]{GeissYlinen}, so that \cref{en:regu} holds true under \cref{A1_alt} as well.
\medskip

Next, we introduce  a method which permits to prove new bounds for the gradient $\nabla u(t,x)$ because the known bounds   (for example in  \cite[Lemma 7]{BGGL21})
have a singularity $(T-t)^{-1/2}.$  However, we need square integrability in time,  for example in \eqref{eqn:2:statement:main_result_local_version}.
It turns out that it is possible to make the changes so subtle that in \cref{statement:main_result_local_version} and \cref{coro_final_pathwise_estimate}
the rate $1/n^{\alpha \wedge \frac{\ep}{2}}$ (known from pointwise estimates like in   \cite[Theorem 12]{BGGL21}) is only changed  by the  logarithmic 
factors $|\log(n+1)|^{\beta}$ and $|\log(n+1)|^{\ep + \beta}$, respectively.
\medskip

For  $\beta  \in \{0\} \cup \left ]\frac{1}{2},\infty\right [$ and  $\ep \in ]0,1]$  we introduce
$\Phi_{\beta}$ and $\ps \cdot : ]0, \infty[ \to [0,\infty[$ by
\begin{align*}
 \Phi_{\beta}(r) & := \begin{cases}
                      \left( \log \Big (\frac{1}{r\wedge r_\beta} \Big ) \right)^{ \beta} &: \beta > \frac{1}{2}
                      \sptext{1}{with}{1} r_\beta := e^{-2\beta}\\
                      1 &: \beta = 0
                      \end{cases}, \\
        \ps r   & :=    r^\ep \Phi_{\beta }(r^{2\ep} ).
\end{align*}
The notation $ \ps r$ was chosen to indicate that it is just a modification of $r^\ep$ for small
$r$ by the help of $\Phi_\beta$.
We continuously extend $ \ps \cdot $ to $ \ps \cdot :\R \to [0,\infty[$ by setting $ \ps 0 :=0 $
and \[ \ps x  := |x|^\ep \Phi_{\beta }(|x|^{2\ep}).  \]

We will use the following properties:
\smallskip

\begin{lemma}
\label{Psi-lemma}
\begin{enumerate}[{\rm (i)}]
\item \label{the-Phi} $\Phi_{\beta }$ is non-increasing and $\Phi_{\beta }(r) \ge 1$ for all $r\in ]0,\infty[$.
\item \label{the-Psi} The map $]0, \infty[ \ni r \mapsto r^a  \Phi_{\beta }(r^2) $ is non-decreasing for all $a \geq 1$.
      Especially, the function $]0, \infty[ \ni r \mapsto    \Ps r {\ep}^2$ is  non-decreasing.
\item \label{item:4:Psi-lemma} 
      There are a continuous concave function $\Psi_{\beta}: [0,\infty[\to [0,\infty[$
      and $\kappa_\beta\ge 1$ 
      such that 
      \[ \frac{1}{\kappa^2_\beta} \Psi_{\beta}(r) \le  \Ps r {\frac{1}{2}}^2 \le  \Psi_{\beta}(r). \]
\end{enumerate}
\end{lemma}

\begin{proof} 
If $\beta = 0$, then all the properties \eqref{the-Phi}--\eqref{item:4:Psi-lemma} are true, so we assume $\beta > \frac{1}{2}$.

\eqref{the-Phi} If $\beta > 1/2$, then  it follows that   $\left( \log \Big (\frac{1}{r} \Big ) \right)^{ \beta} \ge 1$ for all
$r \in  ]0, e^{-2\beta}]$, and obviously $\Phi_{\beta }$ is non-increasing.
\smallskip

\eqref{the-Psi}
We see that  $r \mapsto r^a  \Phi_{\beta }(r^2)$ is  non-decreasing for $a \geq 1$ since  for $r < r_\beta$  we have
\[  \frac{\od}{\od r} \left( r^a  \Phi_{\beta }(r^2) \right)
 =  r^{a-1} \left( \log \Big (\frac{1}{r^2} \Big ) \right)^{ \beta-1}
    \left(  a \log \Big (\frac{1}{r^2} \Big ) -  2\beta   \right)>0.
\]

\eqref{item:4:Psi-lemma} For $r\in ]0,r_\beta[$ we have that
\[   ( \Ps r {\frac{1}{2}}^{2} )' = \frac{d}{dr}   r\Phi_{\beta }^2(r)
   = \left( \log \Big (\frac{1}{r} \Big ) \right)^{2 \beta-1}  \left(  \log \Big (\frac{1}{r} \Big ) -  2\beta   \right) >0 \]
is decreasing so that $ \Ps \cdot {\frac{1}{2}}^2$ is concave on $[0,r_\beta]$.
We define the concave function $\Psi_{\beta}:[0,\infty[\to [0,\infty[$ by
\begin{align*}
\Psi_{\beta}(r) & := \begin{cases}
                       \Ps r {\frac{1}{2}}^2 &:  0\le r\le r^2_\beta \\
                     a( r-r^2_\beta) +  \Ps  {r^2_\beta} {\frac{1}{2}}^2 &: r>r^2_\beta
                      \end{cases}, 
\end{align*} 
where $a $ is chosen to be the derivative of  $r\mapsto  \Ps r {\frac{1}{2}}^2 $ in $r^2_\beta.$  By a  short calculation one proves that for  $\kappa_\beta = 2^{\beta} $ it holds
\[     \frac{1}{\kappa^2_\beta}  \Psi_{\beta}  \le  \Ps \cdot {\frac{1}{2}}^2 
    \le  \Psi_{\beta} . \] 
\end{proof} 

\begin{remark} \label{phi_remark}
By the help of  $\Phi_{\beta }$ we achieve  (and will use) square integrability in time since for $\beta >1/2$ it holds
\begin{equation*}
\int_0^T \frac{1}{s |\Phi_{\beta }(s^\ep)|^2} ds < \infty.
\end{equation*}
\end{remark}
\bigskip

We will show regularity properties of $\nabla u $ by the help of the following lemma.
\medskip

\begin{lemma}\label{lemma36}
Assume for a continuous function  $\varphi:\R \to \R$, $H(u,s):=\frac{B_{s}-B_{u}}{s-u}$ for  $0\le u<s\le T$,
$\ep \in ]0,1]$, and $\gamma \in \{0\} \cup ]\frac{1}{2},\infty[$ such that
\begin{align} \label{theF}
    |\varphi(x)-\varphi(x')|
\le \Psg {x-x'} \ep
    \sptext{1}{for}{1} x,x'\in \R.
\end{align}
Then, for  $0\le r \le t< s \le T$ one has
\begin{equation} \label{eq_new_claim}
     \left | \E [ \varphi(B_s^{r,x})H(r,s)-\varphi(B_s^{t,x})H(t,s) ]\right |
\le \kappa_\gamma
    \left [
        \frac{\Psg  {t-r} {\frac{\ep}{2}}}{|s-r|^\frac{1}{2} } 
      + \frac{\Psg  {s-t} {\frac{\ep}{2}}|t-r|^{\frac{\ep}{2}}} {|s-r|^{\frac{\ep}{2}}\,\, |s-t|^{\frac{1}{2}}} \right ].
\end{equation}
Consequently we have the following:
\begin{enumerate}[{\rm (i)}]
\item \label{first}
      For $\gamma=0$ and $\bar \gamma \in \{0\} \cap ]\frac{1}{2},\infty[$ we get
      \[  \left |  \E [ \varphi (B_s^{r,x})H(r,s)-\varphi(B_s^{t,x})H(t,s) ] \right |
         \le   2 \kappa_\gamma
         \frac{\Psgb  {t-r} {\frac{\ep}{2}}}{ |s-t|^{\frac{1}{2}} \,\Phi_{\bar\gamma}(|s-t|^\ep)}. \]
\item \label{second}
      For $\gamma > \frac{1}{2}$ we get
      \[  \left | \E [\varphi(B_s^{r,x})H(r,s)-\varphi(B_s^{t,x})H(t,s) ]\right |
          \le 2 \kappa_\gamma \frac{\Phi_\gamma (   |s-t|^{\ep} ) \, \Psg {t-r}  {\frac{\ep}{2}} }{ |s-t|^\frac{1}{2}}.\]
\end{enumerate}
\end{lemma}

\begin{proof}
We have  
\begin{align} \label{F-zero-estimate-split}
&   \left |\e \left[\varphi(B_s^{r,x})H(r,s) \right] - \e\left[ \varphi(B_s^{t,x}) H(t,s)\right] \right|  \notag\\
& = \left | \e \left[(\varphi(B_s^{r,x}) -\varphi(B_s^{t,x})) \,\,  H(r,s)\right]  +  \e \left[ (\varphi(B_s^{t,x}) - \varphi(x) )( H(r,s) - H(t,s))\right] \right | \notag\\
&\le\e \left[|\varphi(B_s^{r,x}) -\varphi(B_s^{t,x})| \,\, | H(r,s)|\right]  +  \e \left[ |\varphi(B_s^{t,x}) -\varphi(x)| \,\, | H(r,s) - H(t,s)| \right].
\end{align}
Since  by \cref{Psi-lemma} the function $\Psi_{\gamma}$ is concave, Jensen's inequality permits to move the expectation inside, we have
\begin{align*}
     \left( \e  \Psg {B_t-B_r} {\ep}^2 \right)^\frac{1}{2}
& =  \left( \e  \Psg {\,|B_t-B_r|^{2\ep} \,} {\frac{1}{2}}^2 \right)^\frac{1}{2}
 \le \left( \e \Psi_{\gamma }(|B_t-B_r|^{2\ep})    \right)^\frac{1}{2}  \le  \left(  \Psi_{\gamma }(\e |B_t-B_r|^{2\ep})    \right)^\frac{1}{2} \\
&\le   \Psi^\frac{1}{2}_{\gamma }( |t-r|^{\ep})   \le  \kappa_\gamma  \Psg {t-r}  {\frac{\ep}{2}},   
\end{align*}
where we also used that  $\Psi_{\gamma}$ is increasing and that  $\e |B_t-B_r|^{2\ep} =|t-r|^\ep \e |B_1|^{2\ep} \le |t-r|^\ep$.
For the first term on the RHS of \eqref{F-zero-estimate-split}
we get from \eqref{theF} by applying  the Cauchy-Schwarz  inequality that
\begin{align*} \label{eq_est_first_term}
     \e \left[\left|\varphi(B_s^{r,x}) -\varphi(B_s^{t,x})\right | \frac{|B_s-B_r|}{s-r}  \right]  \nonumber
&\le \e\left(  \Psg {B_t-B_r} {\ep}  \frac{ |B_s-B_r|}{s-r}\right)  \nonumber \\
&\le \kappa_\gamma \,  \Psg {t-r}  {\frac{\ep}{2}} \, \frac{1}{|s-r|^\frac{1}{2}}.
\end{align*}
For the second term on the RHS of \eqref{F-zero-estimate-split} we get  from \eqref{theF} that
\begin{align*}
      \e \left[|\varphi(B_s^{t,x}) -\varphi(x)| \left | H(r,s) - H(t,s) \right |\right]
& \le \left( \e  \Psg {B_s-B_t} {\ep}^2 \right)^\frac{1}{2}
      \frac{|t-r|^{\frac{1}{2}}}{|s-r|^{\frac{1}{2}}|s-t|^{\frac{1}{2}}} \\
&\le \kappa_\gamma \,  \Psg {s-t}  {\frac{\ep}{2}}
      \frac{|t-r|^{\frac{\ep}{2}}}{|s-r|^{\frac{\ep}{2}}|s-t|^{\frac{1}{2}}},
\end{align*}
where we used that $ \E|H(r,s)-H(t,s)|^2 \le  \frac{t-r}{|s-r|\,|s-t|} $ (see \cite[proof of Proposition 10]{BGGL21})
and
\[     \frac{|t-r|^{\frac{1}{2}}}{|s-r|^{\frac{1}{2}}|s-t|^{\frac{1}{2}} } 
   \le \frac{|t-r|^{\frac{\ep}{2}}}{|s-r|^{\frac{\ep}{2}}|s-t|^{\frac{1}{2}}}.\]
Combining both estimates we did prove \eqref{eq_new_claim}.
\medskip

\eqref{first}
We can assume that $t\not = r$.
Here \eqref{eq_new_claim} for $\gamma=0$, where we can chose $\kappa_\gamma=1$, reads as
\begin{align*}
    | \E(\varphi(B_s^{r,x})H(r,s)-\varphi(B_s^{t,x})H(t,s))|
&\le |t-r|^\frac{\ep}{2}
     \left [   \frac{1}{|s-r|^\frac{1}{2}} + \frac{1}{|s-r|^\frac{\ep}{2} |s-t|^{\frac{1-\ep}{2}} } \right ] \\
&\le 2 \frac{|t-r|^\frac{\ep}{2}}{|s-r|^\frac{\ep}{2} |s-t|^{\frac{1-\ep}{2}} } \\
&\le 2 \frac{|t-r|^\frac{\ep}{2}}{|s-r|^\frac{\ep}{2} |s-t|^{\frac{1-\ep}{2}} }
     \frac{\Phi_{\bar\gamma}(|t-r|^\ep)}{ \Phi_{\bar\gamma}(|s-r|^\ep)  } \\
&\le 2 \frac{|t-r|^\frac{\ep}{2}}{|s-t|^\frac{\ep}{2} |s-t|^{\frac{1-\ep}{2}} }
     \frac{\Phi_{\bar\gamma}(|t-r|^\ep)}{ \Phi_{\bar\gamma}(|s-t|^\ep)  } \\
& =  2 \frac{\Psgb  {t-r} {\frac{\ep}{2}}}{|s-t|^{\frac{1}{2}} \Phi_{\bar\gamma}(|s-t|^\ep)  },
\end{align*}
where we used
$\Phi_{\bar\gamma}(|t-r|^\ep)\ge \Phi_{\bar\gamma}(|s-r|^\ep)$ and
$|s-r|^\frac{\ep}{2}\Phi_{\bar\gamma}(|s-r|^\ep)  \ge
 |s-t|^\frac{\ep}{2}\Phi_{\bar\gamma}(|s-t|^\ep)$.
\medskip

\eqref{second}
Again we can assume that $t\not = r$. Here \eqref{eq_new_claim} reads as
\begin{align*}
      \frac{1}{\kappa_\gamma}| \E(\varphi(B_s^{r,x})H(r,s)-\varphi(B_s^{t,x})H(t,s))  |
& \le |t-r|^\frac{\ep}{2}
      \left [   \frac{\Phi_\gamma(|t-r|^{\ep})}{|s-r|^\frac{1}{2}}
              + \frac{\Phi_\gamma(|s-t|^{\ep})}{|s-r|^\frac{\ep}{2} |s-t|^{\frac{1-\ep}{2}} } \right ] \\
& \le 2 |t-r|^\frac{\ep}{2}
     \frac{\Phi_\gamma(|t-r|^{\ep} ) \Phi_\gamma (|s-t|^{\ep} )}{|s-t|^{\frac{1}{2}} }
\end{align*}
where we use that $\Phi_\gamma \ge 1$.
\end{proof}

The  next proposition is a  modification of \cite[Lemma 7 and Proposition 10]{BGGL21}. More precisely, it generalizes the inequality
of Lemma 7(b)(i) giving the possibility to weaken  the singularity in time by  slightly reducing the  regularity in space, while in
relation to Proposition 10 the H\"older continuity in time is reduced  to achieve  a weaker singularity in time.
\medskip
	
\begin{prop}\label{lemma37}  Let \cref{A1_alt} hold 
 and put $F(s,x):= {\sf f}(s,x,u(s,x),\nabla u(s,x))$.
\begin{enumerate}[{\rm (i)}]
\item \label{ma_zhang_formula}
      The function $u$ belongs to $\m C^{0,1}([t_o,T[\times\rset)$ and, for all $(t,x) \in [t_o,T[ \times \rset$,  we have,
      \begin{align}\label{eqZ}
      Z^{t,x}_s = \nabla u(s, B^{t,x}_s)  \quad \text{for a.e.} \,\,
      (s,\omega)\in [t,T[\times \Omega,
      \end{align}
       as well as
      \begin{equation} \label{nabla-u-representation}
      \nabla u(t,x)=\e\left(g(B^{t,x}_T)\frac{B_T-B_t}{T-t}\right)+\e\left(\int_t^T
       F(s,B^{t,x}_s)\frac{B_s-B_t}{s-t}ds\right).
      \end{equation}
      Consequently, for $\e_r := \e [\, \cdot\, |\m F_r]$,
      \begin{align}\label{eq: nabla-u-b}
       \nabla u(r, B^{t,x}_r)&=\e_r\left(g(B^{t,x}_T)\frac{B_T-B_r}{T-r}\right)+\e_r\left(\int_r^T
                     F(s,B^{t,x}_s)\frac{B_s-B_r}{s-r}ds\right) \notag \\& \hspace{15em}\,\,  \text{ a.s. for } \, r \in  [t,T[.
       \end{align}
\item  There exists a constant  $c_\eqref{eqn:item:2:lemma37}=c_\eqref{eqn:item:2:lemma37}(T,\ep,\gho,g(0),\fhoxt,L_{{\sf f}},c^{(2)}_{{\sf f}_0})>0$
       such that
      \label{nabla u-bound}
      \begin{equation}\label{eqn:item:2:lemma37}
      |\nabla u(t,x)|\le \frac{c_\eqref{eqn:item:2:lemma37}}{|T-t|^{(1-\ep)/2}}  \text{ for all } (t,x) \in [t_o,T[ \times \rset.
      \end{equation}
\item \label{nabla-estimate} There exists a constant
       $c_{\eqref{nabla u-diff-eq}}=c_{\eqref{nabla u-diff-eq}}(T,\ep, \beta,\gho,g(0),\fhoxt,L_{{\sf f}},c^{(2)}_{{\sf f}_0})>0$
       such that,
       \label{nabla u-diff} for all $t \in [t_0,T[,$  \, $x,y \in \rset$ and for all $\beta > \frac{1}{2}$  \text{or} $\beta =0$,
       \begin{align}   \label{nabla u-diff-eq}
           |\nabla u(t,x)-\nabla u(t,y)|
       \le c_{\eqref{nabla u-diff-eq}}  \frac{\ps {x-y}}{|T-t|^{\frac{1}{2}} \Phi_{\beta }(|T-t|^\ep)}.
       \end{align}
\item   \label{en:timenabu}
       There exists a constant  $c_{\eqref{en:timenabu-eq}}=c_{\eqref{en:timenabu-eq}}(T,\ep,\beta,\gho,g(0),\fhoxt,L_{{\sf f}},c^{(2)}_{{\sf f}_0})>0$
       such that for  $\beta >\frac{1}{2}$ or $\beta=0$ and for  all $x \in \rset$,
       \begin{align}  \label{en:timenabu-eq}
       |\nabla u(t,x) -\nabla u(r,x) | \le  c_{\eqref{en:timenabu-eq}}   \frac{\Ps  {t-r} {\frac{\ep}{2}} }{ |T-t|^{\frac{1}{2}}  \Phi_{\beta}(|T-t|^{\ep})}
        \quad  \text{ for all} \,\,\, t_o\le r <t <T.
       \end{align}
\end{enumerate}
\end{prop}
\medskip

\begin{proof}
Items \eqref{ma_zhang_formula} and \eqref{nabla u-bound} are proven in \cite[Lemma 7]{BGGL21}
where the global H\"older regularity of the generator in time is not used which means that our \cref{A1_alt} is sufficient.
\smallskip

\eqref{nabla u-diff} The representation \eqref{eq: nabla-u-b} yields to
 \begin{align} \label{L2-estimate}
	 \notag
 \| \nabla u(r, B^{t,x}_r) - \nabla u(r, B^{t,y}_r)\|_{L^2}  
 &\le \left \|    \e_r\left(\left (g(B^{t,x}_T) - g(B^{t,y}_T) \right )\frac{B_T-B_r}{T-r}\right) \right \|_{L^2}  \\
 & \quad +  \int_r^T  \left \|  \e_r\left(
                     \left(F(s,B^{t,x}_s) -F(s,B^{t,y}_s)\right) \frac{B_s-B_r}{s-r} \right)   \right \|_{L^2}   ds. 
  \end{align}

  Let us start with the $g$ difference.  Throughout the paper  we will repeatedly use that one can add terms like  $g(B^{t,x}_r)\frac{B_T-B_r}{T-r}$ to the calculation since their conditional expectation  w.r.t.~$\cF_r$ is zero. Note that
  \begin{align*}
   & \e_r\left(\left (g(B^{t,x}_T) - g(B^{t,y}_T) \right )\frac{B_T-B_r}{T-r}\right)\\ 
   &=\e_r\left(\left (g(B^{t,x}_T) - g(B^{t,y}_T) - g(B^{t,x}_r)+g(B^{t,y}_r) \right )\frac{B_T-B_r}{T-r}\right)\\
    &\le \e_r \left( \underbrace{|g(B^{t,x}_T) - g(B^{t,y}_T) - g(B^{t,x}_r)+g(B^{t,y}_r)|}_{:=A}\frac{|B_T-B_r|}{T-r} \right).
  \end{align*}

Since  $g$ is $\ep$-Hölder  continuous we have the estimates 
$$ A \le 2  \gho |x-y|^\ep     \quad  \text{and} \quad A\le   2 \gho |B_T-B_r|^\ep,$$
and therefore 
$$ A  \le 2 \gho \min\{ |x-y|^\ep,  |B_T-B_r|^\ep  \}.  $$
Setting $v:=T-r \in\, ]0,T]$ we are interested in the estimate of 
\begin{align*} 
K &:= \e  \left ( \min\{ |x-y|,  |B_v|\}^\ep \, \frac{|B_v|}{v} \right ) \\
     &=  \frac{1}{v}  \int_{|B_v|\le |x-y|} |B_v|^{1+\ep} d \p  + \frac{ |x-y|^\ep}{v}   \int_{|B_v| > |x-y|} |B_v|    d \p \\
     &=:K_1+K_2.
\end{align*}
For $K_1$ we use that $B_v$ and $\sqrt{v} B_1$ have the same distribution  and get
\begin{align*} 
K_1&=\frac{1}{v}   \int_{|z|\le \frac{|x-y|}{\sqrt{v}}}   |z|^{1+\ep}  v^\frac{1+\ep}{2}  e^{-z^2/2}  \frac{dz}{\sqrt{2\pi}} \\
& \\
      &\le \begin{cases}   \sqrt{\frac{1}{2\pi}} \,\, v^\frac{\ep-1 }{2} \left ( \frac{|x-y|}{\sqrt{v}} \right)^{2+\ep}  &\mbox{if } \frac{|x-y|}{\sqrt{v}}  \le 1  \\
      & \\
      v^\frac{\ep-1 }{2} \e |B_1|^{1+\ep}    &\mbox{if } \frac{|x-y|}{\sqrt{v}}  > 1 
      \end{cases}  \\
      & \\
& =   \frac{|x-y|^\ep}{\sqrt{v}}    \left(\sqrt{\frac{1}{2 \pi}}  \left ( \frac{|x-y|}{\sqrt{v}} \right)^{2}  \ind_{ \{ |x-y|  \le \sqrt{v}  \}}  +  \left ( \frac{\sqrt{v}}{|x-y|}\right )^\ep   \e |B_1|^{1+\ep} \ind_{ \{ |x-y|  > \sqrt{v}  \}}   \right ).
\end{align*}
Similarly, we have for $K_2$ that
\begin{align*} 
K_2&=  \frac{ |x-y|^\ep}{v}   \int_{|z| >\frac{|x-y|}{\sqrt{v}}} |z| v^\frac{1}{2}  e^{-z^2/2}  \frac{dz}{\sqrt{2\pi}} \\
& \le \frac{ |x-y|^\ep}{\sqrt{v}} \left ( \e|B_1| \, \ind_{ \{ |x-y|  \le \sqrt{v}  \}}  +    \sqrt{\frac{2}{\pi}}  e^{-|x-y|^2/2v} \,\ind_{ \{ |x-y|  > \sqrt{v}  \}}     \right ).
\end{align*}
Hence
\begin{align*} 
K 
&\le 2 \left [ \frac{ |x-y|^\ep}{\sqrt{v}}   \ind_{ \{ |x-y|  \le \sqrt{v}  \}}  +  v^\frac{\ep -1}{2}\ind_{ \{ |x-y|  > \sqrt{v}  \}} \right ],
\end{align*} 
where we used that  $  e^{-\frac{z^2}{2}} <  z^{-\ep}$   for $z>1$.
Since $\ps {\cdot}$ is non-decreasing we get  that
\[    v^\frac{\ep -1}{2}\ind_{ \{ |x-y|  > \sqrt{v}  \}}
   =  \frac{1}{v^{\frac{1}{2}}   \Phi_{\beta }(v^\ep)}  v^\frac{\ep}{2}    \Phi_{\beta }(v^\ep)  \ind_{ \{ |x-y|  > \sqrt{v}  \}}
  \le \frac{1}{v^{\frac{1}{2}}   \Phi_{\beta }(v^\ep)} \ps {x-y}. \]
For the case $|x-y|  \le  \sqrt{v} $  we use that  $ \Phi_{\beta }(v^\ep) \le \Phi_{\beta }(|x-y|^{2\ep})$ (if $x\not =y$)
and get as well
\[     \frac{ |x-y|^\ep}{\sqrt{v}}   \ind_{ \{ |x-y|  \le \sqrt{v}  \}}
   \le \frac{1}{v^{\frac{1}{2}}   \Phi_{\beta }(v^\ep)}  \ps {x-y}.  \]
Consequently, 
\begin{align*}   \left | \e_r\left(\left (g(B^{t,x}_T) - g(B^{t,y}_T) \right )\frac{B_T-B_r}{T-r}\right) \right | 
&\le 4 \gho  \, \frac{\ps {x-y}}{|T-r|^{\frac{1}{2}}   \Phi_{\beta }(|T-r|^\ep)}  .
\end{align*}

 We continue with the estimate of the  second term on the RHS~of  \eqref{L2-estimate}. By the conditional Cauchy-Schwarz inequality  we obtain
$$ \left \|  \e_r\left(
                     (F(s,B^{t,x}_s) -F(s,B^{t,y}_s)) \frac{B_s-B_r}{s-r} \right)   \right \|_{L^2} \le \left \| 
                     F(s,B^{t,x}_s) -F(s,B^{t,y}_s)   \right \|_{L^2}  \frac{1}{\sqrt{s-r}}.   $$
Using  \eqref{eq:prop_f_new} we have
\begin{align}\label{F_difference2}
	&| F(s,B^{t,x}_s) -F(s,B^{t,y}_s)| \\
	&\leq \fhoxt |x-y|^{\ep}+    L_{{\sf f}}| u(s,B^{t,x}_s) -u(s,B^{t,y}_s) |+ L_{{\sf f}}|  \nabla u(s,B^{t,x}_s)-  \nabla u(s,B^{t,y}_s)|, \notag
\end{align}
and the Hölder continuity of $u$ stated in \cref{en:regu} yields
\begin{align} \label{F-difference}
	 | F(s,B^{t,x}_s) -F(s,B^{t,y}_s)| 
 \le( \fhoxt +  L_{{\sf f}}\, c_{\eqref{u-estimates}})  |x -y |^\ep+ L_{{\sf f}}|  \nabla u(s,B^{t,x}_s)-  \nabla u(s,B^{t,y}_s)|.
\end{align}
By combining the above estimates we conclude from \eqref{L2-estimate}  that 
\begin{align*}
     \| \nabla u(r, B^{t,x}_r) - \nabla u(r, B^{t,y}_r)\|_{L^2}
&\le   4 |g|_\ep \frac{ \ps {x-y} }{|T-r|^{\frac{1}{2}}   \Phi_{\beta }(|T-r|^\ep)} \\
& \quad + 2\, ( \fhoxt + c_{\eqref{u-estimates}}\, L_{{\sf f}})\, |x -y |^\ep \sqrt{T-r} \\
& \quad +   \int_r^T  \left \|   \nabla u(s,B^{t,x}_s)-  \nabla u(s,B^{t,y}_s)  \right \|_{L^2} \frac{ L_{{\sf f}}}{\sqrt{s-r}}  ds.
\end{align*}
By \cref{Psi-lemma}  we have
\[ 1 =  \frac{|T-t|^{\frac{1}{2}} \Phi_{\beta }(|T-t|^\ep)}{|T-t|^{\frac{1}{2}} \Phi_{\beta }(|T-t|^\ep)}
     =  \frac{|T-t|^{\frac{1 - \ep}{2}} \,\,\Ps  {T-t} {\frac{\ep}{2}} }{|T-t|^{\frac{1}{2}} \Phi_{\beta }(|T-t|^\ep)}
    \le \frac{T^{\frac{1 - \ep}{2}} \,\, \Ps  {T} {\frac{\ep}{2}} }{|T-t|^{\frac{1}{2}} \Phi_{\beta }(|T-t|^\ep)} \]
implying
\begin{align}\label{eqn:nable_u_from_above}
&    \| \nabla u(r, B^{t,x}_r) - \nabla u(r, B^{t,y}_r)\|_{L^2} \notag \\
&\le c_\eqref{eqn:nable_u_from_above}   \frac{  \ps {x-y} }{|T-r|^{\frac{1}{2}}   \Phi_{\beta }(|T-r|^\ep)}
     +   L_{{\sf f}} \int_r^T \frac{ \left \|   \nabla u(s,B^{t,x}_s)-  \nabla u(s,B^{t,y}_s)  \right \|_{L^2}}{\sqrt{s-r}}  ds
\end{align}
for $c_\eqref{eqn:nable_u_from_above}:=  4|g|_\ep +2 T^{1-\frac{\ep}{2}} ( \fhoxt + L_{{\sf f}}\, c_{\eqref{u-estimates}})
  \, \Ps  {T} {\frac{\ep}{2}}$.
  Because of \eqref{eqn:item:2:lemma37} we have the integrable bound
\begin{equation*}
	\left \|   \nabla u(s,B^{t,x}_s)-  \nabla u(s,B^{t,y}_s)  \right \|_{L^2} \le 2 \frac{c_\eqref{eqn:item:2:lemma37}}{|T-s|^{(1-\ep)/2}},
\end{equation*}
so that  we may apply Gronwall's lemma  (\cref{volterra_gronwall}) with
$a:= c_\eqref{eqn:nable_u_from_above} \ps {x-y}$ and
$b:= L_{{\sf f}}$ and get
\[      \| \nabla u(r, B^{t,x}_r) - \nabla u(r, B^{t,y}_r)\|_{L^2}
  \le  c_\eqref{eqn:volterra_gronwall} c_\eqref{eqn:nable_u_from_above}
       \frac{ \ps {x-y} }{|T-r|^{\frac{1}{2}}  \Phi_{\beta }(|T-r|^\ep)}. \]
Especially, for $r=t$ this implies
\begin{equation}\label{eqn:nabla_u_diff}
  | \nabla u(t, x) - \nabla u(t, y)|
  \leq c_\eqref{eqn:volterra_gronwall} c_\eqref{eqn:nable_u_from_above} \frac{ \ps {x-y} }{ |T-t|^{\frac{1}{2}} \Phi_{\beta }(|T-t|^\ep)}
\end{equation}
for some  $c_\eqref{eqn:volterra_gronwall}=c_\eqref{eqn:volterra_gronwall}(T,\ep,L_{{\sf f}},\beta)>0$ where
$c_\eqref{eqn:volterra_gronwall}(T,\ep,L_{{\sf f}},\beta)$ is non-decreasing in $T$.
\medskip

\eqref{en:timenabu}
We use \eqref{nabla-u-representation}  to write
 \begin{align*}
  \nabla u(r,x)-\nabla u(t,x)=&\e\left[ g\left(B_T^{r,x}\right)H(r,T)-g\left(B_T^{t,x}\right) H(t,T)\right]\\
  &+\E \int_t^T F(s,B_s^{r,x})H(s,r)-F(s,B_s^{t,x}) H(s,t) ds \\
  &+  \E \int_r^t (F(s,B_s^{r,x})-F(s,x))H(r,s) ds
 \end{align*}
where $H$ was defined in \cref{lemma36}. We bound the three terms of the right hand side:
\medskip

\underline{First term:}  By applying \cref{lemma36}\eqref{first} for $(\gamma,\bar\gamma)=(0,\beta)$ and $s=T$
we get
\begin{align*}
    \left |\e\left[ g\left(B_T^{r,x}\right)H(r,T)-g\left(B_T^{t,x}\right) H(t,T)\right]\right |
\le 2 \kappa_\beta \gho  \frac{\Ps {t-r} {\frac{\ep}{2}}}{  |T-t|^{\frac{1}{2}} \Phi_{\beta}(|T-t|^{\ep}) }.
\end{align*}

\underline{Second term:}  Using \eqref{F-difference}  for $s=t$ and \eqref{eqn:nabla_u_diff}
leads to
\begin{align} \label{F-Lipschitz}
     | F(s,x) - F(s,y)|
&\le ( \fhoxt +  L_{{\sf f}}\, c_{\eqref{u-estimates}})  |x -y |^\ep+ L_{{\sf f}}|  \nabla u(s,x)-\nabla u(s,y)| \notag \\
&\le ( \fhoxt +  L_{{\sf f}}\, c_{\eqref{u-estimates}})  |x -y |^\ep+ L_{{\sf f}}
     c_\eqref{eqn:volterra_gronwall} c_\eqref{eqn:nable_u_from_above}
     \frac{  \ps {x-y} }{|T-s|^{\frac{1}{2}} \Phi_{\beta }(|T-s|^\ep)} \notag \\
& \leq c_{\eqref{F-Lipschitz}}    \frac{\ps {x-y} }{|T-s|^{\frac{1}{2}} \Phi_{\beta }(|T-s|^\ep)},
\end{align}
with
\[ c_{\eqref{F-Lipschitz}}: =  ( \fhoxt +  L_{{\sf f}}\, c_{\eqref{u-estimates}}) T^\frac{1-\ep}{2} \, \Ps T  {\frac{\ep}{2}}
                               +L_{{\sf f}} c_\eqref{eqn:volterra_gronwall} c_\eqref{eqn:nable_u_from_above}. \]
Hence we may apply \cref{lemma36}\eqref{second}  and get
\begin{align*}
&    \E \int_t^T |F(s,B_s^{r,x})H(r,s)-F(s,B_s^{t,x}) H(t,s) |ds \\
&\le  c_{\eqref{F-Lipschitz}} \, 2 \, \kappa_\beta \, \Ps {t-r} {\frac{\ep}{2}}
      \int_t^T \frac{\Phi_\beta (|s-t|^{\ep})}{ |T-s|^{\frac{1}{2}}  \Phi_{\beta }(|T-s|^\ep) |s-t|^{\frac{1}{2}}    }  ds \\
&\le  c_{\eqref{F-Lipschitz}} \, 2 \, \kappa_\beta  \frac{\Ps {t-r} {\frac{\ep}{2}}}{ \Phi_{\beta }(|T-t|^\ep)}
      \int_t^T \frac{\Phi_\beta (|s-t|^{\ep})}{ |T-s|^{\frac{1}{2}}|s-t|^{\frac{1}{2}}    }  ds.
\end{align*}
For the integral term we observe
\begin{align}\label{eqn:beta_psi_integral_upper_bound}
    \int_t^T \frac{\Phi_\beta (|s-t|^{\ep})}{ |T-s|^{\frac{1}{2}}|s-t|^{\frac{1}{2}}    }  ds
& = \int_t^\frac{T+t}{2} \frac{\Phi_\beta (|s-t|^{\ep})}{ |T-s|^{\frac{1}{2}}|s-t|^{\frac{1}{2}}    }  ds
    + \int_\frac{T+t}{2}^T \frac{\Phi_\beta (|s-t|^{\ep})}{ |T-s|^{\frac{1}{2}}|s-t|^{\frac{1}{2}}    }  ds \notag \\
&\le \sqrt{\frac{2}{T-t}} \left [ \int_0^T \frac{\Phi_\beta (|s|^{\ep})}{ s^{\frac{1}{2}}    }  ds
    + \int_\frac{T+t}{2}^T \frac{\Phi_\beta (|s-t|^{\ep})}{ |T-s|^{\frac{1}{2}}    }  ds  \right ]  \notag \\
& = \sqrt{\frac{2}{T-t}} \left [ \int_0^T \frac{\Phi_\beta (|s|^{\ep})}{ s^{\frac{1}{2}}    }  ds
    + \int_\frac{T+t}{2}^T \frac{ |s-t|^\frac{\ep}{2}  \Phi_\beta (|s-t|^{\ep})}{ |s-t|^\frac{\ep}{2} |T-s|^{\frac{1}{2}}    }  ds  \right ]  \notag \\
&\le \sqrt{\frac{2}{T-t}} \left [ \int_0^T \frac{\Phi_\beta (|s|^{\ep})}{ s^{\frac{1}{2}}    }  ds
    + \Ps T {\frac{\ep}{2}}  \int_\frac{T+t}{2}^T \frac{1}{|s-t|^\frac{\ep}{2} |T-s|^{\frac{1}{2}}    }  ds  \right ]  \notag \\
&\le \sqrt{\frac{2}{T-t}} \left [ \int_0^T \frac{\Phi_\beta (|s|^{\ep}b_\ep)}{ s^{\frac{1}{2}}    }  ds
    + \Ps T {\frac{\ep}{2}} T^\frac{1-\ep}{2} B\left (\frac{1}{2},\frac{1}{2}\right ) \right ] \notag \\
& =: c_\eqref{eqn:beta_psi_integral_upper_bound} \, |T-t|^{-\frac{1}{2}}
\end{align}
for some $c_\eqref{eqn:beta_psi_integral_upper_bound} = c_\eqref{eqn:beta_psi_integral_upper_bound}(T,\ep,\beta)>0$.
Summarizing, we proved that
\[
      \E \int_t^T |F(s,B_s^{r,x})H(r,s)-F(s,B_s^{t,x}) H(t,s) |ds
 \le 2 \kappa_\beta c_{\eqref{F-Lipschitz}}c_\eqref{eqn:beta_psi_integral_upper_bound}
     \frac{\Ps {t-r} {\frac{\ep}{2}}}
          {|T-t|^\frac{1}{2} \Phi_{\beta }(|T-t|^\ep)}. \]
\underline{Third term:} Using \eqref{F-Lipschitz}, H\"older's inequality, and \cref{Psi-lemma}
we obtain
\begin{align*}
     \left| \E \int_r^t (F(s,B_s^{r,x})-F(s,x))\frac{B_s-B_r}{s-r} ds \right|
&\le c_{\eqref{F-Lipschitz}}  \kappa_\beta \int_r^t
     \frac{\Ps {s-r} {\frac{\ep}{2}}}
          {|T-s|^{\frac{1}{2}} \Phi_{\beta }(|T-s|^\ep) |s-r|^\frac{1}{2}}ds\\
&\le c_{\eqref{F-Lipschitz}} \kappa_\beta
     \frac{ \Ps {t-r} {\frac{\ep}{2}} }
          {|T-t|^{\frac{1}{2}} \Phi_{\beta }(|T-t|^\ep)  } \int_r^t \frac{1} {|s-r|^\frac{1}{2}}ds\\
&= 2c_{\eqref{F-Lipschitz}} \kappa_\beta \frac{\Ps {t-r} {\frac{\ep}{2}}}
   {|T-t|^{\frac{1}{2}} \Phi_{\beta }(|T-t|^\ep)  }  |t-r|^\frac{1}{2}.
\end{align*}
Collecting all terms proves item \eqref{en:timenabu}.
\end{proof}


\section{Local H\"older continuity  of the generator  with frozen mean field terms }
\label{loc-Holder-cont}

In this section we show that if the measure components are frozen, then the obtained generators satisfy  \cref{A1_alt}
with a uniform control of the constants. The third term in the estimate \eqref{eqn:integral_bound_generator-f} 
is needed to use results from \cite{BGGL21}.

\begin{prop}\label{statement:the-new-f}
Under \cref{H:A1} define (and recall)
${\sf f}, {\sf f}_n: [t_o,T[\times \R\times \R\times \R \to \R$ as
\begin{align} \label{the-new-f}
\f(t,x,y,z)   & := f(t,x,y,z, [Y^{t_o,\xi}_t], [Z^{t_o,\xi}_t] ), \\
\fn(t,x,y,z) &  = f(t\wedge t_{n-1} ,x,y,z,[Y^{t_o,\xi}_t],  [Z^{t_o,\xi}_{t\wedge t_{n-1}}])
\end{align}
with $\|\xi\|_{L^{2\ep}} \le M$.
Then there exists a version of $(Z^{t_o,\xi}_t)_{t\in [t_o,T[}$ such that
$\f$ and $\f_n$ satisfy \cref{A1_alt}:
For $(t,x,y,z),(t',x',y',z')\in [t_o,T[\times\rset\times\rset\times\rset $ one has
\begin{align}
       \left | {\sf f}(t,x,y,z)- {\sf f}(t',x',y',z')\right|
& \leq c_\eqref{eq:the-new-f-new}  \,\frac{ \left|t-t'\right|^{\alpha\wedge \frac{\ep}{2} }}{|T-(t\vee t')|^\frac{1}{2}}
       + \fhox\, \left|x-x'\right|^\ep \notag \\
&      \quad + L_f \Big (  \left|y-y'\right|+ \left|z-z'\right| \Big ), \label{eq:the-new-f-new} \\
       \left | \fn(t,x,y,z)- \fn(t',x',y',z')\right|
& \leq c_\eqref{eq:the-new-fn-new}   \,\frac{ \left|t-t'\right|^{\alpha\wedge \frac{\ep}{2} }}{|T-(t_{n-1}\wedge (t\vee t'))|^\frac{1}{2}}
       + \fhox\, \left|x-x'\right|^\ep \notag \\
&      \quad  + L_f \Big (  \left|y-y'\right|+ \left|z-z'\right| \Big ), \label{eq:the-new-fn-new}
\end{align}
and
\begin{equation}\label{eqn:integral_bound_generator-f} 
       \int_{t_o}^T |{\sf f}(t,0,0,0)|^2 d t
     + \int_{t_o}^T  |{\fn}(t,0,0,0)|^2  d t   
     + \int_{]t_o,T]} |{\fn}(t,0,0,0)|^2 d\langle \rwB \rangle_t
    \le  c_\eqref{eqn:integral_bound_generator-f}
\end{equation}
for $c_\eqref{eq:the-new-f-new},c_\eqref{eq:the-new-fn-new},c_\eqref{eqn:integral_bound_generator-f}>0$
depending at most on $(T,\ep,\alpha,\gho,g(0),\fhot,\fhox ,$ $L_f,c_{f_0}^{(2)},M)$.
\end{prop}
\medskip

\begin{proof}
(a) We prove that there is a
$ c_\eqref{eqn:upper_bound:Wasserstein-xi-solution}
 =c_\eqref{eqn:upper_bound:Wasserstein-xi-solution}(T,\ep,\gho,g(0),\fhox ,L_f,c_{f_0}^{(2)},M)>0$ such that
\begin{align} \label{eqn:upper_bound:Wasserstein-xi-solution}
 \left \| Y^{t_o,\xi}_t - Y^{t_o,\xi}_{t'} \right \|_{L^2} \le  c_\eqref{eqn:upper_bound:Wasserstein-xi-solution} |t-t'|^\frac{\ep}{2} 
 \sptext{1}{and}{1}
 \left\| Z^{t_o,\xi}_t - Z^{t_o,\xi}_{t'} \right \|_{L^2}   \le c_\eqref{eqn:upper_bound:Wasserstein-xi-solution} \frac{  |t-t'|^{\frac{\ep}{2}}  \,  }{ |T-t\vee t'|^{\frac{1}{2}}  }, 
\end{align}
provided an appropriate version for the $Z$-process is chosen.
To show \eqref{eqn:upper_bound:Wasserstein-xi-solution}
we start with $\mu_t^0=\nu_t^0:= \delta_0$ and consider the iteration
\[ \left . \begin{array}{r}	
      Y_s^{k+1,t_o, \xi} 
    = g\left(B^{t_o,\xi}_T\right) + \int_s^T   
      f\left(r,B^{t_o,\xi}_r, Y_r^{k+1, t_o, \xi}, Z_r^{k+1, t_o, \xi},  \mu_r^k, \nu_r^k \right) dr 
      - \int_s^T Z_r^{k+1, t_o, \xi}\, dB_r  \\ 
      t_o \leq s\leq T 
    \end{array} \right \}, \]
where   $\mu_r^k =[ Y_r^{k, t_o,\xi}]$ and  $ \nu_r^k  =[ Z_r^{k, t_o, \xi}  ]$.
By  \cite[Theorem A.1]{Li17} we get by the equivalence of the '$\beta$-norms'  that 
\begin{equation}\label{eqn:convergence_beta_LI}
\| (Y^{k, t_o, \xi}-Y^{t_o, \xi},  Z^{k, t_o, \xi}-Z^{t_o,\xi} )\|_0 \to 0,
\end{equation}
for $k \to \infty$, where 
$ \| (\cY, \cZ)\|_\beta :=  \left(\e \int_{t_o}^T  e^{\beta s} (\cY_s^2 + \cZ_s^2) ds\right)^\frac{1}{2} . $
This implies
\begin{align*}
& \hspace*{-3em} \left (  \int_{t_o}^T | f(t,0,0,0 ,[Y^{k, t_o, \xi}_t], [Z^{k, t_o, \xi}_t]) |^2dt \right )^\frac{1}{2} \\
& \le c^{(2)}_{f_0}  + L_f  \left (  \int_{t_o}^T [ \cW ([Y^{k, t_o, \xi}_t], \delta_0)+  \cW ( [Z^{k, t_o, \xi}_t],\delta_0)   ]^2dt \right )^\frac{1}{2}  \\
& \le c^{(2)}_{f_0}  + \sqrt{2} L_f \, \| (Y^{k, t_o, \xi}-Y^{t_o, \xi},  Z^{k, t_o, \xi}-Z^{t_o,\xi} )\|_0 
                     + \sqrt{2} L_f \, \| (Y^{t_o, \xi},  Z^{t_o,\xi}) \|_0.
\end{align*}
We choose $k_0\in \N$  such that for all $k\ge k_0$  the second term on the RHS is less than $1$.
By  \cite[Theorem A.2 used for $t=0$ and the expected value]{Li17} we upper bound  $ \| (Y^{t_o, \xi},  Z^{t_o,\xi}) \|_0$
by a $c(\| g(B^{t_o, \xi}_T) \|_{L^2}, T,\ep,\fhox ,L_f,c_{f_0}^{(2)},M)>0$ 
where $\| g\left(B^{t_o, \xi}_T\right) \|_{L^2}$ is bounded again in terms of $(T,\ep, \gho,g(0),M)$.
Hence there is a  $c_\eqref{eqn:bound_L1_generator}:= c_\eqref{eqn:bound_L1_generator}(  T,\ep, \gho,g(0),\fhox ,L_f,c_{f_0}^{(2)},M)$  such that
\begin{equation}\label{eqn:bound_L1_generator}
 c_\eqref{eqn:bound_L1_generator}:=\sup_{k\ge k_0} \int_{t_o}^T | f(t,0,0,0 ,[Y^{k, t_o, \xi}_t], [Z^{k, t_o, \xi}_t]) |^2dt < \infty
\end{equation}
which implies \eqref{eqn:integral_bound_generator-f} for the generator $\f$ with constant
$c_\eqref{eqn:integral_bound_generator-f}^\f>0$.
Now we use \cref{en:regu} for
\[  Y_t^{k,t_o, \xi} = u^{k}(t, B_t^{t_o,\xi}) \quad \text{and} \quad  Z_t^{k, t_o, \xi} =\nabla u^k(t, B_t^{t_o,\xi}),  \]
where $k>k_0$ and $u^k$ corresponds to  $f(t,x,y,z, [Y^{k-1, t_o, \xi}_t], [Z^{k-1, t_o, \xi}_t])$,  to get
\begin{align}\label{Y-difference}
          \|Y_t^{k, t_o, \xi} - Y_{t'}^{k, t_o, \xi} \|_{L^2}
&\le \|u^k(t, B_t^{t_o, \xi}) - u^k(t',B_t^{t_o, \xi})\|_{L^2} + \|u^k(t', B_t^{t_o, \xi}) - u^k(t',B_{t'}^{t_o, \xi})\|_{L^2}  \notag \\
&\le  c_\eqref{u-estimates} \left [ \left \| (1+|B_t^{t_o,\xi}|)^\ep \right \|_{L^2} |t-t'|^\frac{\ep}{2} + \left \| \left |B_t^{t_o, \xi} - B_{t'}^{t_o, \xi}\right |^\ep \right \|_{L^2}
    \right ] \notag \\
&\le  c_\eqref{u-estimates} \left [ \left \| 1+|B_t-B_{t_o}|^\ep+|\xi|^\ep \right \|_{L^2} |t-t'|^\frac{\ep}{2} + \left \| |B_{|t-t'|}|^\ep \right \|_{L^2} \right ] \notag \\
&\le c_\eqref{u-estimates}  c_{\eqref{Y-difference}} \,|t-t'|^\frac{\ep}{2},
\end{align}
where
$c_\eqref{u-estimates} =c_\eqref{u-estimates}(T,\ep,\gho, g(0), |f|_{\ep;x},L_f,c_\eqref{eqn:bound_L1_generator})>0$
and
$c_{\eqref{Y-difference}}=c_{\eqref{Y-difference}}(T,\ep,M)>0$.
\medskip

Similarly by \cref{lemma37}(\eqref{nabla-estimate},\eqref{en:timenabu}) applied with $\beta=0$ we get
\begin{align} 	 \label{Z-difference}
  &\|Z_t^{k, t_o, \xi} - Z_{t'}^{k, t_o, \xi} \|_{L^2}\notag  \\
  & =   \|\nabla u^k(t,B_t^{t_o, \xi}) -\nabla u^k(t',B_{t'}^{t_o, \xi}) \|_{L^2}  \notag \\
  &\le   \|\nabla u^k(t,B_t^{t_o, \xi}) -\nabla u^k(t',B_{t}^{t_o, \xi}) \|_{L^2}
        +\|\nabla u^k(t',B_{t}^{t_o, \xi}) -\nabla u^k(t',B_{t'}^{t_o, \xi}) \|_{L^2} \notag \\
  &\le c_{\eqref{Z-difference}} \frac{  |t-t'|^{\frac{\ep}{2}}  \,  }{ |T-t\vee t'|^{\frac{1}{2}}  },
\end{align}
where   $c_{\eqref{Z-difference}} = c_{\eqref{Z-difference}}(T,\ep, \gho, g(0),\fhox ,  L_f, c_\eqref{eqn:bound_L1_generator})>0$.
From
\[ \lim_{k\to \infty}\| (Y^{k, t_o, \xi}-Y^{t_o, \xi},  Z^{k, t_o, \xi}-Z^{t_o,\xi} )\|_0^2=0 \]
we conclude that there is a subsequence $(k_l)_{l=1}^\infty$, $k_l>k_0$, such that
\[ L^2\mbox{-}\lim_{l\to \infty} Y_t^{k_l, t_o, \xi} = Y_t^{t_o, \xi}
   \sptext{1}{and}{1}
   L^2\mbox{-}\lim_{l\to \infty} Z_t^{k_l,t_o, \xi} = Z_t^{t_o, \xi} \]
for $t \in I \subseteq [t_o,T[$ where $I$ is a Borel set with measure $T-t_o$.
Since for this subsequence we have the relations \eqref{Y-difference} and \eqref{Z-difference}, we get
\begin{equation}\label{eqn:proof:statement:the-new-f}
 \left \| Y^{t_o,\xi}_t - Y^{t_o,\xi}_{t'} \right \|_{L^2} \le  c_\eqref{u-estimates}  c_\eqref{Y-difference} |t-t'|^\frac{\ep}{2}
 \sptext{1}{and}{1}
 \left \| Z^{t_o,\xi}_t - Z^{t_o,\xi}_{t'} \right \|_{L^2}  \le c_{\eqref{Z-difference}} \frac{  |t-t'|^{\frac{\ep}{2}}  \,  }{ |T-t\vee t'|^{\frac{1}{2}}  }
\end{equation}
for $t,t'\in I$. Because of the continuity of the paths of $Y^{t_o,\xi}$ and since  $\E \sup_{t\in [t_o,T]} |Y^{t_o,\xi}_t|^2< \infty$ the estimate for $ Y^{t_o,\xi}$ extends
to all $t,t'\in [t_o,T]$. Moreover, we define $\tilde Z_t^{t_o,\xi}:=  Z_t^{t_o,\xi}$ for $t\in I$ and
$\tilde Z_t^{t_o,\xi}:= L_2\mbox{-}\lim_{l\to \infty} Z_{t_l}^{t_o,\xi}$ for some $(t_l)_{l=1}^\infty \subseteq I$ with $t_l\to t$. Set
$\nu_t^{t_o,\xi}$ to be the law of $\tilde Z_t^{t_o,\xi}$ so that
\[     W_2 \left (\nu^{t_o,\xi}_t,\nu^{t_o,\xi}_{t'} \right ) 
   \le c_{\eqref{Z-difference}} \frac{  |t-t'|^{\frac{\ep}{2}}}{|T-t\vee t'|^{\frac{1}{2}}}
        \sptext{1}{for all}{1} t,t'\in [t_o,T[.\]
Then $(Y_t^{t_o,\xi})_{t\in [t_o,T]}$ and $(Z_t^{t_o,\xi})_{t\in [t_o,T[}$ solve the BSDEs with
generator $f(s,x,y,z,[Y_s^{t_o,\xi}],\nu_s^{t_o,\xi})$ and the generator
satisfies \eqref{eq:the-new-f-new} and \eqref{eqn:integral_bound_generator-f}.
Therefore \cref{lemma37}\eqref{ma_zhang_formula} applies, and we may assume $(Z_t^{t_o,\xi})_{t\in [t_o,T[}$ to be continuous and adapted
and that \eqref{eqn:convergence_beta_LI} holds.
Therefore \eqref{eqn:proof:statement:the-new-f} for the $Z$-process extends to all $t,t'\in [t_o,T[$ by  Fatou's lemma. 
\medskip

(b) To verify  \eqref{eqn:integral_bound_generator-f}   for $\fn$ we conclude from  \eqref{eq:prop_f} that
\begin{align*}
&   \int_{t_o}^T | \fn (t,0,0,0) |^2dt   \\
& = \int_{t_o}^{t_{n-1}} | \f (t,0,0,0) |^2dt   + \int_{t_{n-1}}^T | f (t_{n-1},0,0,0 ,[Y^{t_o,\xi}_t], [Z^{t_o,\xi}_{t_{n-1}}] ) |^2dt  \\
&\le c^\f_\eqref{eqn:integral_bound_generator-f}
     + \int_{t_{n-1}}^T \left [ |f_0| +   \fhot \frac{t_{n-1}^\alpha}{|T-t_{n-1}|^\frac{1}{2}}  + L_f \|Y^{t_o,\xi}_t\|_{L^2} + L_f \| Z^{t_o,\xi}_{t_{n-1}}  \|_{L^2} \right ]^2 dt \\
&\le c^\f_\eqref{eqn:integral_bound_generator-f} +  4 \big ( f_0^2 T + (\fhot)^2 T^{2\alpha} )
     + 4 L_f^2   \int_{t_{n-1}}^T  \left [ \| Y^{t_o, \xi}_t  \|_{L^2}^2 + \| Z^{t_o,\xi}_{t_{n-1}}  \|_{L^2}^2 \right ] dt .
     \end{align*}
By \eqref{eqn:proof:statement:the-new-f} and the triangular inequality we get
 \begin{align*}
 \int_{t_{n-1}}^T\| Z^{t_o,\xi}_{t_{n-1}}  \|_{L^2}^2  dt  &=  \int_{t_{n-2}}^{t_{n-1}}\| Z^{t_o,\xi}_{t_{n-1}}  \|_{L^2}^2  dt \\
 &\le  2   \int_{t_{n-2}}^{t_{n-1}}\| Z^{t_o,\xi}_t  \|_{L^2}^2  dt  +2  c_{\eqref{Z-difference}}  \int_{t_{n-2}}^{t_{n-1}} \frac{  |t_{n-1}-t|^\ep }{ |T-t_{n-1}|} dt
 \end{align*}
and estimate the second term on the RHS by  $2  c_{\eqref{Z-difference}} T^\ep$.  Hence 
\[     \int_{t_o}^T | \fn (t,0,0,0) |^2dt
   \le c_\eqref{eqn:integral_bound_generator-f}^\f
       +  4 \big ( f_0^2 T + (\fhot)^2 T^{2\alpha} )  + 8 L_f^2 \left (   \| Y^{t_o, \xi},  Z^{t_o,\xi} \|_0 +  c_{\eqref{Z-difference}} T^\ep\right ). \]

(c) To verify \eqref{eqn:integral_bound_generator-f}  for $\fn$ with integrator $ d\langle \rwB \rangle_t$ we note that
\begin{align*}
&    \int_{]t_o,T]} |{\fn}(t,0,0,0)|^2 d\langle \rwB \rangle_t \\
& =  \int_{]t_o,T]} | f(\ot \wedge t_{n-1},0,0,0, [Y^{t_o,\xi}_{\ot}], [Z^{t_o,\xi}_{\ot \wedge t_{n-1}} ]  )|^2 dt \\
&\le 2 c^\fn_\eqref{eqn:integral_bound_generator-f}
     + 6 \int_{]t_o,t_{n-1}]}
             \left [  (\fhot)^2 \frac{| \ot - t |^{2\alpha}}{|T-\ot| }
                  + c_\eqref{eqn:upper_bound:Wasserstein-xi-solution}^2 \frac{  |\ot-t|^{\ep}  \,  }{ |T-\ot|} \right ] dt
                  +  6 \int_{]t_o,T]} c_\eqref{eqn:upper_bound:Wasserstein-xi-solution}^2 |\ot-t|^\ep  d t
\end{align*}
and since   $\frac{| \ot - t |}{|T-\ot| } \le 1$ for $t\in ]t_0,t_{n-1}]$
we get a  bound  for the integral depending at most on
$(\alpha, \ep, T, \fhot, c^\fn_\eqref{eqn:integral_bound_generator-f}, c_\eqref{eqn:upper_bound:Wasserstein-xi-solution})$.
\medskip

(d) Finally, \eqref{eq:the-new-f-new} and \eqref{eq:the-new-fn-new}
follow from \eqref{eqn:upper_bound:Wasserstein-xi-solution} and the properties of the generator $f$.
\end{proof}
\medskip

\begin{convention}
In the following we assume that  \eqref{eqn:upper_bound:Wasserstein-xi-solution} is satisfied  for all $t,t' \in [t_o,T[ $.
\end{convention}
\medskip

As an immediate consequence we obtain the following regularity results:
\medskip

\begin{prop}\label{regularity-for-mainBSDE}
\cref{en:regu} and \cref{lemma37} hold
\begin{enumerate}[{\rm (i)}]
\item for $u(t,x)$ and $\nabla u(t,x)$ corresponding to  ${\sf f}$ given by \eqref{the-new-f},
      i.e. for the solution to \eqref{eq:mainBSDE},
\item for $u_\fn(t,x)$ and $\nabla u_\fn(t,x)$   corresponding to  ${\sf f}_n$ given by  \eqref{f-zero},
      i.e. for the solution to \eqref{eq:mainBSDEn-tilde}.
\end{enumerate}
The corresponding constants in \cref{en:regu} and \cref{lemma37} depend at most on
the parameters
$(T,\ep,\alpha,\gho,g(0),\fhot,\fhox ,L_f,c_{f_0}^{(2)},M)$.
\end{prop}


\section{The discrete scheme}
\label{sec:discrete_scheme}

Assume $\underline{t_o} =\frac{jT}{n}$  for  some $j \in \{0,1,...,n-3\}$
and $t\in [t_o,T]$ (recall $n\ge n_0>2T/(T-t_o)$).
\bigskip

\underline{\bf Stochastic scheme:}
We provide the scheme for
\[ (Y^{n, t,x}, Z^{n, t, x})= \left ((Y^{n, t,x}_{t_l})_{l=k}^n,(Z^{n, t,x}_{t_l})_{l=k+1}^n \right )
   \sptext{1}{with}{1}
   \ut = t_k  \sptext{.5}{for some}{.5}
   k\in \{j,\ldots,n\}. \]
For $l= k,...,n-1$ one has that
\begin{align*}
  Y^{n,t,x}_{t_l} &= Y^{n,t,x}_{t_{l+1}} +h f\!\left( t_{l+1}\wedge t_{n-1},B^{n,t,x}_{t_l},Y^{n,t,x}_{t_l},Z^{n,t,x}_{t_{l+1}},[Y^{n,t_o,\xi}_{t_{l+1}}], [Z^{n,t_o,\xi}_{t_{l+1}}]\right)-\sqrt{h}\,Z^{n,t, x}_{t_{l+1}} \, \zeta_{l+1},\\
  Y^{n,t,x}_{t_n} &=g(B^{n,t,x}_T).
\end{align*}
Thus, if $Y^{n,t,x}_{t_{l+1}}$ is given and $\cF^{n,t}_{t_l} := \sigma(\zeta_{k+1},...,\zeta_l)$,
where $\cF^{n,t}_{t_k}$ is the trivial $\sigma$-algebra, then
\begin{align}
	Z^{n,t,x}_{t_{l+1}}
&:= h^{-1/2}\,\e\left[Y^{n,t,x}_{t_{l+1}} \,\zeta_{l+1}\,|\,\m F^{n,t}_{t_l}\right],     \label{z-eqn}   \\
	Y^{n,t,x}_{t_l}
&= \e\left[ Y^{n,t,x}_{t_{l+1}}\,|\,\m F^{n,t}_{t_l}\right]+h\,
    f\left( t_{l+1}\wedge t_{n-1},B^{n,t,x}_{t_l},Y^{n,t,x}_{t_l},Z^{n,t,x}_{t_{l+1}},[Y^{n,t_o,\xi}_{t_{l+1}}], [Z^{n,t_o,\xi}_{t_{l+1}}]
	\right),
	 \label{y-eqn}
\end{align}
where in \eqref{y-eqn}  the value  $Y^{n,t,x}_{t_l}$ is the corresponding fixed point (recall that $L_f \frac{T}{n_0} <1$).
Similarly to \eqref{z-eqn} and \eqref{y-eqn} we define 
\[ (Y^{n,t,x}_{{\sf f}_n}, Z^{n,t,x}_{{\sf f}_n}) = \left ((Y^{n, t,x}_{\fn,t_l})_{l=k}^n,(Z^{n, t,x}_{\fn,t_l})_{l=k+1}^n \right ) \]
as
\begin{align}
    Z^{n,t,x}_{{\sf f}_n,t_{l+1}}
&:= h^{-1/2}\,\e\left[Y^{n,t,x}_{{\sf f}_n,t_{l+1}} \,\zeta_{l+1}\,|\,\m
    F^{n,t}_{t_l}\right],     \label{z-fn-eqn}   \\
	Y^{n,t,x}_{{\sf f}_n,t_l}
& = \e\left[ Y^{n,t,x}_{{\sf f}_n,t_{l+1}}\,|\,\m F^{n,t}_{t_l}\right]
	  +h\, f\left( t_{l+1}\wedge t_{n-1},B^{n,t,x}_{t_l},Y^{n,t,x}_{{\sf f}_n, t_l},Z^{n,t,x}_{{\sf f}_n, t_{l+1}},[Y^{t_o,\xi}_{t_{l+1}}], [Z^{t_o,\xi}_{t_{l+1}\wedge t_{n-1}}]
	\right)  \notag\\
	 \label{y-fn-eqn}
&  = \e\left[ Y^{n,t,x}_{{\sf f}_n,t_{l+1}}\,|\,\m F^{n,t}_{t_l}\right]
	  +h\, \fn\left( t_{l+1},
	B^{n,t,x}_{t_l},Y^{n,t,x}_{{\sf f}_n, t_l},Z^{n,t,x}_{{\sf f}_n, t_{l+1}}
	\right),
\end{align}
where the generator ${\sf f}_n$ was given in \eqref{f-zero}.
\bigskip

\underline{\bf Finite difference scheme:}
The related difference scheme is given by
\begin{align}\label{eq:mainPDEnfn}
 \partial^h_t U^n_\fn(t_l,x)  + L^h U^n_\fn(t_{l+1},x) + \fn\left (t_{l+1},x,U^n_\fn(t_l,x),   \nabla^h U^n_\fn(t_{l+1},x)\right ) & = 0,
 \quad l=j,...,n-1, \notag \\
  \!\!\!U^n_\fn(t_n,x)&=g(x),
\end{align}
where
\begin{align*}
L^h u(x) &:= \frac{1}{2h}\left(u(x+\sqrt{h})+u(x-\sqrt{h})  - 2u(x)\right), \\
\nabla^h u(x)&:= \frac{1}{2 \sqrt{h} }\left(u(x+\sqrt{h})-u(x-\sqrt{h})\right), \\
\partial^h_t u(t) &:= \frac{1}{h} (u(t+h)-u(t)).
\end{align*}
\underline{\bf Connection between the schemes:}
First we extent $U^n$ to 
\[ U^n:[t_j,T]\times \R\to \R \sptext{1}{by}{1}
   U^n_\fn(s,x):= U^n_\fn(\us,x)\]
and set

\begin{equation}
 \Delta^n_\fn: [t_j,T[ \times \R\to \R   \sptext{1}{by}{1} \Delta^n_\fn(s,x)  
:= \nabla^h U^n_\fn  (\us+h,x) \sptext{1}{for}{1} s\in [t_j,T[.
   \label{eqn:def:Deltanfn}
\end{equation}
We will exploit relations from  \cite{BGGL21}:
\begin{align}
   \os                     & := \min \left\{ t_k \mid k \in \left\{ 0, ..., n \right\}, s\leq t_k  \right\},  \notag \\
   Y^{n,t,x}_{{\sf f}_n,s} & = Y^{n,t,x}_{{\sf f}_n,\us}=U^n_\fn(s,B^{n,t,x}_s) \sptext{2.4}{for}{1} s\in [\ut,T], \label{eqn:fn:U->Y} \\
   Z^{n,t,x}_{{\sf f}_n,s} & = Z^{n,t,x}_{{\sf f}_n,\os}=\nabla^h U^n_\fn(\os,B^{n,t,x}_{s-}) \sptext{1}{for}{1} s \in \,]\ut,T],
                                                                                                                   \label{eqn:fn:nablaU->Z}  \\
   Z^{n,t,x}_{{\sf f}_n,s} & = \Delta^n_\fn(s,B^{n,t,x}_s) \sptext{6}{for}{1} s \in \,]\ut,T] \sptext{.5}{and}{.5} s \neq \us,
         \label{eqn:fn:Delta->Z-I} \\
   Z^{n,t,x}_{\fn,t_l} & =  \Delta^n_\fn(t_{l-1},B^{n,t,x}_{t_{l-1}}) \sptext{4.7}{for}{1}  l=k+1,\ldots,n.  \label{eqn:fn:Delta->Z-II} 
\end{align}
In exactly the same way we get for the processes defined by   \cref{eq:mainBSDEn-hat}  with generator  
\[ f(t_{l+1} \wedge t_{n-1},  x, y, z, [   Y^{n,t_o, \xi}_{t_{l+1}}], [Z^{n,t_o, \xi}_{t_{l+1}}] )\]
a finite difference scheme like in \eqref{eq:mainPDEnfn} denoting its solution by 
\begin{equation}\label{eqn:def:mainPDEnf_Deltanf}
U^n:[t_j,T]\times \R\to \R \sptext{1}{and letting}{1}
\Delta^n(s,x) :=\nabla^h U^n(\us+h,x)
\end{equation}
and the relations corresponding to
\eqref{eqn:fn:U->Y}--\eqref{eqn:fn:Delta->Z-II}. 


\section{From continuous to discrete in three steps}
\label{sect:conv_rate}

In this section we establish necessary results required to deduce the main result, \cref{coro_final_pathwise_estimate}.


\subsection{Comparing \texorpdfstring{$(Y,Z)$}{Y,Z} and \texorpdfstring{$(Y_{{\sf f}_n}, Z_{{\sf f}_n})$}{Yf,Zf} }
We start by comparing the solution to \eqref{eq:mainBSDE} to the solution to \eqref{eq:mainBSDEn-tilde}.

\begin{prop}\label{prop_tilde_difference}
Under \cref{H:A1} for the solutions to  \eqref{eq:mainBSDE} and \eqref{eq:mainBSDEn-tilde}
one has
\begin{align}
 \label{cont_tilde_a_priori}
     |u(t,x)-u_{{\sf f}_n}(t,x)|& +
     \left (  \e_{t,x} \int_t^T |\nabla u(s,B_s) - \nabla u_\fn(s,B_s)|^2  \od s  \right )^\frac{1}{2}
\le c_{\eqref{cont_tilde_a_priori}} \, n^{-{((\alpha \wedge \frac{\ep}{2}) +\frac{1}{2})}}
\end{align}
for $n \ge n_0$ and $(t,x) \in [t_o,T]\times \R$, where the constant
$c_{\eqref{cont_tilde_a_priori}}>0$
depends at most on $(T,\ep,\alpha, \gho, g(0),\fhot, \fhox,L_f,c^{(2)}_{f_0},M)$.
Moreover, there is an absolute  constant $c_\eqref{eqn:statement:Lexp_bound_difference_gradient_continuous_case}>0$ such that
\begin{equation}\label{eqn:statement:Lexp_bound_difference_gradient_continuous_case}
    \left |  \int_t^T |\nabla u(s,B_s^{t,x}) - \nabla u_\fn(s,B_s^{t,x})|^2  \od s  \right |_{M^0_{\varphi_1}}^\frac{1}{2}
\le c_\eqref{eqn:statement:Lexp_bound_difference_gradient_continuous_case} c_{\eqref{cont_tilde_a_priori}}  \, n^{-{((\alpha \wedge \frac{\ep}{2}) +\frac{1}{2})}}.
\end{equation}
\end{prop}
\begin{proof} 
To show \eqref{cont_tilde_a_priori} we use an  a priori estimate (see, for example, \cite[Lemma 5.26]{GeissYlinen}  which is also true for conditional  expectations)
and use \eqref{eqn:upper_bound:Wasserstein-xi-solution}  to get, for some $c=c(T,L_f)$,
\begin{align*} 
&     \e \left [  |Y_t^{t,x} -Y^{t,x}_{{\sf f}_n,t}|^2 +  \int_{t}^T |Z_s^{t,x} -Z^{t,x}_{{\sf f}_n,s}|^2  ds |\cF_t \right ]  \notag\\
& \le c \, \e \Bigg [ \Bigg |\int_{t}^T   |f (s,B^{t,x}_s,Y^{t,x}_s, Z^{t,x}_s, [Y_s^{t_o,\xi}], [Z_s^{t_o,\xi}] ) \\
& \hspace{15em}                                        -f( s\wedge t_{n-1} ,B^{t,x}_s,Y^{t,x}_s, Z^{t,x}_s,[Y^{t_o,\xi}_s],  [Z^{t_o,\xi}_{s\wedge t_{n-1}}]) | ds  \Bigg |^2 | \cF_t \Bigg ]  \\
&\le  c  \left | \int_{t}^T  \left [ \fhot  \,\frac{ | s- t_{n-1}|^{\alpha}}{|T-s|^\frac{1}{2}}  + L_f
      \cW_2 (Z^{t_o,\xi}_{s\wedge t_{n-1}}, Z^{t_o,\xi}_s) \right ]  ds  \right |^2 \\
&\le  c  \left | \int_{t}^T  \left [ \fhot  \,\frac{ |s-s\wedge   t_{n-1}|^{\alpha}}{|T-s|^\frac{1}{2}}
      + L_f  c_{\eqref{eqn:upper_bound:Wasserstein-xi-solution}}\,  \frac{|s-s\wedge t_{n-1}|^\frac{\ep}{2}}{|T-s|^\frac{1}{2}}  \right ]  ds  \right |^2 \\
& \le c' \left | \int_{t_{n-1}}^T  \frac{|s-t_{n-1}|^{\frac{\ep}{2}\wedge \alpha} }{|T-s|^\frac{1}{2}}    ds  \right |^2
  \le c"   n^{-{2(\alpha \wedge \frac{\ep}{2} +\frac{1}{2})}}.
\end{align*}
Finally we use that
$Y^{t,x}_t              = u(t,x)$,
$Y^{t,x}_{{\sf f}_n,t}  = u_{{\sf f}_n}(t,x)$,
$Z^{t,x}_t              = \nabla u(t,x)$, and
$Z^{t,x}_{{\sf f}_n,t}  = \nabla u_{{\sf f}_n}(t,x)$. Setting $\e_{t,x}[ \cdot] :=    \e \left [  \cdot | B_t =x\right ]$ we conclude
\begin{align*}  
\e \left [   \int_{t}^T |Z_s^{t,x} -Z^{t,x}_{{\sf f}_n,s}|^2  ds |\cF_t \right ] &=  \e \left [    \int_{t}^T  |\nabla u(s,B^{t,x}_s) - \nabla u_\fn(s,B^{t,x}_s)|^2  ds |\cF_t \right ] \\
&= \e_{t,x} \int_t^T |\nabla u(s,B_s) - \nabla u_\fn(s,B_s)|^2  \od s. 
\end{align*}
For the relation \eqref{eqn:statement:Lexp_bound_difference_gradient_continuous_case} we use
\eqref{cont_tilde_a_priori} and  \cref{statement:JN_continuous_time}.  Indeed, setting   
\[         A_s :=         \int_t^s |\nabla u(r,B_r^{t,x}) - \nabla u_\fn(r,B_r^{t,x})|^2  \od r     \]
we derive from  \eqref{cont_tilde_a_priori} and  \cref{contBMO-def} that    
\[\|A\|_{{\rm BMO}([0,T])} \le  c_{\eqref{cont_tilde_a_priori}}^2 \, n^{-{((2 \alpha \wedge \ep) +1)}}.    \]
Then  \cref{statement:JN_continuous_time} implies that 
 \[ |A_T-A_t|_{M^0_{\varphi_1}}^\frac{1}{2} \le c_\eqref{eqn:statement:JN_continuous_time}^\frac{1}{2}  c_{\eqref{cont_tilde_a_priori}} \, n^{-{((\alpha \wedge \frac{\ep}{2}) +\frac{1}{2})}}  .  \]
\end{proof}


\subsection{Finite difference approximation - comparing \texorpdfstring{$(Y_{\sf f_n},Z_{\sf f_n})$ and $(Y^n_{\sf f_n},Z^n_{\sf f_n})$}{comp}}
\label{sec:middle_part}

We will use   \cite[Corollary 5]{BGGL21}  and reformulate it as follows: 
\begin{lemma} \label{corollary5}
There is a $c_\eqref{eqn:corollary5_g}= c_\eqref{eqn:corollary5_g}(T)> 0$
such that for $x\in\rset$ and $t_o\leq t\leq s\leq T$ one has
\begin{equation} \label{eqn:corollary5_g}
\left| \e\left[g\left(B^{n,t,x}_s\right)\right]-\e\left[g\left(B^{t,x}_s\right)\right]\right| \leq c_\eqref{eqn:corollary5_g} \, \gho\, n^{-\ep/2}.
\end{equation}
Moreover, for $\beta\in \{0\}\cup ]\frac{1}{2},\infty[$ there is a $c_\eqref{eqn:corollary5_g}= c_\eqref{eqn:corollary5_g}(T,\beta)> 0$
such that for $\delta(t,s):=\max\left(s-t,\us-\ut\right)$ one has
\begin{multline}\label{eqn:corollary5_nabla_g}
\left| \e\left[g\left(B^{n,t,x}_s\right)\left(B^{n,t,x}_s-x\right)\right]-\e\left[g\left(B^{t,x}_s\right)\left(B^{t,x}_s-x\right)\right]\right| \\
\leq  c_\eqref{eqn:corollary5_nabla_g}\, \gho\, \frac{  \delta(t,s)^{\frac{1}{2}}} { \Phi_{\beta }(|s-t|^\ep)}\, n^{-\frac{\ep}{2}} \Phi_{\beta }\left( \Big( \frac{T}{n} \Big)^\ep \right).
\end{multline}
\end{lemma}
\smallskip

\begin{proof}The first result comes from \cite[Corollary 5]{BGGL21}. To prove the second result we adapt the proof of \cite[Corollary 5]{BGGL21}.
We put $G(y):=(g(x+y)-g(x))y$. Then
$\e\left[g\left(B^{n,t,x}_s\right)\left(B^{n,t,x}_s-x\right)\right]-\e\left[g\left(B^{t,x}_s\right)\left(B^{t,x}_s-x\right)\right]=\E(G(B^n_s-B^n_t))-\E(G(B_s-B_t))$.
\medskip

In case of $\underline{s}=\underline{t}$ we have 
$\cW_1([G(B^n_s-B^n_t)],[G(B_s-B_t)])\le \E(|G(B_s-B_t)|)\le \gho(s-t)^{(1+\ep)/2}$
by \cite[page 908, line 5]{BGGL21}.
Since $(s-t)^{(1+\ep)/2}=(s-t)^{\frac{1}{2}}(s-t)^{\frac{\ep}{2}}  \leq  \delta(t, s)^{\frac{1}{2}}  ( s - t)^{\frac{\ep}{2}} $ and $s-t \le h$
and the function $r \mapsto r \Phi_{\beta}(r^2)$ is non-decreasing by \cref{Psi-lemma}, we get
\begin{align*}
        \cW_1([G(B^n_s-B^n_t)],[G(B_s-B_t)]) 
& \leq  \gho\, \frac{  \delta(t,s)^{\frac{1}{2}} |s - t|^{\frac{\ep}{2}} \Phi_{\beta} ( |s - t|^{{\ep}} ) }{ \Phi_{\beta }(|s-t|^\ep)} \\
& \leq  T^\frac{\ep}{2}\, \gho\, \frac{  \delta(t,s)^{\frac{1}{2}}}{ \Phi_{\beta }(|s-t|^\ep)}\, n^{-\frac{\ep}{2}}\Phi_{\beta }\left ( \Big( \frac{T}{n} \Big)^\ep \right ).
\end{align*}
\medskip

In case of $\underline{s}>\underline{t}$ we get from the proof of \cite[Corollary 5, page 908, line 18]{BGGL21} that 
$\cW_1([G(B^n_s-B^n_t)],[G(B_s-B_t)])\le c\gho n^{-\frac{\ep}{2}}\delta(t,s)^{1/2}$,
where $c=c(T,\ep)>0$.
Furthermore, we have $h<(s-t)$, which implies that $\Phi_{\beta }(|s-t|^\ep) \le  \Phi_{\beta }\left ( \Big( \frac{T}{n} \Big)^\ep \right )$, and therefore
\[     \cW_1([G(B^n_s-B^n_t)],[G(B_s-B_t)]) 
   \le c \, \gho\, \frac{  \delta(t,s)^{\frac{1}{2}}} { \Phi_{\beta }(|s-t|^\ep)}\, n^{-\frac{\ep}{2}} \Phi_{\beta }
       \left( \Big( \frac{T}{n} \Big)^\ep \right). \]
       
Now we finish the proof by choosing $f(x)=x$ in the Kantorovich-Rubinstein relation 
\begin{equation*}
     \cW_1([X], [X']) = \sup\{ \e\left[f(X)\right]-\e\left[f( X')\right] : \|f\|_\lip \leq 1\}. \qedhere
\end{equation*} 
\end{proof}
\medskip

We modify  \cite[Theorem 11]{BGGL21}  by introducing the function $\Phi_{\beta} $  with the aim to weaken the singularity in time to get square integrability at the expense of the convergence rate.

\begin{lemma}
\label{en:mainPDE}
Suppose that \cref{H:A1} holds and that $\fn$ is defined as in  \eqref{f-zero}.
Let $\beta\in \{0\} \cup ]\frac{1}{2},\infty[$. 
Then for $(t, x) \in [t_o, T[ \times \R$ and $n \geq n_0$ one has
\begin{align}
       |u_\fn(t,x)-U^n_\fn(t,x)|
&\leq  c_\eqref{eqn:1:en:mainPDE}\, (1+|x|)^{\ep} \,  n^{- \alpha \wedge\frac{\ep}{2}} \label{eqn:1:en:mainPDE}, \\
      |\nabla u_\fn(t,x)-  \Delta^n_\fn(t,x)|
&\leq c_\eqref{eqn:2:en:mainPDE}
      \, \frac{(1+|x|)^{\ep} n^{-\alpha \wedge\frac{\ep}{2}} }{|T - t|^{\frac{1 - \ep}{2} } |T - \ut|^{\frac{\ep}{2}} \Phi_{\beta}(|T - \ut |^{\ep} )    }    
      \Phi_{\beta} \left( \left( \frac{T}{n} \right)^{\ep} \right) \label{eqn:2:en:mainPDE},
\end{align}
where 
$c_\eqref{eqn:1:en:mainPDE}, c_\eqref{eqn:2:en:mainPDE}>0$ depend at most on 
$(T,\ep,\alpha,\gho,g(0),\fhot,\fhox,L_f,c_{f_0}^{(2)},M)>0$ and
$c_\eqref{eqn:2:en:mainPDE} >0$ may depend additionally on $\beta$.
\end{lemma}
\smallskip

\begin{proof}
\underline{Estimate \eqref{eqn:1:en:mainPDE}} is shown in \cite[Theorem 11(i)]{BGGL21}
for the setting in \cite{BGGL21}. For the setting of this article
we have to check how our local H\"older continuity  in time given 
in \eqref{eq:the-new-fn-new} influences the  calculations
done in  \cite{BGGL21}. The proof of  \cite[Theorem 11(i)]{BGGL21} contains only one estimate   where the H\"older continuity in time is used. We need to change it  as follows:
Denoting $\Theta^{n,t,x}_s:=(B^{n,t,x}_{s-},Y^{n,t,x}_{s-},Z^{n,t,x}_s) $ we have, by \cref{statement:the-new-f},
\begin{align*} 
 \left| \e\left[\int_{\ut}^T \left({\sf f}_n\left(\os,\Theta^{n,t,x}_s\right)-{\sf f}_n\left(s,\Theta^{n,t,x}_s\right)\right)ds\right] \right| 
&\leq c_{\eqref{eq:the-new-fn-new}} \int_{\ut}^T  \,\frac{ \left|\os -s\right|^{\alpha\wedge \frac{\ep}{2} }}{|T- \os \wedge t_{n-1}|^\frac{1}{2}} ds\\
&\leq c_{\eqref{eq:the-new-fn-new}}  c \, n^{-\alpha\wedge \frac{\ep}{2}}
  \end{align*}
for some $c=c(T)>0$.
\smallskip

\underline{Estimate \eqref{eqn:2:en:mainPDE}}  for the case $\beta=0$  is shown in  \cite[Theorem 11]{BGGL21}. The estimate
concerning the H\"older continuity in time is changed as follows. We use the Cauchy-Schwarz inequality, \cref{statement:the-new-f}
and the fact that $\e| B^{n,t,x}_s-x  |^2=\us-\ut$ to get
for the expression $R$ defined in  \cite[page 920, line 18]{BGGL21} that
  \begin{align} \label{theR}
	|R|&=\left |\int_{\ut+h}^T \e\left[\left({\sf f}_n \left(\os,\Theta^{n,t,x}_s\right)-{\sf f}_n \left(s,\Theta^{n,t,x}_s\right)\right) \frac{B^{n,t,x}_s-x}{\us-\ut} \right] 
	ds \right |  \notag \\
	&\le c_{\eqref{eq:the-new-fn-new}} \left (\frac{T}{n} \right )^{\alpha\wedge \frac{\ep}{2}}   \int_{\ut+h}^T \frac{ 1}{|T- \os \wedge t_{n-1}|^\frac{1}{2}}  \frac{1}{|s-(\ut+h)|^\frac{1}{2}}\, ds \notag \\
	&\le c' c_{\eqref{eq:the-new-fn-new}} \left (\frac{T}{n} \right )^{\alpha\wedge \frac{\ep}{2}},
\end{align}
where the last inequality follows since the integral is bounded.
Indeed, if $t < t_{n-2}$  we have    
  \begin{align*}
 \int_{\ut+h}^{t_{n-2}} \frac{ 1}{|T- \os \wedge t_{n-1}|^\frac{1}{2}}  \frac{1}{|s-(\ut+h)|^\frac{1}{2}}\, ds
& =  \int_{\ut+h}^{t_{n-2}} \frac{ 1}{|T- \os |^\frac{1}{2}}  \frac{1}{|s-(\ut+h)|^\frac{1}{2}}\, ds \\
 &\le \sqrt{2} \int_{\ut+h}^T \frac{ 1}{|T- s |^\frac{1}{2}}  \frac{1}{|s-(\ut+h)|^\frac{1}{2}}\, ds \\
 &=  \sqrt{2} \mathrm{B}\big (\tfrac{1}{2},\tfrac{1}{2} \big ),
\end{align*}
where we used that  $T+s = t_{n-2}+2h+s \ge  2 \os$  for $s< t_{n-2},$ implying  $T-s \le 2(T-\os).$\\
 Since $t <  t_{n-2}$ implies $\ut+h \le t_{n-2}$  we have
 \begin{align*}
 \int_{t_{n-2}}^T \frac{ 1}{|T- \os \wedge t_{n-1}|^\frac{1}{2}} \, \frac{1}{|s-(\ut+h)|^\frac{1}{2}}\, ds
& \le   \int_{t_{n-2}}^T \frac{ 1}{|T-  t_{n-1}|^\frac{1}{2}} \,  \frac{1}{|s-t_{n-2}|^\frac{1}{2} }\, ds = 2 \sqrt{2}.
\end{align*}
If  $t\in [t_{n-1},T]$, then $\ut+h=T$ so it remains to check $t\in [t_{n-2}, t_{n-1}[.$ In this case we have 
 \begin{align*}
 \int_{\ut+h}^T \frac{ 1}{|T- \os \wedge t_{n-1}|^\frac{1}{2}} \, \frac{1}{|s-(\ut+h)|^\frac{1}{2}}\, ds
& =  \int_{t_{n-1}}^T \frac{ 1}{|T-  t_{n-1}|^\frac{1}{2}} \, \frac{1}{|s-t_{n-1}|^\frac{1}{2}}\, ds =2.
\end{align*}
 
To show \eqref{eqn:2:en:mainPDE} for $\beta >\frac{1}{2}$  we only need to make some slight changes in the proof of  \cite[Theorem 11]{BGGL21}
where the statement was shown resulting in a  singularity $\frac{1}{\sqrt{T-t}}$. The needed changes are as follows: 
 As in \cite[Theorem 11]{BGGL21} we  write
\begin{align*} 
  | \nabla u_\fn(t,x)- \Delta^n_\fn(t,x)|\le |\mbox{ $g$ difference }| + |\mbox{ $f$  difference }|.
  \end{align*}
Using the representation \eqref{eq: nabla-u-b}  for $\nabla u_\fn(t,x)$ with $F(s,x)= {\sf f}_n (s,x, u_\fn(s,x),  \nabla u_\fn(s,x))$ and the representation
\begin{align*}
	\Delta^n_\fn(t,x) =& \e\left[g(B^{n,t,x}_T)\frac{B^{n,t,x}_T-x}{T-\ut} \right] \\
&+\e\left[\int_{\ut+h}^T {\sf f}_n\left(\os, B^{n,t,x}_s,U^n_\fn(s,B^{n,t,x}_s),\Delta^n_\fn(s,B^{n,t,x}_{s})\right)\frac{B^{n,t,x}_s-x}{\us-\ut} \,ds\right],
\end{align*}
given in  \cite[equation (33)]{BGGL21} and   recalling that $B_s^{t,x} = x+B_s-B_t$ and $B_s^{n,t,x} = x+B^n_s-B^n_t$, we split the $g$ difference  into
\begin{align*}
 |\mbox{ $g$ difference }| &\le  \left |\e\left[g( B_T^{t,x})(B_T-B_t)\right]\left(\frac{1}{T-t}-\frac{1}{T-\ut}\right) \right | \notag \\
  &\qquad\quad +  \left | \frac{\e\left[g( B_T^{t,x})(B_T-B_t)\right]-\e\left[g( B_T^{n, t,x})(B^n_T-B^n_t)\right] }{T-\ut} \right | \\
&=: S_1 +S_2.
\end{align*}
Then for $S_1$ we have 
\begin{align*}
S_1 & \le \gho\, \frac{|t-\ut|^\frac{\ep}{2}  }{|T-t|^{\frac{1-\ep}{2}}\, |T-\ut|^\frac{\ep}{2} }  
 \le \gho\, \frac{|t-\ut|^\frac{\ep}{2} \Phi_{\beta }(|\ot-\ut|^\ep) }{|T-t|^{\frac{1-\ep}{2}}\, |T-\ut|^\frac{\ep}{2}\,  \Phi_{\beta }(|T-\ut|^\ep)}.
\end{align*}
For the second term we  use  \cref{corollary5} and get
\begin{align*}
      S_2
& \le c_\eqref{eqn:corollary5_nabla_g} \, \gho\,\frac{  \delta(t,T)^{\frac{1}{2}}} { |T-\ut| \,\Phi_{\beta }(|T-\ut|^\ep)}\, n^{-\frac{\ep}{2}}\Phi_{\beta }\left ( \Big( \frac{T}{n} \Big)^\ep \right) \\
&  =  c_\eqref{eqn:corollary5_nabla_g} \, \gho\, \frac{1}{|T-\ut|^{\frac{1}{2}} \, \Phi_{\beta }(|T-\ut|^\ep)  } n^{-\frac{\ep}{2}}\Phi_{\beta }\left ( \Big( \frac{T}{n} \Big)^\ep \right).
\end{align*}
 Then, since $\ot -\ut  \le T/n$, we get
\begin{align*}
 |\mbox{ $g$ difference }|\le \left [ T^\frac{\ep}{2}+ c_\eqref{eqn:corollary5_nabla_g} \right ]  \gho\,  \frac{1}{|T-t|^{\frac{1-\ep}{2}}\, |T-\ut|^\frac{\ep}{2}\, \Phi_{\beta }(|T-\ut|^\ep)   }  n^{-\frac{\ep}{2}}\Phi_{\beta }\left ( \Big( \frac{T}{n} \Big)^\ep \right ).
\end{align*}

Analyzing the proof of \cite[Theorem 11]{BGGL21} it turns out that for 
\begin{equation*} 
	\beta_n(s):=\sup_{x\in\rset}\left\{ \frac{|u_\fn-U^n_\fn|(s,x)}{  (1+|x|)^{\ep}  }\right\}
\sptext{1}{and}{1}
	\gamma_n(s):=\sup_{x\in\rset}\left\{ \frac{|\nabla u_\fn-\Delta^n_\fn|(s,x)}{ (1+|x|)^{\ep} }\right\}
\end{equation*}	
it holds  (see \cite[(43)]{BGGL21} and  \cite[(44)]{BGGL21} together with the estimates for $H(t)$ in this proof)
\begin{equation*}
	\gamma_n(t) \leq C \left( |g \text{ difference}|  +  \int_{(\ut+h)\wedge T}^T  (\beta_n(s)+ \gamma_n(s) ) \frac{ds}{\sqrt{\us-\ut}} +   n^{-\alpha\wedge \frac{\ep}{2}} \right),
\end{equation*}
where the estimate \eqref{theR} for $|R|$ contributes to the last term.
This implies, by the arguments in \cite{BGGL21},
\begin{equation}\label{eq:almostZ}
 \gamma_n(t) \leq c_\eqref{eq:almostZ}
             \left [ \frac{ n^{-\alpha \wedge\frac{\ep}{2}}\Phi_{\beta }\left ( \Big( \frac{T}{n} \Big)^\ep \right )}{|T-t|^{\frac{1-\ep}{2}} \, |T-\ut|^\frac{\ep}{2} \,\Phi_{\beta }(|T-\ut|^\ep)  }
                     + \int_{(\ut+h) \wedge T}^T \gamma_n(s)\, \frac{ds}{\sqrt{\us-\ut}}\right ] \sptext{.75}{for}{.75} t \in [0,T[.
\end{equation}
Observing that
\begin{equation}\label{eqn:upper_bound_Phibeta}
\Phi_{\beta }\left( \Big( \frac{T}{n} \Big)^\ep \right) \leq c_\eqref{eqn:upper_bound_Phibeta} |\log (n+1)|^{\beta}
\end{equation}
for $c_\eqref{eqn:upper_bound_Phibeta} =c_\eqref{eqn:upper_bound_Phibeta} (T, \ep, \beta)>0$,
and inserting \eqref{eq:almostZ} into itself leads to
\begin{align*}
      \gamma_n(t)
&\leq c_\eqref{eq:almostZ} c_\eqref{eqn:upper_bound_Phibeta}
      \frac{|\log (n+1)|^{\beta}}{n^{\alpha \wedge \frac{\ep}{2}}}
      \frac{1}{|T-t|^{\frac{1-\ep}{2}} \, |T-\ut|^\frac{\ep}{2} \, \Phi_{\beta }(|T-\ut|^\ep)} \\
&\quad +  c_\eqref{eq:almostZ}^2 c_\eqref{eqn:upper_bound_Phibeta}  
      \frac{|\log (n+1)|^{\beta}}{n^{\alpha \wedge \frac{\ep}{2}}} \int_{(\ut+h) \wedge T}^T    \frac{1}{ |T-s|^{\frac{1-\ep}{2}}\, |T-\us|^\frac{\ep}{2} \, \Phi_{\beta }(|T-\us|^\ep)}     \frac{ds}{\sqrt{\us-\ut}}  \\
&\quad +  c_\eqref{eq:almostZ}^2 \int_{(\ut+h) \wedge T}^T  \int_{(\us+h) \wedge T}^T \gamma_n(r)\, \frac{dr}{\sqrt{\ur-\us}}    \frac{ds}{\sqrt{\us-\ut}}. 
\end{align*}
Since $T - \us \geq T - s$ and $\Phi_{\beta }((T-\us)^\ep) \geq 1$ by \cref{Psi-lemma}, it holds that 
\begin{equation*}
|T-s|^{\frac{1-\ep}{2}} \, |T-\us|^\frac{\ep}{2} \, \Phi_{\beta }(|T-\us|^\ep) \geq |T-s|^{\frac{1}{2}}
\end{equation*}
and, in particular, 
\begin{equation*}
  \int_{(\ut+h) \wedge T}^T    \frac{1}{|T-s|^{\frac{1-\ep}{2}} \, |T-\us|^\frac{\ep}{2} \, \Phi_{\beta }(|T-\us|^\ep)}     \frac{ds}{\sqrt{\us-\ut}} \leq \int_{(\ut+h) \wedge T}^T    \frac{1}{|T-s|^{\frac{1}{2}}}     \frac{ds}{\sqrt{\us-\ut}}.
\end{equation*}

In  \cite{BGGL21}  (see the end of the proof) it is shown that 
$$ \int_{(\ut+h) \wedge T}^T    \frac{1}{|T-s|^{\frac{1}{2}}}     \frac{ds}{\sqrt{\us-\ut}}  \le  \text{ \tt B} \big (\tfrac{1}{2},\tfrac{1}{2} \big ), $$
 where $ \text{ \tt B}$  denotes the beta function, 
 and that 
 $$  \int_{(\ut+h) \wedge T}^T  \int_{(\us+h) \wedge T}^T \gamma_n(r)\, \frac{dr}{\sqrt{\ur-\us}}    \frac{ds}{\sqrt{\us-\ut}}   \le   \text{ \tt B} \big (\tfrac{1}{2},\tfrac{1}{2} \big )  \int_{(\ut+2h)\wedge T}^T  \gamma_n(r)\, dr.   $$
Then Gronwall's lemma implies the assertion.
\end{proof}
\medskip

\begin{prop} \label{coro_pathwise_estimate_wasserstein}
Suppose that \cref{H:A1} holds and that $\fn$ is defined as in  \eqref{f-zero}.
Let $\beta\in \{0\} \cup ]\frac{1}{2},\infty[$. Then for $(t, x), (t,y) \in [t_o, T[ \times \R$ and
for $n \geq n_0$ one has
\begin{align}
     |u_\fn(t, x) -  U^n_\fn  (t,y) |
&\le c_{\eqref{u-ufn-est}} \left [  |x-y|^\ep + (1+|x|)^{\ep} \,  n^{- \alpha \wedge\frac{ \ep}{2}} \right ], \label{u-ufn-est} \\
     \left | \nabla u_\fn(t,x) - \Delta^n_\fn(t,y)\right|
&\le \frac{c_{\eqref{nabla-u-nabla-ufn-est-pointwise}} }{|T-t|^\frac{1}{2} \Phi_{\beta }(|T-t|^\ep)}\Bigg [  \ps {x-y}   
     + (1  +   |y|)^{\ep} \frac{|\log (n+1)|^{\beta}}{n^{\alpha \wedge \frac{\ep}{2}}} \Bigg ], \label{nabla-u-nabla-ufn-est-pointwise} 
\end{align}
and for any  continuous function $[t,T]\ni s\mapsto {\bf x}(s)$ and any function $[t,T]\ni s\mapsto {\bf x}^n(s)$ that is constant on
the intervals $[t_{k-1},t_k)\cap [t,T]$ it holds
\begin{multline}
     \left ( \int_t^T  \left| \nabla u_\fn(s, \bx(s)) - \Delta^n_\fn(s, \bx^n(s))\right|^2   \od s \right )^\frac{1}{2} \\
 \le c_{\eqref{nabla-u-nabla-ufn-est}} \sup_{t\le s\le T} \Bigg [  \ps  {\bx(s) - \bx^n(s)} 
     +  \Big (1  +  |\bx^n(s)| \Big)^{\ep} \frac{|\log (n+1)|^{\beta}}{n^{\alpha \wedge \frac{\ep}{2}}} \Bigg ], \label{nabla-u-nabla-ufn-est}
\end{multline}
where $c_{\eqref{u-ufn-est}},c_\eqref{nabla-u-nabla-ufn-est-pointwise},c_{\eqref{nabla-u-nabla-ufn-est}}>0$ depend at most on 
$(T,\ep,\alpha,\gho,g(0),\fhot,\fhox,L_f,c^{(2)}_{f_0},M)>0$ and
$c_\eqref{nabla-u-nabla-ufn-est-pointwise},c_{\eqref{nabla-u-nabla-ufn-est}}>0$
may depend additionally on $\beta$.
\end{prop}

\begin{proof}
By \cref{en:regu}  and \cref{en:mainPDE} we have
\begin{align*}
     |u_\fn(t, x) -  U^n_\fn (t,y) |
&\le |u_\fn(t, x) -  u_\fn(t, y)| + |u_\fn(t, y) - U^n_\fn (t, y)|\\
& \le  c_{\eqref{u-estimates}} |x-y|^\ep +   c_\eqref{eqn:1:en:mainPDE} \, (1+|y|)^{\ep} \,  n^{- \alpha \wedge\frac{ \ep}{2}}\\
& \le  c_{\eqref{u-estimates}} |x-y|^\ep +   c_\eqref{eqn:1:en:mainPDE} \, (1+|x| + |x-y|)^{\ep} \,  n^{- \alpha \wedge\frac{ \ep}{2}}\\
& \le (c_{\eqref{u-estimates}}+ c_\eqref{eqn:1:en:mainPDE} ) \,|x-y|^\ep
      +  c_\eqref{eqn:1:en:mainPDE} \, (1+|x|)^{\ep} \,  n^{- \alpha \wedge\frac{ \ep}{2}}.
\end{align*}

\cref{lemma37}\eqref{nabla-estimate} and \cref{en:mainPDE} imply
\begin{align*}
&    \left| \nabla u_\fn(t,x) - \Delta^n_\fn(t,y)\right| \\
&\le \left| \nabla u_\fn(t,x) - \nabla u_\fn(t,y)| +  |\nabla u_\fn(t,y) - \Delta^n_\fn(t,y)\right|\\
&\le c_{\eqref{nabla u-diff-eq}}  \frac{  \ps {x-y}   }{|T-t|^\frac{1}{2} \Phi_{\beta }(|T-t|^\ep)}
    + c_\eqref{eqn:2:en:mainPDE} \frac{(1+|y|)^\ep n^{-\alpha\wedge \frac{\ep}{2}} \Phi_{\beta}
        \left( \left( \frac{T}{n} \right)^{\ep} \right)}{|T - t|^\frac{1 - \ep}{2} \, |T - \ut|^\frac{\ep}{2} \Phi_{\beta}( |T - \ut|^{\ep} )} \\
&\le \frac{c_{\eqref{nabla u-diff-eq}} \vee c_\eqref{eqn:2:en:mainPDE} }{ |T-t|^\frac{1}{2} \Phi_{\beta }( |T-t|^\ep)}
     \left [    \ps {x-y}  
           +   \frac{(1+|y|)^\ep}{n^{\alpha\wedge \frac{\ep}{2}}} \Phi_{\beta}
        \left( \left( \frac{T}{n} \right)^{\ep} \right)  \right ].
\end{align*}
To get \eqref{nabla-u-nabla-ufn-est} we use \eqref{eqn:upper_bound_Phibeta} and observe that for $\beta>\frac{1}{2}$ one has
\[ \int_0^T \frac{\od s}{|T-s|^\frac{1}{2}  \, \Phi_{\beta }(|T-s|^\ep)} < \infty.\]
\end{proof}
\medskip

\begin{prop}\label{rate-singularity}
Assume \cref{H:A1}, $\beta \in \{0\} \cup ] \frac{1}{2},\infty [$, and $n \geq n_0$.
Then there are constants $c_\eqref{wasserstein-Y},c_\eqref{wasserstein-Z}>0$ 
that depend at most on 
$(T,\ep,\alpha,\gho,g(0),\fhot,\fhox,L_f,c^{(2)}_{f_0},M)>0$ and
where $c_\eqref{wasserstein-Z}>0$ may depend additionally on $\beta$, such that
\begin{align}
     \cW_2 \left([Y^{t_o,\xi}_{{\sf f}_n,t}], [Y^{n,t_o,\xi}_{{\sf f}_n,t}]\right)
&\le c_\eqref{wasserstein-Y}\, \, n^{-(\alpha \wedge\frac{\ep}{2})}
        \sptext{7.5}{for}{1}   t\in [t_o,T], \label{wasserstein-Y} \\
     \cW_2\left([Z^{t_o,\xi}_{{\sf f}_n,t}], [Z^{n,t_o,\xi}_{{\sf f}_n,t}]\right)
&\le c_\eqref{wasserstein-Z} \frac{n^{-(\alpha \wedge\frac{\ep}{2})} |\log (n+1)|^\beta}{|T-t|^{\frac{1}{2}}\Phi_{\beta }(|T-t|^\ep)}
     \sptext{2}{for}{1}
     t\in ]t_o,T[, \label{wasserstein-Z}
\end{align}
where we assume that $Z^{t_o,\xi}_{\fn,t}$ is given by   $Z^{t_o,\xi}_{\fn,t} = \nabla u_{\fn}(t,B_t^{t_o,\xi})$  to have a continuous version.
\end{prop}

\begin{proof}
From \cref{coro_pathwise_estimate_wasserstein} we get
\begin{align*}
      \cW_2 \left([Y^{t_o,\xi}_{{\sf f}_n,t}], [Y^{n,t_o,\xi}_{{\sf f}_n,t}]\right)
& =   \cW_2 \left(\left [u_\fn\left (t, B_t^{t_o,\xi}\right )\right ],\left [U^n_\fn  \left (t,B_t^{n,t_o,\xi} \right ) \right ] \right ) \\
&\le c_{\eqref{u-ufn-est}} \left [ \cW_{2\ep} \left ([B_t^{t_o}],[B_t^{n,t_o}] \right )^\ep
     + \left (1+ \left \| B_t^{t_o,\xi} \right \|_{L^{2\ep}}^\ep \right ) \,  n^{- \alpha \wedge\frac{ \ep}{2}} \right ]  \\
&\le c_{\eqref{u-ufn-est}} \left [ \left ( c \sqrt{\frac{T}{n}} \right )^\ep
     + \left ( 1 + T^\frac{\ep}{2}  \right ) \,  n^{- \alpha \wedge\frac{ \ep}{2}} \right ]
\end{align*}
for an absolute constant $c>0$, where we use that  $\cW_p \left ([B_t^{t_o}],[B_t^{n,t_o}] \right ) \le c_p\sqrt{\frac{T}{n}}  $  for $p\ge 1$  (see  \cite[Proposition 4]{BGGL21}). 
In the same way, for $t\not = \ut$ \cref{coro_pathwise_estimate_wasserstein}  and \cref{Psi-lemma} imply
\begin{align*}
&   \cW_2 \left([Z^{t_o,\xi}_{{\sf f}_n,t}], [Z^{n,t_o,\xi}_{{\sf f}_n,t}]\right) \\
& = \cW_2 \left(\left [ \nabla u_\fn\left (t, B_t^{t_o,\xi}\right )\right ],\left [\Delta^n_\fn  \left (t,B_t^{n,t_o,\xi} \right ) \right ] \right ) \\
&\le \frac{c_{\eqref{nabla-u-nabla-ufn-est-pointwise}} }{|T-t|^\frac{1}{2} \, \Phi_{\beta }(|T-t|^\ep)}
      \Bigg [ \sqrt{\E \Ps  {\,|X-Y|^{2\ep} \,} {\frac{1}{2}} ^2 }
     + \left ( 1   + \left \| B_t^{n,t_o} \right \|_{L^{2\ep}}^\ep  \right ) \frac{|\log (n+1)|^{\beta}}{n^{\alpha \wedge \frac{\ep}{2}}} \Bigg ] \\
&\le \frac{c_{\eqref{nabla-u-nabla-ufn-est-pointwise}} }{|T-t|^\frac{1}{2} \Phi_{\beta }(|T-t|^\ep)}
      \Bigg [ \kappa_\beta \Ps {\,\E |X-Y|^{2 \ep}\,}   {\frac{1}{2}}   
     + \left ( 1   + \left \| B_t^{n,t_o} \right \|_{L^{2\ep}}^\ep  \right ) \frac{|\log (n+1)|^{\beta}}{n^{\alpha \wedge \frac{\ep}{2}}} \Bigg ],
\end{align*}
where $(X,Y)$ is a coupling of $(B_t^{t_o},B_t^{n,t_o})$.
Again using  \cite[Proposition 4]{BGGL21}, we continue to
\begin{align*}
&   \cW_2 \left([Z^{t_o,\xi}_{{\sf f}_n,t}], [Z^{n,t_o,\xi}_{{\sf f}_n,t}]\right) \\
&\le \frac{c_{\eqref{nabla-u-nabla-ufn-est-pointwise}} }{|T-t|^\frac{1}{2} \Phi_{\beta }(|T-t|^\ep)}
      \Bigg [ \ps  { c \left ( \frac{T}{n} \right )^\frac{1}{2} }  
     + \left (1   + \left \| B_t^{n,t_o} \right \|_{L^{2\ep}}^\ep \right ) \frac{|\log (n+1)|^{\beta}}{n^{\alpha \wedge \frac{\ep}{2}}} \Bigg ].
\end{align*}
For $t=\ut$ we choose $(\ut -h) \vee t_0 < s_l \uparrow t$ so that 
$Z_{\fn,t}^{n,t_o,\xi}=Z_{\fn,s_l}^{n,t_o,\xi}$. As 
$Z_{\fn,s_l}^{t_o,\xi} \to Z_{\fn,t}^{t_o,\xi}$ in $L^2$ for $l\to \infty$ and
the RHS of \eqref{wasserstein-Z} is continuous in $t$, inequality \eqref{wasserstein-Z} holds for $t=\ut$ as well.
\end{proof}


\subsection{Comparing \texorpdfstring{$(Y^n_{{\sf f}_n}, Z^n_{{\sf f}_n})$}{YnfnZnfn} and  \texorpdfstring{$(Y^n,Z^n)$}{YnZn} }

We proceed with comparing the solutions to \eqref{eq:mainBSDEn-hat} and \eqref{eq:mainBSDEn}. To do so we use a priori estimates:
For this we assume probability spaces $(\Omega_j,\cG_j,\P_j)$ and $(\Omega_{j+1}^n,\cG_{j+1}^n,\P_{j+1}^n)$,
a random variable $\xi:\Omega_j\to \R$ and independent Rademacher variables $\zeta_{j+1},\ldots,\zeta_n:\Omega_{j+1}^n\to \{-1,1\}$.
We take the canonical
extensions $\xi,\zeta_{j+1},\ldots,\zeta_n:\Omega:= \Omega_j\times \Omega_{j+1}^n\to \R$ and define the $\sigma$-algebras
\[ \cF_k:= \sigma(\xi,\zeta_{j+1},\ldots,\zeta_k)
   \sptext{1}{for}{1}
   k=j,\ldots,n. \]

\begin{lemma}[A priori estimates]
\label{statement:a-priori-discrete}
Assume that $f^i:[t_o,T]\times \R\times \R\times\R\to \R$, $i=1,2$, are jointly continuous in $(x,y,z)$ for each fixed
$t\in [t_o,T]$ and such that
\[      \left| f^i(t,x,y,z)-  f^i(t,x,y',z')\right|
   \le  L_f \Big (  \left|y-y'\right|+ \left|z-z'\right|\Big ) \]
for some $L_f>0$, let  $\eta_i:\R\times \R^{n-j}\to \R$ be measurable.
For $n \geq n_0> T \max\{72  L^2_f+1\}$ assume the corresponding solutions
$(( Y^{n,i,\xi}_{t_k})_{k=j}^n,(Z^{n,i,\xi}_{t_k})_{k=j+1}^n)$
to the random walk BSDEs 
\[ Y^{n,i,\xi}_s =  \eta_i(\xi,(\zeta_k)_{k=j+1}^n)
   +  \int_{]s,T]} f^i(r, B^{n,t_o,\xi}_{r-},Y^{n,i,\xi}_{r-}, Z^{n,i,\xi}_r)  \, d\langle \rwB \rangle_r - \int_{]s,T]} Z^{n,i,\xi}_r\, d \rwB_r, \hspace{.3em} s \in [\ut,T]. \]
\begin{enumerate}[{\rm (i)}] 
\item \label{item:1:statement:a-priori-discrete:general}
      For $k=j+1,\ldots,n$ and $\eta_i(\omega):= \eta_i(\xi(\omega),\zeta_{j+1}(\omega),\ldots,\zeta_n(\omega))$ one has that
\begin{align} \label{eqn:item:1:statement:a-priori-discrete:general}
& \E \left [ \sup_{k-1 \le m \le  n} \left| Y^{n,1,\xi}_{t_m} - Y^{n,2,\xi}_{t_m} \right|^2
    +  \int_{]t_{k-1}, T]} \left | Z^{n,1,\xi}_s - Z^{n,2,\xi}_s\right |^2  d\langle \rwB \rangle_s \Bigg | \cF_{k-1}\right ]  \notag \\
&  \le  c_{\eqref{eqn:item:1:statement:a-priori-discrete:general}} \E \Bigg [|\eta_1 -\eta_2|^2  \notag \\
&
 \quad+   \int_{]t_{k-1}, T]}
  |       f^1(s,{B}_{s-}^{n,t_o,\xi},Y^{n,1,\xi}_{s-},Z^{n,1,\xi}_s)
         - f^2(s,{B}_{s-}^{n,t_o,\xi},Y^{n,1,\xi}_{s-},Z^{n,1,\xi}_s)|^2  d\langle \rwB \rangle_s\Bigg | \cF_{k-1} \Bigg  ]
\end{align}
with $c_{\eqref{eqn:item:1:statement:a-priori-discrete:general}}=c_{\eqref{eqn:item:1:statement:a-priori-discrete:general}}(T,L_f)>0$,
where we work with generalized conditional expectations for non-negative random variables.

\item \label{item:2:statement:a-priori-discrete:special-1}
        Specializing item \eqref{item:1:statement:a-priori-discrete:general} to $\eta_1= \eta_2$,
        \begin{align*}
       f^1(s,x,y,z)  & :=f(s\wedge t_{n-1},x,y,z,  \mu^1_s, \nu^1_s),\\
       f^2(s, x,y,z) & :=f(s\wedge t_{n-1},x,y,z,\mu^2_s, \nu^2_s),
       \end{align*}
        with $f$ satisfying \eqref{eq:prop_f}, and   $\mu^i, \nu^i : [t_o,T] \to \cP_2(\R)$ we get
         \begin{multline} \label{eqn:item:2:statement:a-priori-discrete:special-1}
         \left| Y^{n,1,\xi}_{t_{k-1}} - Y^{n,2,\xi}_{t_{k-1}}\right|^2
          +  \E \left [ \int_{]t_{k-1}, T[} \left| Z^{n,1,\xi}_s -  Z^{n,2,\xi}_s \right|^2  d\langle \rwB \rangle_s \Bigg | \cF_{k-1} \right ] \\
          \le  c_{\eqref{eqn:item:2:statement:a-priori-discrete:special-1}}
               \int_{ ]t_{k-1}, T]} \left [ \cW_2^2(\mu^1_{s}, \mu^2_s)  d\langle \rwB \rangle_s +
                \cW_2^2( \nu^1_{s}, \nu^2_{s}) \right ] d\langle \rwB \rangle_s
        \end{multline}
       with $c_{\eqref{eqn:item:2:statement:a-priori-discrete:special-1}}=c_{\eqref{eqn:item:2:statement:a-priori-discrete:special-1}}(T,L_f)>0$.

\item \label{item:3:statement:a-priori-discrete:special-2}
      Specializing item \eqref{item:2:statement:a-priori-discrete:special-1} further to $\eta_1= \eta_2$,
      \begin{align*}
      f^1(s,x,y,z)  & :=f( s\wedge t_{n-1},x,y,z,[Y^{n,1,\xi}_s],[Z^{n,1,\xi}_s]),\\
      f^2(s, x,y,z) & :=f( s\wedge t_{n-1},x,y,z, \mu^2_s,\nu^2_s),
      \end{align*} 
      we get
      \begin{multline}\label{eqn:item:3:statement:a-priori-discrete:special-2}
       \E \left [  \left| Y^{n,1,\xi}_{t_{k-1}} - Y^{n,2,\xi}_{t_{k-1}}\right|^2
           + \int_{]t_{k-1}, T[}  \left| Z^{n,1,\xi}_s -  Z^{n,2,\xi}_s \right|^2  d\langle \rwB \rangle_s \right ] \\
      \le c_{\eqref{eqn:item:3:statement:a-priori-discrete:special-2}}
           \int_{ ]t_{k-1}, T]}  \left [ \cW_2^2([Y^{n,2,\xi}_{s}], \mu^2_s) d\langle \rwB \rangle_s +
                                         \cW_2^{2}([Z^{n,2,\xi}_{s}], \nu^2_s ) \right] d\langle \rwB \rangle_s
     \end{multline}
     with $c_{\eqref{eqn:item:3:statement:a-priori-discrete:special-2}}=c_{\eqref{eqn:item:3:statement:a-priori-discrete:special-2}}(T,L_f) >0$.
\end{enumerate}
\end{lemma}
\medskip

Item \eqref{item:3:statement:a-priori-discrete:special-2} of \cref{statement:a-priori-discrete}
describes in a general way how the solution processes of the mean field BSDE differ if in the 
mean field component not the correct distributions are inserted. Note that in the case 
$[Y_s^{n,2,\xi}]= \mu_s^2$ and  $[Z_s^{n,2,\xi}]= \nu_s^2$, i.e. if the correct distributions are inserted, we get
that the RHS equals zero.

\begin{proof}[Proof of \cref{statement:a-priori-discrete}]
Before we start with the actual proof we choose any
\[ n_0=n_0(T,L_f) >  (72  L^2_f+1)T
   \sptext{1}{and set}{1}
   \delta := \frac{1}{36 L_f^2}. \]
For $h = \frac{T}{n}$ with $n\ge n_0$ this implies
\begin{align}
A & :=  3L_f^2(\delta +h)
   \le 3L_f^2 \left (\frac{1}{36 L_f^2} +\frac{T}{n_0} \right ) =:A'
    < \frac{1}{8}, \label{eqn:condition_A} \\
B & := h \left ( \frac{1}{\delta} + A\right )
    < \frac{T}{n_0} \left ( 36 L_f^2  + \frac{1}{8}\right )
   \le \frac{5}{8}. \label{eqn:condition_B}
\end{align}

\eqref{item:1:statement:a-priori-discrete:general}
We first assume that $\xi\equiv x \in \R$.
We use the short notation
\[ \eta:=  |\eta_1-\eta_2|
   \sptext{1}{and}{1}
   {\tY}_{t_k}:= Y^{n,1,x}_{t_k} -  Y^{n,2,x}_{t_k} \sptext{1}{for}{1} k\ge j, \]
and for $k>j$,
\begin{align*}
{f^i}(t_k)&:= f^i(t_k,B_{t_{k-1}}^{n,t_o,x},Y^{n,i,x}_{t_{k-1}},Z^{n,i,x}_{t_k}),\\
\Delta Y^{n,i}_{t_k} &:= Y^{n,i,x}_{t_k} - Y^{n,i,x}_{t_{k-1}}   =- h   f^i(t_k) + Z^{n,i,x}_{t_k}  \Delta \rwB_{t_k},\\
                     {\tZ}_{t_k}&:= Z^{n,1,x}_{t_k} -  Z^{n,2,x}_{t_k},\\
            f_{t_{k}}&:= f^1(t_{k})-f^2(t_k), \\
         f^=_{t_{k}} &:= f^1(t_k,B_{t_{k-1}}^{n,t_0,x},Y^{n,1,x}_{t_{k-1}},Z^{n,1,x}_{t_k})-f^2(t_k,B_{t_{k-1}}^{n,t_o,x},Y^{n,1,x}_{t_{k-1}},Z^{n,1,x}_{t_k}).
  \end{align*}
Furthermore, we abbreviate
$\E_k [\cdot] := \E[\cdot | \cF_k ]$. Then, for $m>j$,
\begin{align*}
     |\tY_{t_m}|^2 -  |\tY_{t_{m-1}}|^2
& =  |\tY_{t_{m-1}} + \Delta \tY_{t_m}|^2 - |\tY_{t_{m-1}}|^2 \\
& =  |\Delta  \tY_{t_m}|^2 + 2 \tY_{t_{m-1}}  \Delta  \tY_{t_m}  \\
&\ge -h^2 f^2_{t_m} + \frac{1}{2} \tZ_{t_m}^2 ( \Delta \rwB_{t_m})^2 + 2 \tY_{t_{m-1}} \Delta  \tY_{t_m},
\end{align*}
  where we used that $|\Delta  \tY_{t_m}|^2 = (- h   f_{t_m}+ \tZ_{t_m}  \Delta \rwB_{t_m})^2 \ge -h^2 f_{t_m}^2 + \frac{1}{2} \tZ_{t_m}^2 ( \Delta \rwB_{t_m})^2. $
  From this we derive exploiting relations like 
\[ \sum_{m=k}^n    \tY_{t_{m-1}}  \Delta  \tY_{t_m}  = \int_{]t_{k-1},T]} \tY_{s-}  d \tY_s, \]
with the extension $Y_s=Y_{\us}$, for $k>j$  the estimates
   \begin{align*} 
   |\tY_{t_{k-1}}|^2 &=   \eta^2 -  \sum_{m=k}^n [|\tY_{t_m}|^2  - |\tY_{t_{m-1}}|^2]  \\
          &\le    \eta^2 - 2 \int_{]t_{k-1},T]} \tY_{s-} d \tY_s   +    \sum_{m=k}^n  (h^2    | f_{t_m}|^2 - \frac{1}{2} \tZ_{t_m}^2   (\Delta \rwB_{t_m})^2) \\
        &\le    \eta^2    + 2  \int_{]t_{k-1},T]}  |\tY_{s-}|| f_s|  d\langle \rwB \rangle_s  - 2 \int_{]t_{k-1},T]} \tY_{s-}  \tZ_s d\rwB_s 
        & \quad   \\
        &\quad  +  \int_{]t_{k-1},T]}   h | f_s|^2 d\langle \rwB \rangle_s  - \frac{1}{2} \int_{]t_{k-1},T]}   \tZ_s^2  d\langle \rwB \rangle_s.
 \end{align*}
Assume $\delta>0$.
By Young's inequality we have   $2|\tY_{s-}| | f_s| \le \frac{|\tY_{s-}|^2  }{\delta} + \delta| f_s|^2  $ which implies
 \begin{align*}
  & |\tY_{t_{k-1}}|^2   +  \frac{1}{2} \int_{]t_{k-1},T]}    \tZ_s^2  d\langle \rwB \rangle_s  \notag \\
   & \le   \eta^2  +  \frac{1}{\delta}  \int_{]t_{k-1},T]}  |\tY_{s-}|^2  d\langle \rwB \rangle_s + (\delta+h) \int_{]t_{k-1},T]}  | f_s|^2 d\langle \rwB \rangle_s \notag\\
  &\quad  - 2 \int_{]t_{k-1},T]} \tY_{s-}  \tZ_s d\rwB_s.   
   \end{align*}
   From  $|f_{t_k}|^2\le  3  L_f^2   |\tZ_{t_{k}}|^2 +   3  | f^=_{t_{k}}|^2+  3  L_f^2|\tY_{t_{k-1}}|^2$ and for
\begin{equation}\label{eqn:condition_on_delta_1}
  A  = 3L_f^2(\delta +h)<\frac{1}{8}
\end{equation}
we obtain 
\begin{align*} 
  & |\tY_{t_{k-1}}|^2   +   \Big (\frac{1}{2}-A \Big)  \int_{]t_{k-1},T]}   \tZ_s^2  d\langle \rwB \rangle_s    \\
  & \le   \eta^2  +  \left ( \frac{1}{\delta} + A  \right ) \int_{]t_{k-1},T]}  |\tY_{s-}|^2  d\langle \rwB \rangle_s  + 3  (\delta +h)   \int_{]t_{k-1},T]}  | f^=_s |^2 d \langle \rwB \rangle_s \\
  &\quad  - 2 \int_{]t_{k-1},T]} \tY_{s-}   \tZ_s d\rwB_s.
  \end{align*}
As we also have
\begin{equation}\label{eqn:condition_on_delta_2}
 B = h \left ( \frac{1}{\delta} + A\right ) < \frac{5}{8}
\end{equation}
and
\[ \int_{]t_{k-1},T]}  |\tY_{s-}|^2  d\langle \rwB \rangle_s = |\tY_{t_{k-1}}|^2 h +  \int_{]t_{k-1},T[}  |\tY_s|^2  d\langle \rwB \rangle_s \]
we continue to
\begin{align} \label{with_Gronwall_term}
&    \frac{3}{8} |\tY_{t_{k-1}}|^2   + \frac{3}{8} \int_{]t_{k-1},T]}   \tZ_s^2  d\langle \rwB \rangle_s  \notag  \\
&\le (1-B) |\tY_{t_{k-1}}|^2   +  \Big (\frac{1}{2}-A \Big)  \int_{]t_{k-1},T]}   \tZ_s^2  d\langle \rwB \rangle_s  \notag  \\
&\le \eta^2  +  \left ( \frac{1}{\delta} + A \right ) \int_{]t_{k-1},T[}  |\tY_{s}|^2  d\langle \rwB \rangle_s
      +  3 (\delta + h)   \int_{]t_{k-1},T]}  | f^=_s |^2 d \langle \rwB \rangle_s \notag\\
&\quad  - 2 \int_{]t_{k-1},T]} \tY_{s-}   \tZ_s d\rwB_s \notag \\
&\le \eta^2  +  \left ( \frac{1}{\delta} + \frac{1}{8} \right ) \int_{]t_{k-1},T[}  |\tY_{s}|^2  d\langle \rwB \rangle_s
      +  3 (\delta + h)   \int_{]t_{k-1},T]}  | f^=_s |^2 d \langle \rwB \rangle_s \notag\\
&\quad  - 2 \int_{]t_{k-1},T]} \tY_{s-}   \tZ_s d\rwB_s.
\end{align}
Taking the conditional expectation we get for $j < l \le k \le n$ that
\begin{multline}  \label{eqn:upper_bound_Yk-1_Zk-1-T}
      \frac{3}{8}  \E_{l-1} |\tY_{t_{k-1}}|^2   +   \frac{3}{8}  \E_{l-1} \int_{]t_{k-1},T]}   \tZ_s^2  d\langle \rwB \rangle_s  \\
 \le  \E_{l-1} \eta^2  +  \left ( \frac{1}{\delta} +  \frac{1}{8}  \right )
\int_{]t_{k-1},T[} \E_{l-1} |\tY_{s}|^2  d\langle \rwB \rangle_s \\ +  3   (\delta + h)   \E_{l-1} \int_{]t_{l-1},T]}  | f^=_s |^2 d \langle \rwB \rangle_s
\end{multline}
where we enlarged the RHS by extending $]t_{k-1},T]$ to $]t_{l-1},T]$.
Gronwall's \cref{gronwall1} implies
\begin{align} \label{eqn:upper_bound_Yk-1}
\E_{l-1} |\tY_{t_{k-1}}|^2  \le c_\eqref{eqn:upper_bound_Yk-1} \left [ \E_{l-1} \eta^2  +  \E_{l-1} \int_{]t_{l-1},T]}  | f^=_s |^2 d \langle \rwB \rangle_s \right ]
    \end{align}
with $c_\eqref{eqn:upper_bound_Yk-1}=c_\eqref{eqn:upper_bound_Yk-1}(T,L_f)>0$  for $k=l,...,n$.
Denoting
\[ K: = \eta^2 +  \int_{]t_{k-1},T]} | f^=_s |^2 d \langle \rwB \rangle_s, \]
the relations \eqref{eqn:upper_bound_Yk-1_Zk-1-T} and \eqref{eqn:upper_bound_Yk-1}
imply for $l=k$
\begin{align} \label{eqn:Z_upper_bound_K}
\e_{k-1} \,  \int_{]t_{k-1},T]}    \tZ_s^2  d\langle \rwB \rangle_s  
\le c_\eqref{eqn:Z_upper_bound_K}  \,\, \e_{k-1} \,K
\end{align}
with  $c_\eqref{eqn:Z_upper_bound_K}=c_\eqref{eqn:Z_upper_bound_K}(T,L_f)>0$.
Coming again  back to \eqref{with_Gronwall_term},
similarly as before we apply Gronwall's inequality, where we additionally estimate the stochastic integral by its supremum, and get
  \begin{align}\label{eqn:upper_bound:Y-star}
      \tY_{*}^2
& \le c_\eqref{eqn:upper_bound:Y-star}\Bigg ( K   + \sup_{k\le m\le n} \left|  \int_{]t_{m-1},T]}
      \tY_{s-}   \tZ_s d\rwB_s \right| \Bigg )
    \end{align}
with $\tY_{*}:= \sup_{m=k-1,\ldots,n} |\tY_{t_m}|$
and $c_\eqref{eqn:upper_bound:Y-star}=c_\eqref{eqn:upper_bound:Y-star}(T,L_f)>0$.
By the conditional BDG inequality for $p=1$ with constant $\beta_1>0$  and Young's inequality we have for $\lambda >0$ that
\begin{align} \label{gamma-estimate}
      \frac{1}{\beta_1} \e_{k-1} \,  \sup_{k \le m \le n} \left  | \int_{]t_{m-1},T]} \tY_{s-}  \tZ_sd\rwB_s \right  | \notag
& \le \e_{k-1} \, \left  |  \int_{]t_{k-1},T]} |\tY_{s-}|^{2}  \tZ_s^2  d \langle \rwB \rangle_s \right |^\frac{1}{2}  \notag \\
& \le \lambda   \e_{k-1} \, \tY_{*}^{2}  +  \frac{1}{\lambda}  \e_{k-1} \,  \left  |  \int_{]t_{k-1},T]}   \tZ_s^2  d \langle \rwB \rangle_s \right |.
\end{align}
By \eqref{eqn:upper_bound:Y-star}, \eqref{gamma-estimate}, and \eqref{eqn:Z_upper_bound_K} we conclude 
\begin{align*} 
     \e_{k-1} \, \tY_{*}^2
& \le c_\eqref{eqn:upper_bound:Y-star}  \e_{k-1} \Bigg ( K   + \sup_{k\le m\le n} \left|  \int_{]t_{m-1},T]}
      \tY_{s-}   \tZ_s d\rwB_s \right| \Bigg ) \\
& \le c_\eqref{eqn:upper_bound:Y-star}   \Bigg (\e_{k-1} K   + \beta_1 \lambda   \e_{k-1} \, \tY_{*}^{2}  +  \frac{\beta_1}{\lambda}  \e_{k-1} \,  \left  |  \int_{]t_{k-1},T]}   \tZ_s^2  d \langle \rwB \rangle_s \right | \Bigg ) \\
& \le c_\eqref{eqn:upper_bound:Y-star}  \Bigg (\e_{k-1} K   + \beta_1 \lambda   \e_{k-1} \, \tY_{*}^{2}  +  \frac{\beta_1}{\lambda}  
c_\eqref{eqn:Z_upper_bound_K}  \,\, \e_{k-1} \,K \Bigg ).
\end{align*}
For small enough  $\lambda>0$ we conclude 
\begin{equation}\label{eqn:end_of_proof_item:1:statement:a-priori-discrete:general}  
    \e_{k-1} \,\tY_{*}^2  
\le c_\eqref{eqn:end_of_proof_item:1:statement:a-priori-discrete:general} \e_{k-1} \, K 
\end{equation}
for 
$  c_\eqref{eqn:end_of_proof_item:1:statement:a-priori-discrete:general}
  =c_\eqref{eqn:end_of_proof_item:1:statement:a-priori-discrete:general}(T,L_f)>0$.
Now \eqref{eqn:item:1:statement:a-priori-discrete:general}
for $\xi=x$ follows from \eqref{eqn:Z_upper_bound_K} and from the last inequality.
So we proved the inequality
\begin{align*}
& \E_{k-1} \left [ \sup_{k-1 \le m \le  n} \left| Y^{n,1,x}_{t_m} - Y^{n,2,x}_{t_m} \right|^2
    +  \int_{]t_{k-1}, T]} \left | Z^{n,1,x}_s - Z^{n,2,x}_s\right |^2  d\langle \rwB \rangle_s\right ]  \notag \\
&  \le  c_{\eqref{eqn:item:1:statement:a-priori-discrete:general}} \E_{k-1} \Bigg [|\eta_1(x,\zeta) -\eta_2(x,\zeta)|^2  \notag \\
&
 \quad+   \int_{]t_{k-1}, T]}
   |   f^1 (s,{B}_{s-}^{n,t_o,x},Y^{n,1,x}_{s-},Z^{n,1,x}_s)
         - f^2(s,{B}_{s-}^{n,t_o,x},Y^{n,1,x}_{s-},Z^{n,1,x}_s)|^2 d\langle \rwB \rangle_s\Bigg  ]
\end{align*}
when the Brownian motion starts in $x\in \R$. By replacing $x$ by $\xi(\omega)$ and using the
measurable dependence of the solution processes on the initial value $x\in \R$, we get the desired result.
\medskip

\eqref{item:2:statement:a-priori-discrete:special-1}
By  \eqref{eq:prop_f} we observe that 
\begin{align*}
     |f^1(s,x,y,z)-f^2(s, x,y,z)|^2
& =  |   f( s\wedge t_{n-1}, x,y,z,  \mu^1_s, \nu^1_s)
       - f(s\wedge t_{n-1},x,y,z,  \mu^2_s, \nu^2_s)|^2 \\
&\le 2 L_f^2 \Big [ 
           \cW_2^2( \mu^1_s, \mu^2_s) +
           \cW_2^2(\nu^1_s,\nu^2_s) \Big ].
\end{align*}
The assertion follows by \eqref{item:1:statement:a-priori-discrete:general}.
\medskip

\eqref{item:3:statement:a-priori-discrete:special-2}
Similar as above we use
$\tY_{t_k}^\xi:= Y_{t_k}^{n,1,\xi}-Y_{t_k}^{n,2,\xi}$,
$\tZ_{t_k}^\xi:= Z_{t_k}^{n,1,\xi}-Z_{t_k}^{n,2,\xi}$, and
$f_{t_k}^{=,\xi}:= f^1(t_k,B_{t_{k-1}}^{n,t_o,\xi},Y_{t_{k-1}}^{n,t_o,\xi},Z_{t_k}^{n,t_o,\xi})-
                   f^2(t_k,B_{t_{k-1}}^{n,t_o,\xi},Y_{t_{k-1}}^{n,t_o,\xi},Z_{t_k}^{n,t_o,\xi})$
and get
\begin{align} \label{fdifeq_estimate}
      \left| f_s^{=,\xi} \right| &\leq L_f \left [\cW_2([Y_s^{n,1,\xi}], \mu_s^2) + \cW_2 ([Z_s^{n, 1,\xi}], \nu_s^2 \right ] \notag \\
&\leq L_f \left [ \left\Vert \tY_s^{\xi} \right\Vert _{L^2} + \left\Vert \tZ_s^{\xi} \right\Vert _{L^2}
                  + \cW_2([Y_s^{n,2,\xi}], \mu_s^2) + \cW_2 ([Z_s^{n, 2,\xi}], \nu_s^2) \right ],
\end{align}
so that, with $f_s^{\cW,\xi}:= \cW_2([Y_s^{n,2,\xi}], \mu_s^2) + \cW_2 ([Z_s^{n, 2,\xi}], \nu_s^2)$,
  we have
  \begin{align*}
   \left| f_s^{=,\xi} \right|^2
   &\leq   3   L_f^2 ( \| \tZ_s^{\xi}\|^2_{L^2}  +  \|\tY_s^{\xi}\|_{L^2}^2  +   |f_s^{\cW,\xi}|^2).
   \end{align*}
If we take the expectation  in \eqref{with_Gronwall_term}  
we get using \eqref{eqn:condition_A}
\begin{align}
&      \frac{3}{8} \e |\tY_{t_{k-1}}^{\xi}|^2   +   \left ( \frac{1}{2} - A' \right )  \e \int_{]t_{k-1},T]}    |\tZ_s^{\xi}|^2  d\langle \rwB \rangle_s \notag \\
& \le  \e \eta^2  +   \left ( \frac{1}{\delta} + \frac{1}{8} \right ) \e \int_{]t_{k-1},T[} \e |\tY_s^{\xi}|^2  d \langle \rwB \rangle_s
       +    \frac{3}{8}  \int_{]t_{k-1},T]}  \e |\tY_s^{\xi}|^2  d \langle \rwB \rangle_s \notag \\
& \quad  +  3 A' \int_{]t_{k-1},T]}  \e |\tZ_s^{\xi}|^2  d \langle \rwB \rangle_s +  \frac{3}{8}  \e \int_{]t_{k-1},T]}   | f_s^{\cW,\xi} |^2 d \langle \rwB \rangle_s.  \notag
\end{align}
Since $\eta =0$, denoting $K' := \int_{]t_{k-1},T]}   | f_s^{\cW,\xi} |^2 d \langle \rwB \rangle_s$, we derive
\begin{align}
&    \frac{3}{8} \e |\tY_{t_{k-1}}^{\xi}|^2   +    \Big (\frac{1}{2}- 4A' \Big)  \e \int_{]t_{k-1},T]}    |\tZ_s^{\xi}|^2  d\langle \rwB \rangle_s \notag \\
&\le c_\eqref{eqn:before_Growall_EYt} \left [ \int_{]t_{k-1},T[} \e |\tY_s^{\xi}|^2  d \langle \rwB \rangle_s +  \e K' \right ] \label{eqn:before_Growall_EYt}
\end{align}
with $c_\eqref{eqn:before_Growall_EYt}=c_\eqref{eqn:before_Growall_EYt}(T,L_f)>0$.
By  Gronwall's inequality and using $\frac{1}{2}-4A'>0$,
\begin{equation}\label{eqn_Yk-1:after_Grownall}
  \e |\tY_{t_{k-1}}^{\xi}|^2 \le c_\eqref{eqn_Yk-1:after_Grownall} \e K'
\end{equation}
for $c_\eqref{eqn_Yk-1:after_Grownall}=c_\eqref{eqn_Yk-1:after_Grownall}(T,L_f)>0$.
Therefore we get 
\begin{equation}\label{eqn_Zk-1:after_Grownall}
 \e \int_{]t_{k-1},T]}    |\tZ_s^{\xi}|^2  d\langle \rwB \rangle_s \le  c_\eqref{eqn_Zk-1:after_Grownall} \e K'
\end{equation}
for $c_\eqref{eqn_Zk-1:after_Grownall}=c_\eqref{eqn_Zk-1:after_Grownall}(T,L_f)>0$.
\end{proof}
\bigskip

We prove the following statement in \cref{sec:proof:statement:disrcrete-rate}:

\begin{prop}\label{statement:disrcrete-rate}
Suppose \cref{H:A1},  $\beta>\frac{1}{2}$, and $k\in \{j+1,\ldots,n\}$. Then for the solutions to \eqref{eq:mainBSDEn-hat} and \eqref{eq:mainBSDEn} it holds  for any $n \geq n_0$ that
\begin{multline}\label{eqn:item:1:statement:disrcrete-rate:start-in-0}
 \E \left [ \left|    Y^{n,t_o,\xi}_{t_{k-1}} - Y^{n,t_o,\xi}_{{{\sf f}_n},t_{k-1}} \right|^2  +  \int_{]t_{k-1}, T[} \left|Z^{n,t_o,\xi}_s-  Z^{n,t_o,\xi}_{{{\sf f}_n},s} \right|^2
         d\langle \rwB \rangle_s\right ] \\
 \le   c_\eqref{eqn:item:1:statement:disrcrete-rate:start-in-0}^2 \frac{|\log (n+1)|^{2\beta}}{n^{2(\alpha \wedge \frac{\ep}{2})}},
 \end{multline}
where 
$c_\eqref{eqn:item:1:statement:disrcrete-rate:start-in-0}=
 c_\eqref{eqn:item:1:statement:disrcrete-rate:start-in-0}(\Theta,f_0,M)>0$.
For $ U^n, U^n_\fn, \Delta^n$, and $ \Delta^n_\fn$ as given in
\eqref{eqn:def:Deltanfn}, \eqref{eq:mainPDEnfn}, and \eqref{eqn:def:mainPDEnf_Deltanf} we get
\begin{multline} \label{eqn:item:2:statement:disrcrete-rate:start-in-x}
       \left|    U^n(t_{k-1},x)  -   U^n_\fn(t_{k-1},x)  \right|^2
     + \E \int_{]t_{k-1}, T[} \left| Z^{n,t_{k-1},x}_{{\sf f}_n,s} -   Z^{n,t_{k-1},x}_s \right|^2  d\langle \rwB \rangle_s   \\
  \le  c^2_{\eqref{eqn:item:2:statement:disrcrete-rate:start-in-x}} \frac{|\log (n+1)|^{2\beta}}{n^{2(\alpha \wedge \frac{\ep}{2})}},
\end{multline} 
where
$c_{\eqref{eqn:item:2:statement:disrcrete-rate:start-in-x}}=
 c_{\eqref{eqn:item:2:statement:disrcrete-rate:start-in-x}}(\Theta,f_0,M)>0$
and for $k=n$ the integral is set to zero.
Moreover, there is an absolute constant $c_\eqref{eqn:statement:Lexp_bound_difference_gradient_discrete_case}>0$ such that 
\begin{equation}\label{eqn:statement:Lexp_bound_difference_gradient_discrete_case}
\left |  \int_{]t_{k-1}, T[} \left| Z^{n,t_{k-1},x}_{{\sf f}_n,s} -   Z^{n,t_{k-1},x}_s \right|^2  d\langle \rwB \rangle_s 
\right |_{M_{\varphi_1}}^\frac{1}{2} 
  \le c_\eqref{eqn:statement:Lexp_bound_difference_gradient_discrete_case} c_{\eqref{eqn:item:2:statement:disrcrete-rate:start-in-x}}
       \frac{|\log (n+1)|^{\beta}}{n^{\alpha \wedge   \frac{\ep}{2}}}.
\end{equation}
\end{prop}


\section{Proof of \texorpdfstring{
\cref{statement:main_result_local_version},
\cref{coro_final_pathwise_estimate}}{main theorems},
and  \texorpdfstring{\cref{statement:meanfield_PDE}}{coro} }
\label{subsection_pathwise}

\begin{proof}[Proof of \cref{statement:main_result_local_version}]
\underline{Estimate \eqref{eqn:1:statement:main_result_local_version}:}
We have by Propositions 
\ref{prop_tilde_difference},
\ref{coro_pathwise_estimate_wasserstein}, and
\ref{statement:disrcrete-rate} that
\begin{align*}    |u(t,x) - U^n(t,y)| &\le  |u(t,x)-u_{{\sf f}_n}(t,x)| +  |u_\fn(t, x) -  U^n_\fn  (t,y) |  +    |U^n_\fn  (t,y) -  U^n  (t,y) | \\
&\le  c_{\eqref{cont_tilde_a_priori}} \, n^{-{((\alpha \wedge \frac{\ep}{2}) +\frac{1}{2})}} 
      + c_{\eqref{u-ufn-est}} \left [  |x-y|^\ep + (1+|x|)^{\ep} \,  n^{- \alpha \wedge\frac{ \ep}{2}} \right ] \\
&\quad + c_{\eqref{eqn:item:2:statement:disrcrete-rate:start-in-x}} \frac{|\log (n+1)|^{\beta}}{n^{\alpha \wedge   \frac{\ep}{2}}}.
\end{align*}
\underline{Estimate \eqref{eqn:2:statement:main_result_local_version}:}
We recall that $ \| \cdot \|_{L^{\exp(\gamma)}}  = | \cdot |_{M^0_{\varphi_\frac{1}{\gamma}}}$. Again by
\cref{prop_tilde_difference},
\cref{coro_pathwise_estimate_wasserstein}, and
\cref{statement:disrcrete-rate} we get
\begin{align}
&\hspace{-1em}    \frac{1}{3} \left \| \int_t^T  \Big| \nabla u(s, B_s^{t,x}) - \Delta^n(s, B^{n,\underline{t},y}_s)\Big|^2  \od s  \right \| \notag \\
&\le \left \| \int_t^T  \Big| \nabla u(s, B_s^{t,x}) - \nabla u_\fn (s, B_s^{t,x}) \Big|^2  \od s \right \| \notag \\
& \quad +  \left \| \int_t^T  \Big| \nabla u_\fn (s, B_s^{t,x}) - \Delta^n_\fn(s, B^{n,\underline{t},y}_s)\Big|^2  \od s \right \| 
        +  \left \| \int_t^T  \Big| \Delta^n_\fn (s, B^{n,\underline{t},y}_s) - \Delta^n(s, B^{n,\underline{t},y}_s)\Big|^2  \od s \right \| \notag \\
&\le \kappa \left [ c_{\eqref{eqn:statement:Lexp_bound_difference_gradient_continuous_case}} c_{\eqref{cont_tilde_a_priori}} \, 
     n^{-{((\alpha \wedge \frac{\ep}{2}) + \frac{1}{2})}} \right ]^2 \notag \\
&\quad +  2 c_{\eqref{nabla-u-nabla-ufn-est}}^2
     \left [ \left \| \sup_{t\le s\le T} \ps {B_s^{t,x}-B^{n,\underline{t},y}_s}^2 \right \|  
   +  \left \| \sup_{t\le s\le T}  (1+|B^{n,\underline{t},y}_s|)^{2\ep} \right 
      \| \left | \frac{|\log (n+1)|^{\beta}}{n^{\alpha \wedge \frac{\ep}{2}}} \right |^2 \right ] \label{eqn:part_with_supBnts} \\
&\quad  +\kappa \left [ c_\eqref{eqn:statement:Lexp_bound_difference_gradient_discrete_case} 
      c_{\eqref{eqn:item:2:statement:disrcrete-rate:start-in-x}}
      \frac{|\log (n+1)|^{\beta}}{n^{\alpha \wedge   \frac{\ep}{2}}} \right ]^2, \notag
\end{align}
where for the last estimate we use
\begin{align*}
     \int_t^T  \Big| \Delta^n_\fn (s, B^{n,\underline{t},y}_s) - \Delta^n(s, B^{n,\underline{t},y}_s)\Big|^2  \od s 
&\le \int_{]\ut,T]}  \Big| Z_{\fn,s}^{n,\underline{t},y} - Z_s^{n,\underline{t},y} \Big|^2  \od s \\
& =  \int_{]\ut,T]}  \Big| Z_{\fn,s}^{n,\underline{t},y} - Z_s^{n,\underline{t},y} \Big|^2  \od \langle B^{n} \rangle_s \\
& =  \int_{]\ut,T[}  \Big| Z_{\fn,s}^{n,\underline{t},y} - Z_s^{n,\underline{t},y} \Big|^2  \od \langle B^{n} \rangle_s.
\end{align*}
We get the above  inequality  since by  \eqref{eqn:fn:Delta->Z-I} we may replace the integrand on the LHS by
$ \Big| Z_{\fn,s}^{n,\underline{t},y} - Z_s^{n,\underline{t},y} \Big|^2 $. Furthermore, the  discrete scheme
defined in \cref{sec:discrete_scheme} implies   $Z_{\fn,T}^{n,\underline{t},y} = Z_T^{n,\underline{t},y}$ so that
we may integrate over the open interval.
\medskip

Next we use that from
Doob's maximal inequality, Hoeffding's inequality \eqref{eqn:hoeffding}, and
\eqref{eqn:statement:LPhi_vs_Lp} it follows for   $B^{n,\underline{t}}_s:=B^{n,\underline{t},0}_s$  and $p\in [1,\infty[$ that 
\begin{equation}\label{eqn:Lp_of_Btnx}
      \left \| \sup_{t\le s\le T}  |B^{n,\underline{t}}_s|^{\ep} \right \|_{L^p}
   \le \left \| \sup_{\ut  \le s\le T}  |B^{n,\underline{t}}_s| \right \|_{L^{2p}}^\ep
   \le 2^\ep \left \| B^{n,\underline{t}}_T \right \|_{L^{2p}}^\ep
   \le 2^\ep \left [  c_\eqref{eqn:hoeffding} c_\eqref{eqn:statement:LPhi_vs_Lp} \sqrt{2p} \sqrt{n} \sqrt{\frac{T}{n}} \right ]^\ep
     =: c p^\frac{\ep}{2}.
\end{equation}
With this we estimate the term
$\left \| \sup_{t\le s\le T}  (1+|B^{n,\underline{t},y}_s|)^{2\ep} \right \|$ in \eqref{eqn:part_with_supBnts} by
\begin{align*}
     \left \| \sup_{t\le s\le T}  (1+|B^{n,\underline{t},y}_s|)^{2\ep} \right \|
&\le \kappa \left | \sup_{t\le s\le T}  (1+|B^{n,\underline{t},y}_s|)^{2\ep} \right |_{M_{\varphi_1}^0} \\
&\le \kappa \, c_\eqref{eqn:statement:LPhi_vs_Lp}(1) \,\,
     \sup_{p\in [1,\infty[}\frac{\left \| \sup_{t\le s\le T}  (1+|B^{n,\underline{t},y}_s|)^{2\ep} \right \|_{L_p}}{p} \\
& =  \kappa \, c_\eqref{eqn:statement:LPhi_vs_Lp}(1) \,\,
     \left [ \sup_{p\in [1,\infty[}\frac{\left \| \sup_{t\le s\le T}  (1+|B^{n,\underline{t},y}_s|)^{\ep} \right \|_{L_{2p}}}{\sqrt{p}} \right ]^2\\
&\le \kappa \, c_\eqref{eqn:statement:LPhi_vs_Lp}(1) \,\, \left [ (1+|y|)^\ep +
     \sup_{p\in [1,\infty[}\frac{\left \| \sup_{t\le s\le T}  |B^{n,\underline{t}}_s|^{\ep} \right \|_{L_{2p}}}{\sqrt{p}}
     \right ]^2 \\
&\le \kappa \, c_\eqref{eqn:statement:LPhi_vs_Lp}(1) \,\, \left [ (1+|y|)^\ep +
     c 2^\frac{\ep}{2}\right ]^2.
\end{align*}
\end{proof}
\smallskip

For the proof of \cref{coro_final_pathwise_estimate} we use the following relation
which we verify in \cref{sec:proof:statement:estimate_exponential_norm_with_Psi}.

\begin{lemma}\label{statement:estimate_exponential_norm_with_Psi}
Let $\beta > \frac{1}{2}$.
Assume a random variable $\eta:\Omega \to \R$ such that $\sup_{p\in [1,\infty[} \frac{\|\eta\|_{L^p}}{p} \le c$ for some $c>0$. 
Then for $r_\beta=e^{-2\beta}$ one has
\begin{equation}\label{eqn:statement:estimate_exponential_norm_with_Psi}
 \sup_{p\in [1,\infty[} \frac{\left \| \ps {\eta}  \right \|_{L^p}}{ p^{\ep+\beta}}
   \le c_\eqref{eqn:statement:estimate_exponential_norm_with_Psi} c^\ep \left| \log \left (\frac{1}{c^{2\ep}\wedge \sqrt{r_\beta}} \right ) \right |^{\beta},
\end{equation}
where $c_\eqref{eqn:statement:estimate_exponential_norm_with_Psi}=c_\eqref{eqn:statement:estimate_exponential_norm_with_Psi}(\ep,\beta)>0$.
\end{lemma}
\bigskip

\begin{proof}[Proof of \cref{coro_final_pathwise_estimate}]
We assume a coupling   $(\cW, \cW^n)$ of the Brownian motion with the random walk 
such that \eqref{eqn:th_random_walk_uniform_convergence} is satisfied and define 
$B_t:= W_t-W_{t_o}$ and $B_t^n:= W_t^n - W_{\uto}^n$.
By \eqref{eqn:statement:LPhi_vs_Lp} we have for all $p\in   [2,\infty[$ and $q:=p/2 \in [1,\infty[$ that
\[ \|\cdot \|_{L^q} \le \kappa_q |\cdot |_{M_{\varphi_1}^0}
   \sptext{1}{with}{1} \kappa_q:= c_\eqref{eqn:statement:LPhi_vs_Lp} \, q =  c_\eqref{eqn:statement:LPhi_vs_Lp} \frac{p}{2}. \]  
Then \cref{statement:main_result_local_version} with $\|\cdot\|:= \|\cdot \|_{L^q}$ and
$\kappa:=\kappa_q$ implies that
\begin{align*}
& \hspace*{-4em}\frac{1}{c_\eqref{eqn:2:statement:main_result_local_version}}
      \left \| \left ( \int_t^T  \Big| \nabla u(s, B_s^{t,x}) - \Delta^n(s, B^{n,\underline{t},y}_s)\Big|^2  \od s \right )^\frac{1}{2} \right \|_{L^p}  \\
& =  \frac{1}{c_\eqref{eqn:2:statement:main_result_local_version}}
      \left \| \int_t^T  \Big| \nabla u(s, B_s^{t,x}) - \Delta^n(s, B^{n,\underline{t},y}_s)\Big|^2  \od s \right \|^\frac{1}{2}_{L^q}  \\
&\le  \left \| \sup_{t\le s\le T} \ps {B_s^{t,x}-B^{n,\underline{t},y}_s}^2  \right \|_{L^q}^\frac{1}{2}
    + \left ( 1 + \sqrt{c_\eqref{eqn:statement:LPhi_vs_Lp} q} \, (1+|y|)^\ep \right )
        \frac{|\log (n+1)|^{\beta}}{n^{\alpha \wedge \frac{\ep}{2}}}.
\end{align*}
We observe
\[ \sup_{p\in [1,\infty[} \frac{\left \|  \sup_{t\le s\le T} |B_s^{t,x}-B^{n,\underline{t},y}_s| \right \|_{L^p}}{p}
   \le c_n:= |x-y|+2 c_\eqref{eqn:th_random_walk_uniform_convergence}  \frac{\log(n+1)}{\sqrt{n}} \]
by our assumption  on the coupling. 
We deduce by \cref{statement:estimate_exponential_norm_with_Psi} that
\begin{align}
& \hspace*{-1em}\frac{1}{c_\eqref{eqn:2:statement:main_result_local_version}}
      \left \| \left ( \int_t^T  \Big| \nabla u(s, B_s^{t,x}) - \Delta^n(s, B^{n,\underline{t},y}_s)\Big|^2  \od s \right )^\frac{1}{2} \right \|_{L^p}  \\
&\le   \left \| \sup_{t\le s\le T} \ps {B_s^{t,x}-B^{n,\underline{t},y}_s}  \right \|_{L^p}
     + \left ( 1 + \sqrt{c_\eqref{eqn:statement:LPhi_vs_Lp} q} \,\, (1+|y|)^\ep \right )
        \frac{|\log (n+1)|^{\beta}}{n^{\alpha \wedge \frac{\ep}{2}}} \notag \\
&\le c_\eqref{eqn:statement:estimate_exponential_norm_with_Psi}  p^{\ep+\beta} c_n^\ep 
     \left| \log \Big (\frac{1}{c_n^{2\ep}\wedge \sqrt{r_\beta}} \Big ) \right |^{\beta}  
     + \left ( 1 + \sqrt{c_\eqref{eqn:statement:LPhi_vs_Lp} \frac{p}{2}} \,\, (1+|y|)^\varepsilon \right )
     \frac{|\log (n+1)|^{\beta}}{n^{\alpha \wedge \frac{\ep}{2}}}. \label{eqn:proof:Z:Thm12}
\end{align}
Now using \eqref{eqn:statement:LPhi_vs_Lp}  the inequality for the $Z$-term follows from \eqref{eqn:proof:Z:Thm12}.
Similarly, \cref{statement:main_result_local_version} implies that
\begin{align}\label{eqn:uUn}
&     \left \| \sup_{t\le s \le T} \left | u(s,B_s^{t,x})-U^n(s,B_s^{n,t,y})\right | \right \|_{L^p} \notag \\
&\le  c_\eqref{eqn:1:statement:main_result_local_version}
      \left \| \sup_{t\le s \le T} \left [ |B_s^{t,x}-B_s^{n,t,y}|^\ep + (1+|B_s^{t,x}|)^\ep \frac{|\log (n+1)|^\beta}{n^{\alpha \wedge \frac{\ep}{2}}} \right ]
      \right \|_{L^p} \notag \\
&\le  c_\eqref{eqn:1:statement:main_result_local_version} \Bigg [
          \left \| \sup_{t\le s \le T} |B_s^{t,x}-B_s^{n,t,y}|^\ep  \right \|_{L^p}
        + \left ( 1 + \left \| \sup_{t\le s \le T} |B_s^{t,x}|^\ep  \right \|_{L^p} \right ) \frac{|\log (n+1)|^\beta}{n^{\alpha \wedge \frac{\ep}{2}}}
          \Bigg ].
\end{align}
Finally, we use \cref{th_random_walk_uniform_convergence} which gives
\begin{align}\label{eqn:BtxBnty}
     \left \| \sup_{t\le s \le T} |B_s^{t,x}-B_s^{n,t,y}|^\ep  \right \|_{L^p}
&\le |x-y|^\ep+ \left \| \sup_{t\le s \le T} |B_s^{t,0}-B_s^{n,t,0}| \right \|_{L^p}^\ep \notag \\
&\le |x-y|^\ep+ (2 c_\eqref{eqn:th_random_walk_uniform_convergence})^\ep  \left | \, p \, \frac{\log (n+1)}{n^{\frac{1}{2}}} \right |^\ep,
\end{align}
and the fact that
\begin{equation}\label{eqn:Btx}
       1+\left \| \sup_{t\le s \le T} |B_s^{t,x}|^\ep  \right \|_{L^p}
   \le 1+|x|^\ep + c' p^\frac{\ep}{2}
   \sptext{1}{for}{1}
   p\in [1,\infty[
\end{equation}
and some $c'=c'(T,\ep)>0$ proved on the lines of \eqref{eqn:Lp_of_Btnx}.
Combining \eqref{eqn:uUn}, \eqref{eqn:BtxBnty}, and \eqref{eqn:Btx}   with \eqref{eqn:statement:LPhi_vs_Lp}  yields to the estimate of the $Y$-term.
\end{proof}

\begin{proof}[Proof of \cref{statement:meanfield_PDE}]
We assume a coupling $(\xi,\gamma)$ with $\mu=[\xi]$ and $\nu=[\gamma]$ and choose
independently from this coupling (by taking product spaces) a coupling
$(\cB, \cB^n)   \in \Gamma(B, B^n)$ as in \cref{coro_final_pathwise_estimate}.
With this we get
\begin{align*}
     \left |u(t_o,x,\mu) - U^n(\uto,y,\nu) \right |
&\le \left |u(t_o,x,\mu) - u(t_o,y,\nu) \right | + \left |u(t_o,y,\nu) - U^n(\uto,y,\nu) \right | \\
&\le \left |u(t_o,x,\mu) - u(t_o,y,\nu) \right |
     + c_\eqref{eqn:coro_final_pathwise_estimate:Y} (1+ |y|)^\ep \frac{|\log (n+1)|^{\ep \vee \beta} } {n^{\alpha \wedge\frac{\ep}{2}}}
\end{align*}
and
\begin{align*}
&    \left \| \sqrt{\int_{t_o}^T  \Big| \nabla u(s, B_s^{t_o,x},\mu) - \Delta^n(s, B^{n,\uto,y}_s,\nu)\Big|^2  \od s}  \right \|_{L^2} \\
&\le \left \| \sqrt{\int_{t_o}^T  \Big| \nabla u(s, B_s^{t_o,x},\mu) -  \nabla u(s, B_s^{t_o,y},\nu)\Big|^2  \od s}  \right \|_{L^2} \\
&\hspace*{17em} + \left \| \sqrt{\int_{t_o}^T  \Big| \nabla u(s, B_s^{t_o,y},\nu) - \Delta^n(s, B^{n,\uto,y}_s,\nu)\Big|^2  \od s}  \right \|_{L^2} \\
&\le \left \| \sqrt{\int_{t_o}^T  \Big| \nabla u(s, B_s^{t_o,x},\mu) - \nabla u(s, B_s^{t_o,y},\nu)\Big|^2  \od s}  \right \|_{L^2}
    +  c'_\eqref{eqn:coro_final_pathwise_estimate:Z}
      (1 + |y|)^\ep \frac{|\log (n+1)|^{\ep+\beta} } {n^{\alpha \wedge\frac{\ep}{2}}},
\end{align*}
where we use in the last step that
$\|\cdot\|_{L^2} \le c_{\beta+\ep} \|\cdot\|_{L^{\exp(\beta+\ep)}}$, and that for
$x=y$ we have $D_n=c_0 \log(n+1)/\sqrt{n}$ and
\[ D_n^\ep \left | 1\vee \log \frac{1}{D_n}\right |^\beta
   \le c'_{\ep,\beta,T} \frac{|\log(n+1)|^{\ep+\beta}}{n^\frac{\ep}{2}}.\]
Moreover, we have that
\begin{multline*}
 \left |u(t_o,x,\mu) - u(t_o,y,\nu) \right | +
 \left \| \sqrt{\int_{t_o}^T  \Big| \nabla u(s, B_s^{t_o,x},\mu) -   \nabla u(s, B_s^{t_o,y},\nu)\Big|^2  \od s}  \right \|_{L^2} \\
\le c \left [ \| |\xi-\gamma|^\ep \|_{L^2} + |x-y|^\ep \right ]
\end{multline*}
for some $c=c(T,|g|_\ep, |f|_{\ep;x},\ep,L_f)>0$. This estimate can be obtained in two steps: We estimate the LHS by
a multiple of $\left [|x-y|^\ep + \int_{t_o}^T [\cW_2^2(Y_s^{t_o,\mu},Y_s^{t_o,\nu})+ \cW_2^2(Z_s^{t_o,\mu},Z_s^{t_o,\nu}) ] \od s \right ]$
by a standard a-priori estimate (one can use \cite[Lemma 5.26]{GeissYlinen}). In the next step
the integral term can be estimated by an a-priori estimate in the mean field context, for example \cite[Theorem A.2 for $t=0$,
taking the expected value instead of the conditional expectation]{Li17}.
\end{proof}


\section{Numerical results} 
\label{algorithm}

In this Section we explain how to numerically compute the solution of \eqref{eq:mainBSDEn-hat}. We rewrite this equation in the following way

\begin{align}\label{num_scheme_0}
  \left\{\begin{array}{l}
    Y^{n,t,x}_{t_n}=g(B^{n,t,x}_T) \mbox{ and for all } k \in \{n-1,\cdots,0\} :\\
    Z^{n,t,x}_{t_{k+1}}=h^{-1/2}\E(Y^{n,t,x}_{t_{k+1}}\zeta_{k+1}|\cF^n_{t_k})\\
    Y^{n,t,x}_{t_{k}}=\E(Y^{n,t,x}_{t_{k+1}}|\cF^n_{t_k})+h\E(f(t_{k+1}\wedge t_{n-1},B^{n,t,x}_{t_k},Y^{n,t,x}_{t_{k+1}},Z^{n,t,x}_{t_{k+1}},[Y^{n,t_o,\xi}_{t_{k+1}}],[Z^{n,t_o,\xi}_{t_{k+1}}])|\cF^n_{t_k})
  \end{array}\right.
\end{align}
where $t\ge t_o$.
To compute $(Y^{n,t,x},Z^{n,t,x})$ we first need to compute $(Y^{n,t_o,\xi},Z^{n,t_o,\xi})$ which solves \eqref{z-eqn}-\eqref{y-eqn}. 

\begin{enumerate}
\item {\bf Computation of $(Y^{n,t_o,\xi},Z^{n,t_o,\xi})$}. For the sake of clarity we recall \eqref{z-eqn}-\eqref{y-eqn}:

\begin{align}\label{num_scheme}
  \left\{\begin{array}{ll}
    Y^{n,t_o,\xi}_{t_n}&=g(B^{n,t_o,\xi}_T) \mbox{ and for all } k \in \{n-1,\cdots,0\} :\\
    Z^{n,t_o,\xi}_{t_{k+1}}&=h^{-1/2}\E(Y^{n,t_o,\xi}_{t_{k+1}}\zeta_{k+1}|\cF^n_{t_k})\\
    Y^{n,t_o,\xi}_{t_{k}}&=\E(Y^{n,t_o,\xi}_{t_{k+1}}|\cF^n_{t_k})\\
     & \quad +h\E(f(t_{k+1}\wedge t_{n-1},B^{n,t_o,\xi}_{t_k},Y^{n,t_o,\xi}_{t_{k+1}},Z^{n,t_o,\xi}_{t_{k+1}},[Y^{n,t_o,\xi}_{t_{k+1}}],[Z^{n,t_o,\xi}_{t_{k+1}}])|\cF^n_{t_k}).\\
  \end{array}\right.
\end{align}

 \begin{remark}
The scheme \eqref{num_scheme} is an implicit one, meaning that   we use $Y^{n,t_o,\xi}_{t_{k+1}}$ in the driver while   \eqref{y-eqn}  contains  $Y^{n,t_o,\xi}_{t_{k}}.$ 
 The scheme \eqref{num_scheme}  is easier to solve numerically than \eqref{y-eqn}. The error between \eqref{num_scheme} and \eqref{z-eqn}-\eqref{y-eqn} is negligible compared to the other terms of the error, we neglect it in the following.
\end{remark}

\begin{remark}\label{rem_num_scheme}
  Since $Y^{n,t_o,\xi}_{t_{k+1}}$ is a function $F(\zeta_1,\cdots, \zeta_{k+1},\xi)$ we have $$\E(Y^{n,t_o,\xi}_{t_{k+1}}\zeta_{k+1}|\cF^n_{t_k})= \frac{1}{2}(F(\zeta_1,\cdots,\zeta_k,1,\xi)-F(\zeta_1,\cdots,\zeta_k,-1,\xi)).$$ Moreover $Z^{n,t,x}_{t_{k+1}}$ is $\cF^n_{t_k}$-measurable. \smallskip
\end{remark}

Compared to standard BSDEs the main difficulty lies in the presence of the law of $Y^{n,t_o,\xi},Z^{n,t_o,\xi}$ in the driver.
Since we solve \eqref{num_scheme} by using a random walk on a binomial tree, the law of $Y^{n,t_o,\xi},Z^{n,t_o,\xi}$ is a discrete law given that $\xi$ is discrete.
We are able to compute this law exactly since \eqref{num_scheme} computes the value of $Y^{n,t_o,\xi}$ and $Z^{n,t_o,\xi}$ at each
node of the tree in a backward manner. The law of $Y^{n,t_o,\xi}_{t_{k+1}}$ and $Z^{n,t_o,\xi}_{t_{k+1}}$ is completely known when we compute $Y^{n,t_o,\xi}_{t_{k}}$.  \smallskip

Let us describe the procedure in the special case $(t_o,\xi)=(0,\delta_0)$ (we refer to the following remark for the general case). We first build the binomial tree representing $B^{n,0,0}$. The law of $B^{n,0,0}_{t_k}$ is known for each $k$ in $\{0,\cdots,n\}$ so we know the law of $Y^{n,0,0}_{t_n}$. Using  \cref{rem_num_scheme} enables to compute $Z^{n,0,0}_{t_{n}}$ at each node of step $n$. Since we know the law of $Y^{n,0,0}_{t_n}$ and $Z^{n,0,0}_{t_{n}}$, the last equation of \eqref{num_scheme} enables to compute $Y^{n,0,0}_{t_{n-1}}$ at each node of step $n-1$ and then we iterate. We refer to the following section for a detailed algorithm how to solve \eqref{num_scheme} starting from $(t_o,\xi)=(0,\delta_0)$.

\begin{remark}
  When $(t_o,\xi)\neq (0,\delta_0)$, we use that $B^{n,t_o,\xi}_s=\xi+B^n_s-B^n_{t_o}$ and we solve the BSDE on the sub-tree of depth $\frac{T-\underline{t_o}}{h}$ whose root is $B^n_{t_o}$.
\end{remark}

\item {\bf Computation of $(Y^{n,t,x},Z^{n,t,x})$}. Once the law of $Y^{n,t_o,\xi}$ and $Z^{n,t_o,\xi}$ is known we are able to compute the law of $(Y^{n,t,x},Z^{n,t,x})$ given by \eqref{num_scheme_0}. To do so we follow the procedure described above to compute $Y^{n,t_o,\xi}$ and $Z^{n,t_o,\xi}$. Since $B^{n,t,x}_s=x+B^n_s-B^n_t$ we consider the sub-tree of depth $\frac{T-\underline{t}}{h}$ whose root is $B^n_{t}$. Then we solve \eqref{num_scheme_0} in a backward manner on this sub-tree, by plugging in at each step the known law of $Y^{n,t_o,\xi}$ and $Z^{n,t_o,\xi}$ in the driver.

\end{enumerate}

\subsection{Algorithm}
In the following we assume $(t_o,\xi)=(0,0)$ for the sake of clarity. We only describe the procedure how to compute $Y^{n,t_o,\xi}$ and $Z^{n,t_o,\xi}$ (i.e. we solve \eqref{num_scheme}) since this is the key issue. 
\begin{enumerate}
  \item We first construct the binomial tree representing $B^n$. We store it into an upper triangular matrix of size $(n+1)^2$ whose component $B[i,j]$ represents the value of $B^n_{t_j}$ when $(\zeta_1,\cdots,\zeta_j)$ contains $i$ times the value $(-1)$ and $(j-i)$ times the value $(+1)$. 
  \item We compute the probability matrix $P$ associated to the matrix $B$. $P[i,j]$ represents the probability that $B^n_{t_j}$ contains $i$ times the value $(-1)$ and $(j-i)$ times the value $(+1)$, i.e. $P[i,j]=\frac{(^j_i)}{2^j}$.
  \item We implement the scheme \eqref{num_scheme}. To do this, we compute the values of $Y^n$ and $Z^n$ and the whole tree. As for $B$, the matrices $Y$ and $Z$ are upper triangular ones and $Y[i,j]$ represents the value of $Y^n_{t_j}$ when $(\zeta_1,\cdots,\zeta_j)$ contains $i$ times the value $(-1)$ and $(j-i)$ times the value $(+1)$ (the same definition holds for the matrix $Z$, except that $Z$ is of size $n^2$ because $Z^n_{t_0}$ is not used).
  \begin{algorithm}[H]
    \caption{Computation of $Y^n$ and $Z^n$ on the tree}\label{algo1}
    \begin{algorithmic} 
      \State  Compute the matrices $B$ and $P$ as explained above\\
      \State {\it Computation of $Y^n$ and $Z^n$ at the last time step:}
      \State $Y[:,n]=g(B[:,n])$
      \For {$i=0 : (n-1)$}
      \State Compute $Z[i,n-1]=\frac{1}{2\sqrt{h}}(Y[i,n]-Y[i+1,n])$
      \EndFor \\
      \State {\it Computation of $Y^n$ and $Z^n$ in a backward way:}
      \For {$j=n-1 : 0$}
      \State Store the law of $Y^n$ at time $t_{j+1}$ in $law_Y$
      \For {$i=0 : j$}
      \State Compute $Z[i,j]=\frac{1}{2\sqrt{h}}(Y[i,j+1]-Y[i+1,j+1])$
      \EndFor
      \State Store the law of $Z^n$ at time $t_j$ in $law_Z$
      \For {$i=0 : j$}
      \State Compute $\E(Y^n_{t_{j+1}}|\cF^n_{t_j})$ ($y=\frac{1}{2}(Y[i,j+1]+Y[i+1,j+1])$)
      \State Compute $Y[i,j]=y+\frac{h}{2}(f(j*h,B[i,j],Y[i,j+1],Z[i,j],law_y,law_z)+f(j*h,B[i,j],Y[i+1,j+1],Z[i,j],law_y,law_z))$
      \EndFor
      \EndFor 
      
      \State return (Y,Z)
    \end{algorithmic}
    \end{algorithm}
\end{enumerate}

\begin{remark}
  To store the law of $Y^n_{t_{j}}$ we construct a matrix $law_Y$ of size $(j+1)\times 2$. The first column contains the $(j+1)$ values that can take $Y^n_{t_{j}}$. The second column contains the associated probability : $law_Y[:,0]=Y[0:j,j]$ and $law_Y[:,1]=P[0:j,j]$.
\end{remark}

  
\subsection{Numerical experiments}
In the following numerical examples we assume $(t,x)=(t_o,\xi)=(0,0)$. This boils down to solve \eqref{num_scheme}. We consider cases where the exact solution $(Y^{0,0},Z^{0,0})$ is known. Since we are not able to illustrate numerically the result of Theorem \ref{coro_final_pathwise_estimate} (and more specifically to compute numerically the Wasserstein distance), we aim to emphasize Theorem \ref{statement:main_result_local_version}. To do so, in \eqref{eqn:1:statement:main_result_local_version} we replace $x$ by $B^{0,0}_t$ and $y$ by $B^{n,0,0}_t$ in the case where $B^{n,0,0}_t$ is built by using the Skorokhod embedding (see \cite[Section 3]{GLL20a}).  So $Y^{0,0}$ and $Y^{n,0,0}$ are defined on the same probability space and by taking expectation of the square of the difference, we get
\begin{align}\label{eq:num_error}
\E|Y^{0,0}_t-Y^{n,0,0}_t|^2&\le c_\eqref{eqn:1:statement:main_result_local_version} \left [ \E|B^{0,0}_t-B^{n,0,0}_t|^{2\ep} + C\frac{|\log (n+1)|^{2\beta}}{n^{2\alpha \wedge \ep}} \right ]\notag \\
& \le c_\eqref{eqn:1:statement:main_result_local_version} \left [ \frac{C}{n^{\frac{\ep}{2}}} + C\frac{|\log (n+1)|^{2 \beta}}{n^{2 \alpha \wedge \ep}} \right ]
\end{align}
where the second inequality is obtained by using \cite[Lemma A.1]{GLL20a}. \bigskip

We study numerically the convergence in $n$ of $\E|Y^{n,0,0}_{t}-Y^{0,0}_{t}|^2$, where $(Y^{0,0},Z^{0,0})$ solves \eqref{eq:mainBSDE} and $(Y^{n,0,0},Z^{n,0,0})$ solves \eqref{eq:mainBSDEn-hat}. To do so, we approximate the error $\E|Y^{n,0,0}_{t_k}-Y^{0,0}_{t_k}|^2$ by Monte-Carlo:

\begin{align*}
  \E|Y^{n,0,0}_{t}-Y^{0,0}_{t}|^2\sim \frac{1}{M}\sum_{m=1}^M  |Y^{n,m,0,0}_{t}-Y^{m,0,0}_{t}|^2:=E_Y.
\end{align*}


\subsubsection{Case 1 : \texorpdfstring{$g(x)=x$ and $f(t,x,y,z,\mu,\nu)=y+\E(Y_t)+\E(Z_t)$}{case1}}
In such a case the exact solution is known (see \cite[Section 6]{ABGL22}). $(Y_t,Z_t)=(e^{T-t}(B_t+e^{T-t}-1),e^{T-t})$.
Figure \ref{fig1} represents the evolution of $\log(E_Y)$ w.r.t. $\log(n)$. A linear regression gives that the slope of the figure of the r.h.s. is $-0.52$ (Since $g$ is Lipschitz, \eqref{eq:num_error} ensures that the slope is less than $-0.5$). We run M=20000 Monte Carlo computations and compute the value of $E_Y$ at time $t=0.3$.

\begin{figure}[H]\centering
  \includegraphics[width=7cm]{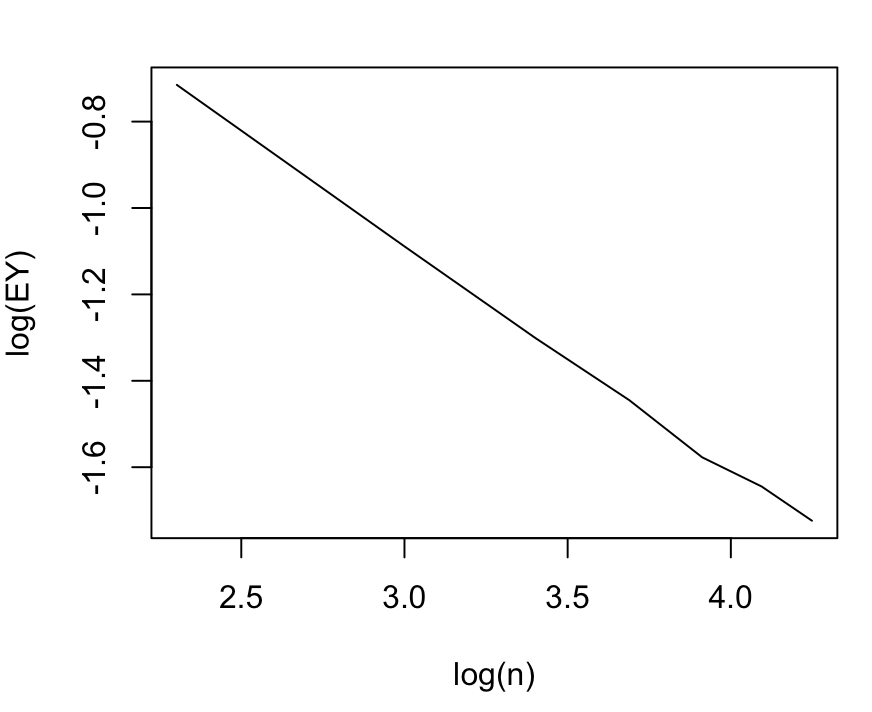}
  \caption{Evolution of $\log(E_Y)$ w.r.t. $\log(n)$}
  \label{fig1}
  \end{figure}


\subsubsection{Case 2: \texorpdfstring{$g(x)=x^2\wedge K$ and $f(t,x,y,z,\mu,\nu)=y+\E(Y_t)+\E(Z_t)$}{case2}}
In such a case the exact solution is known (see \cite[Section 6]{ABGL22}). Since $\E(Z_t)=0$, we get $Y_t=e^{T-t}(\E_t(B^2_T\wedge K)+\E(B^2_T\wedge K)(e^{T-t}-1))$.

Figure \ref{fig3} represents the evolution of $\log(E_Y)$ w.r.t. $\log(n)$.  A linear regression gives that the slope of the figure of the r.h.s. is $-0.63$ (Since $g$ is Lipschitz, \eqref{eq:num_error} ensures that the slope is less than $-0.5$).  We run M=50000 Monte Carlo computations and compute the value of $E_Y$ at time $t=0.3$ and $K=5$.

\begin{figure}[H]
\centering
  \includegraphics[width=7cm]{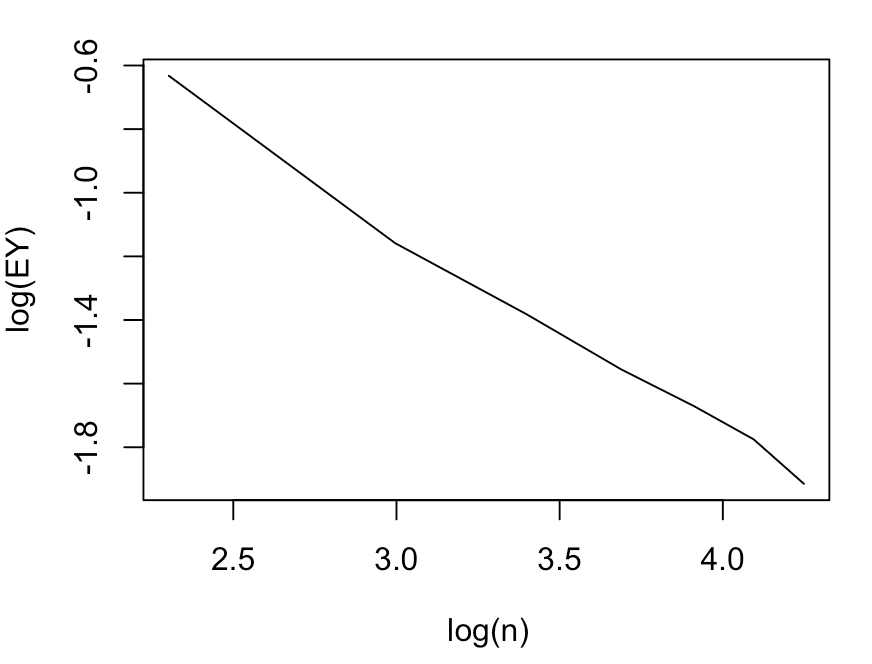}
  \caption{Evolution of $\log(E_Y)$ w.r.t. $\log(n)$}
  \label{fig3}
  \end{figure}


\section{Appendix}
\subsection{Gronwall lemmas}


\begin{lemma} \label{gronwall1}
Let  $T>0$, $\theta \in  [0,1[$ and $a, b\ge 0,$  and assume that $f: [0,T]\to [0,\infty[ $ is a bounded
Borel measurable function. Then  
\begin{enumerate}[{\rm (i)}]
\item  \label{Gronwall-continuous} 
      $f(t)\le a+b \int_t^T \frac{f(s)}{(T-s)^\theta}ds$ for $t\in [0,T]$ implies  that
      $ f(t) \le a e^{ b \int_t^T (T-s)^{-\theta} ds}$ for $t\in [0,T]$,
\item \label{Gronwall-discrete}
$
f(t_{k-1}) \le a+b \int_{]t_{k-1}, T[} \frac{f(r)}{(T-r)^\theta}d\langle \rwB \rangle_r   
$
 for $k=1,...,n$ implies that 
\[ f(t_{k-1})  \le a \, e^{ b \int_{t_k}^T (T-s)^{-\theta} ds}
   \sptext{1}{for}{1} k=1,...,n. \]   

\end{enumerate}
\end{lemma}
\begin{proof}  We follow \cite[Theorem 5.1,  Appendix]{MR838085} saying that (after a change of variables)
for a  finite Borel measure $\mu$ on $[0,T]$  
the relation
\begin{equation}\label{eqn:gronwall:EK}
   f(t) \le a + \int_{]t,T[} f(s) d\mu(s) 
   \sptext{.5}{for}{.5} 
   t\in [0,T] 
   \sptext{1}{implies}{1}
   f(t) \le a e^{\mu(]t,T[)}.
\end{equation}
Item \eqref{Gronwall-continuous} follows from setting $\mu(]t,T[) := b \int_t^T (T-s)^{-\theta} ds.$  
\medskip 

Item \eqref{Gronwall-discrete}: We put $\mu(]t,T[) := b \int_{]t, T[} \frac{1}{(T-r)^\theta}d\langle \rwB \rangle_r$ 
and may assume that $f$ is constant on all $[t_{k-1},t_k)$. Then our assumption of item \eqref{Gronwall-discrete} implies 
that \eqref{eqn:gronwall:EK} can be used, which yields to
\[ f(t_{k-1})  \le a e^{ b \int_{]t_{k-1}, T[} \frac{1}{(T - r)^{\theta}} d \langle \rwB \rangle_r }.\]
Using the definition of $\od \langle \rwB \rangle_r$ we get the assertion from
\[   \int_{]t_{k-1}, T[} \frac{1}{(T - r)^{\theta}} \od \langle \rwB \rangle_r 
   = \sum_{m=k}^{n-1} h \frac{1}{(T - t_m)^{\theta}} 
  \le \sum_{m=k}^{n-1} \int_{t_m}^{t_{m+1}}\frac{1}{(T - r)^{\theta}} dr 
   = \int_{t_k}^T\frac{1}{(T - r)^{\theta}} dr. \]
\end{proof} 
\medskip

We  formulate a Gronwall lemma of Volterra type.
Similar inequalities can be found in \cite{henry}.

\begin{lemma}
\label{volterra_gronwall}
Assume $\ep \in ]0,1]$, $\beta>\frac{1}{2}$, a  nonnegative  $g\in L_1([0,T[)$ and $a,b\ge0$  such that for all $t\in [0,T[$
one has
\[ g(t) \le \frac{a}{|T-t|^\frac{1}{2} \Phi_\beta(|T-t|^\ep)} + b \int_t^T \frac{g(s)}{|s-t|^\frac{1}{2}} ds. \]
Then one has
\begin{equation}\label{eqn:volterra_gronwall}
g(t) \le c_\eqref{eqn:volterra_gronwall} \frac{a}{|T-t|^{\frac{1}{2}} \Phi_\beta(|T-t|^\ep)  }
\sptext{1}{for}{1}
t\in [0,T[,
\end{equation}
where $c_\eqref{eqn:volterra_gronwall}=c_\eqref{eqn:volterra_gronwall}(T,\ep,\beta,b)>0$ is non-decreasing in $T$.
\end{lemma}
\begin{proof}  From  the fact that $ \Phi_\beta\ge 1$  we derive by inserting the inequality for $g$ into itself
\begin{align*}
g(t)  &\le  \frac{a}{|T-t|^{\frac{1}{2}}  \Phi_\beta(|T-t|^\ep)  }   + b \int_t^T  \frac{a}{|T-s|^\frac{1}{2}  } \frac{1}{|s-t|^\frac{1}{2}}  ds \\
&\quad   +   b^2 \int_t^T  \left (  \int_s^T \frac{g(u)}{|u-s|^\frac{1}{2}} du \right ) \frac{1}{|s-t|^\frac{1}{2}} ds \\
&=  \frac{a}{|T-t|^{\frac{1}{2}}  \Phi_\beta(|T-t|^\ep)  }  + a b \,  \mathrm{B}(\tfrac{1}{2},\tfrac{1}{2}) \\
&\quad +   b^2   \int_t^T  \left ( \int_t^u \frac{1}{|u-s|^\frac{1}{2}|s-t|^\frac{1}{2}} ds  \right ) g(u) du \\
&\le c \frac{a}{|T-t|^{\frac{1}{2}}  \Phi_\beta(|T-t|^\ep)  }   +   b^2  \,  \mathrm{B}(\tfrac{1}{2},\tfrac{1}{2})   \int_t^T  g(u) du,
\end{align*}
where $c=c(T,\ep,\beta,b)>0$ and $\mathrm{B}$ stands for the beta function. For the last inequality we used that  the function
$t\mapsto |T-t|^{\frac{1}{2}}  \Phi_\beta(|T-t|^\ep)$ is non-increasing, especially that $1\le \frac{T^{\frac{1}{2}}  \Phi_\beta(T^\ep)}{|T-t|^{\frac{1}{2}}  \Phi_\beta(|T-t|^\ep)}.$
 Now a
Gronwall lemma (see, for example, \cite[Lemma 15]{BGGL21}) yields 
\begin{align*}
g(t)  &\le  c \frac{a}{|T-t|^{\frac{1}{2}}  \Phi_\beta(|T-t|^\ep)  }   +  a c  b^2   \,  \mathrm{B}(\tfrac{1}{2},\tfrac{1}{2})  \int_t^T  \frac{1}{|T-u|^{\frac{1}{2}}  \Phi_\beta(|T-u|^\ep)  }  e^{b^2  \, \mathrm{B}(\tfrac{1}{2},\tfrac{1}{2})(u-t) }du.
\end{align*}
Since the integral is uniformly bounded in $t$ we  use  again  $1\le \frac{T^{\frac{1}{2}}  \Phi_\beta(T^\ep)}{|T-t|^{\frac{1}{2}}  \Phi_\beta(|T-t|^\ep)}$ to get  \eqref{eqn:volterra_gronwall}.
\end{proof}


\subsection{Bounded mean oscillation}

We recall the following notation: If $(\Omega,\cF,\P)$ is a probability space
then, by \cref{sec:mphi-spaces},
\[ |f|_{M^0_{\varphi_1}}
   = \inf \left \{ c > 0 : \P(|f| > \lambda) \le e^{1-\frac{\lambda}{c}} \mbox{ for }
       \lambda \ge c \right \}.\]

\begin{df}  \label{contBMO-def}
For  an adapted, {\it continuous},  nonnegative, and nondecreasing  process \linebreak $(A_t)_{t\in [0,T]}$ 
with $A_0\equiv0$ on a stochastic basis
$(\Omega, \cF, \p , (\cF_t)_{t\in [0,T]})$ satisfying the usual conditions we define
$\|A\|_{{\rm BMO}([0,T])}:=\inf c$, 
where the infimum is taken over all $c\ge 0$ such that for all $0\le t <T$ one has
\[ \e [|A_T-A_t|\,|\cF_t] \le c \quad a.s. \]
\end{df}

The John-Nirenberg theorem for {\it continuous} adapted processes $(A_t)_{t\in [0,T]}$  reads as follows:

\begin{thm}[{\cite[p. 176]{Dellacherie:Meyer:B}}]
\label{statement:JN_continuous_time}
There is a $c_\eqref{eqn:statement:JN_continuous_time}>0$
such that for all $t\in [0,T]$ one has
\begin{equation}\label{eqn:statement:JN_continuous_time}
  |A_T-A_t|_{M^0_{\varphi_1}} \le c_\eqref{eqn:statement:JN_continuous_time}  \|A\|_{{\rm BMO}([0,T])}.
\end{equation}
\end{thm}

\begin{df}  \label{discrBMO-def}
For a stochastic basis $(\Omega, \cF, \p , (\cG_n)_{n=0}^N)$ and
an adapted  process $(A_n)_{n= 0}^N$ with $A_0\equiv 0$ we define
$ \|A\|_{{\rm BMO}}:=\inf c$
where the infimum is taken over all $c\ge 0$ such that for all $1\le n \le N$ one has
\[ \e [|A_N-A_{n-1}||\cG_n] \le c \quad a.s. \]
\end{df}

The John-Nirenberg theorem for {\it discrete time} adapted processes $(A_n)_{n=0}^N $ reads as follows:

\begin{thm}[{\cite[Theorem III.1.1]{Garsia:1973}}]
\label{statement:JN_discrete_time}
There is a constant $c_\eqref{eqn:statement:JN_discrete_time}>0$
such that for all $1\le n \le N$   one has
\begin{equation}\label{eqn:statement:JN_discrete_time}
 |A_N-A_{n-1}|_{M^0_{\varphi_1}} \le c_\eqref{eqn:statement:JN_discrete_time}  \|A\|_{{\rm BMO}}. 
\end{equation}
\end{thm}


\subsection{Proof of \texorpdfstring{ \cref{statement:disrcrete-rate}}{proposition87}}
\label{sec:proof:statement:disrcrete-rate}
{\bf Inequality \eqref{eqn:item:1:statement:disrcrete-rate:start-in-0}:}
We use \cref{statement:a-priori-discrete}\eqref{item:3:statement:a-priori-discrete:special-2}
 for $ \eta_1 =\eta_2 =  g( B^{n,t_o,\xi}_T)$ and
\begin{align*}
  f^1(s,x,y,z)&:=f( s\wedge t_{n-1}, x,y,z, [  Y^{n,t_o,\xi}_s], [ Z^{n,t_o,\xi}_s] ) \\
  f^2(s,x,y,z)&:= f( s\wedge t_{n-1},x,y,z,[Y^{t_o,\xi}_s],   [Z^{t_o,\xi}_{s\wedge t_{n-1}}]),
 \end{align*}
to get
\begin{align}
& \E   \left|    Y^{n,t_o,\xi}_{t_{k-1}} - Y^{n,t_o,\xi}_{{{\sf f}_n},t_{k-1}} \right|^2
  + \E  \int_{]t_{k-1}, T[} \left|Z^{n,t_o,\xi}_s- Z^{n,t_o,\xi}_{{{\sf f}_n},s}  \right|^2  d\langle \rwB \rangle_s  \notag \\
&\le  c_{\eqref{eqn:item:3:statement:a-priori-discrete:special-2}}
      \left [  \int_{]t_{k-1}, T]} \cW_2^2([Y^{t_o,\xi}_{s}],  [Y^{n,t_o,\xi}_{{{\sf f}_n},s}])  d\langle \rwB \rangle_s
     +         \int_{]t_{k-1}, T]}   \cW_2^2([Z^{t_o,\xi}_{s\wedge t_{n-1}}], [Z^{n,t_o,\xi}_{{{\sf f}_n},s}]) d\langle \rwB \rangle_s \label{W-Y-Y-n} \right ] \\
& \le 3 c_{\eqref{eqn:item:3:statement:a-priori-discrete:special-2}} \Bigg [ \int_{]t_{k-1}, T]} \cW_2^2([Y^{t_o,\xi}_{s}],  [Y^{t_o,\xi}_{{\sf f}_n,s}]) d\langle \rwB \rangle_s
     +   \int_{]t_{k-1}, T]} \cW_2^{2}([Z^{t_o,\xi}_{s\wedge t_{n-1}}],[Z^{t_o,\xi}_{{\sf f}_n,s\wedge t_{n-1}}]) d\langle \rwB \rangle_s \label{continuous-nwidetilde}  \\
& \quad +     \int_{]t_{k-1}, T]}  \cW_2^2([Y^{t_o,\xi}_{{{\sf f}_n},s}], [Y^{n,t_o,\xi}_{{{\sf f}_n},s}]) \od \langle \rwB \rangle_s
        +     \int_{]t_{k-1}, T]}  \cW_2^{2}([Z^{t_o,\xi}_{{{\sf f}_n},s\wedge t_{n-1}}], [Z^{n,t_o,\xi}_{{{\sf f}_n},s\wedge t_{n-1}}]) \od \langle \rwB \rangle_s
         \label{widetilde-bar-n}  \\
& \quad +   \int_{]t_{n-1}, T]}  \cW_2^{2}([Z^{n,t_o,\xi}_{{{\sf f}_n},s\wedge t_{n-1}}], [Z^{n,t_o,\xi}_{{{\sf f}_n},s}])d\langle \rwB \rangle_s \Bigg ]
   \label{remainder-term}.  
\end{align}
\underline{Estimate of \eqref{continuous-nwidetilde}, first term:} By \cref{prop_tilde_difference}, exploiting that $Y^{t_o,\xi}_{s}= u(s, B^{t_o,\xi}_s)$ and $Y^{t_o,\xi}_{{\sf f}_n,s}=u_{{\sf f}_n}(s, B^{t_o,\xi}_s),$
\[     \int_{]t_{k-1}, T]}  \cW_2^2([Y^{t_o,\xi}_{s}],   [Y^{t_o,\xi}_{{\sf f}_n, s}])d\langle \rwB \rangle_s
   \le \int_{]t_{k-1}, T]}  \e|Y^{t_o,\xi}_{s} -  Y^{t_o,\xi}_{{\sf f}_n,s}|^2 d\langle \rwB \rangle_s
   \le T c_{\eqref{cont_tilde_a_priori}}^2  n^{-2 \left ( \left( \alpha \wedge   \frac{\ep}{2}\right )   + \frac{1}{2}\right )}.
\]
\underline{Estimate of \eqref{continuous-nwidetilde}, second term:}
We rewrite the term as an integral w.r.t.~$ds$ and get
\begin{align*}
&\int_{]t_{k-1}, T]}   \cW_2^{2}([Z^{t_o,\xi}_{s\wedge t_{n-1}}], [Z^{t_o,\xi}_{{\sf f}_n,s\wedge t_{n-1}}])d\langle \rwB \rangle_s  \\
&= \int_{]t_{k-1}, T]}  \cW_2^{2}([Z^{t_o,\xi}_{\os\wedge t_{n-1}}], [Z^{t_o,\xi}_{{\sf f}_n,\os\wedge t_{n-1}}])ds\\
&\le 3 \int_{]t_{k-1}, T]}   \cW_2^{2}([Z^{t_o,\xi}_{\os\wedge t_{n-1}}], [Z^{t_o,\xi}_{s\wedge t_{n-1}}]) ds    + 3\int_{]t_{k-1}, T]}  \cW_2^{2}([Z^{t_o,\xi}_{s\wedge t_{n-1}}], [Z^{t_o,\xi}_{{\sf f}_n,s\wedge t_{n-1}}])ds\\
& \quad +3 \int_{]t_{k-1}, T]}   \cW_2^{2}( [Z^{t_o,\xi}_{{\sf f}_n,s\wedge t_{n-1}}],[Z^{t_o,\xi}_{{\sf f}_n,\os\wedge t_{n-1}}]) ds.
\end{align*}
Proceeding similarly to \eqref{Z-difference}, by \cref{regularity-for-mainBSDE} and \cref{lemma37}  for any $s, s' \in [t_o, T[$ we get
\begin{align} \label{Z-differences}
     \|Z^{t_o,\xi}_{s'}- Z^{t_o,\xi}_{s}\|_{L^2}
&\le \frac{c_{\eqref{nabla u-diff-eq}} \left ( \e   \Ps  {\, |B_{s'}-B_s|^{2\ep} \,} {\frac{1}{2}}^2   \right )^\frac{1}{2} }{|T-(s' \vee s)|^{\frac{1}{2}} \Phi_{\beta }(|T-(s' \vee s)|^\ep)} +
     \frac{c_{\eqref{en:timenabu-eq}} \,  \Ps {\, s'-s \, }  {\frac{\ep}{2}} }{ |T-(s' \vee s)|^{\frac{1}{2}} \Phi_{\beta }(|T-(s' \vee s)|^\ep)}   \notag\\
&\le c_\eqref{Z-differences} \, \kappa_\beta \,
     \frac{ \,     \Ps {\, s'-s \, }  {\frac{\ep}{2}}  }{ |T-(s' \vee s)|^{\frac{1}{2}} \Phi_{\beta }(|T-(s' \vee s)|^\ep)}
\end{align}
with $c_\eqref{Z-differences}:= c_{\eqref{nabla u-diff-eq}} + c_{\eqref{en:timenabu-eq}}$, where we use \cref{Psi-lemma}.
This yields
\begin{align*}
     \int_{]t_{k-1}, T]} \cW_2^{2}([Z^{t_o,\xi}_{\os\wedge t_{n-1}}], [Z^{t_o,\xi}_{s\wedge t_{n-1}}])  ds 
&\le \int_{]t_{k-1}, t_{n-1}]} \left\Vert Z^{t_o,\xi}_{\os}- Z^{t_o,\xi}_{s}  \right\Vert_{L^2}^{2} ds\\
&\le c_{\eqref{Z-differences}}^2 \left [ \kappa_\beta  \left | \frac{T}{n} \, \right |_  {(\frac{\ep}{2},\beta)} \right ]^2  \int_{]t_{k-1}, t_{n-1}]}  \frac{d s}{ |T-\os| \Phi_{\beta }(|T-\os|^\ep)^{2}} \\
&\le c_{\eqref{Z-differences}}^2 c_1\,  n^{-{{\ep}}}    |\log (n+1)|^{{2}\beta},
\end{align*}
for some $c_1=c_1(T,\ep,\beta)>0$, where for  last inequality also \cref{phi_remark} was used (as we integrate only up to $t_{n-1}$ we can shift
the function in the argument and then use  \cref{phi_remark}). By the same arguments we get
\begin{align*}
\int_{]t_{k-1}, T]}   \cW_2^{2}([Z^{t_o,\xi}_{{\sf f}_n,s\wedge t_{n-1}}], [Z^{t_o,\xi}_{{\sf f}_n,\os\wedge t_{n-1}}]) ds
 \le (c'_{\eqref{Z-differences}})^2 c_2  n^{-{{\ep} }}    |\log (n+1)|^{2\beta}
\end{align*}
for some $c_2=c_2(T,\ep,\beta)>0$. The remaining term we  split again into three terms,

\begin{align*}
&\int_{]t_{k-1}, T]}  \cW_2^{2}([Z^{t_o,\xi}_{s\wedge t_{n-1}}], [Z^{t_o,\xi}_{{\sf f}_n,s\wedge t_{n-1}}])ds \\
&\le 3 \int_{]t_{n-1}, T]}  \|Z^{t_o,\xi}_{s\wedge t_{n-1}} - Z^{t_o,\xi}_{s}\|_{L^{2}}^{2} ds + 3 \int_{]t_{k-1}, T]}  \|Z^{t_o,\xi}_{s} - Z^{t_o,\xi}_{{\sf f}_n,s}\|_{L^{2}}^{2} ds \\
& \quad + 3 \int_{]t_{n-1}, T]}  \|Z^{t_o,\xi}_{{\sf f}_n,s} - Z^{t_o,\xi}_{{\sf f}_n,s\wedge t_{n-1}}\|_{L^{2}}^{2} ds \\ \pagebreak 
& \le  3 c^2_{\eqref{cont_tilde_a_priori}} \, n^{-{2((\alpha \wedge \frac{\ep}{2}) +\frac{1}{2})}} +
       3\left [c_{\eqref{Z-differences}}^2 c_1 + (c'_{\eqref{Z-differences}})^2 c_2\right]\,  n^{-{{\ep}}}    |\log (n+1)|^{2\beta},
\end{align*}
where we applied  \cref{cont_tilde_a_priori} for the second term on the RHS and \cref{Z-differences} as before to the first and third term
(again, thanks to  \cref{regularity-for-mainBSDE}  we may use   \cref{Z-differences} for   $Z^{t_0,\xi}_{\sf f_n,s}$ as well, denoting the corresponding constant with $c'$). 
Hence the resulting order for \eqref{continuous-nwidetilde}  is  
\[ n^{-{{\ep}}}    |\log (n+1)|^{{ 2}\beta}.  \]
\underline{Estimate of \eqref{widetilde-bar-n}:}
We apply \cref{rate-singularity}  which gives
\begin{align*}
&       \int_{]t_{k-1}, T]}  \cW_2^2([Y^{t_o,\xi}_{{\sf f}_n,s}], [Y^{n,t_o,\xi}_{{\sf f}_n,s}]) \od \langle \rwB \rangle_s
      + \int_{]t_{k-1}, T]}  \cW_2^{2}([Z^{t_o,\xi}_{{\sf f}_n, s\wedge t_{n-1}}], [Z^{n,t_o,\xi}_{{\sf f}_n,s\wedge t_{n-1}}]) \od \langle \rwB \rangle_s  \\
&\leq T c_\eqref{wasserstein-Y}^2 n^{-2(\alpha \wedge \frac{ \ep}{2})}  \\
&     \quad + c_\eqref{wasserstein-Z}^2 n^{-2(\alpha \wedge \frac{ \ep}{2})} |\log (n+1)|^{2\beta}
      \int_{]t_{k-1}, T]}   \frac{d s}{|T-(\os\wedge t_{n-1})|\Phi_{\beta }(|T-(\os\wedge t_{n-1})|^\ep)^2}  \\
& \leq c_3 [c_\eqref{wasserstein-Y}^2+c_\eqref{wasserstein-Z}^2] n^{-2(\alpha \wedge \frac{\ep}{2})} |\log (n+1)|^{2\beta},
\end{align*}
by using again \cref{phi_remark} in the way before, and where $c_3=c_3(T,\ep,\beta)>0$. \smallskip

\underline{Estimate of \eqref{remainder-term}:} We start with
\begin{align}  \label{eq_finally_we_estimate_1}
    \int_{]t_{n-1}, T]}  \cW_2^{2}([Z^{n,t_o,\xi}_{{\sf f}_n,s\wedge t_{n-1}}], [Z^{n,t_o,\xi}_{{\sf f}_n,s}])d\langle \rwB \rangle_s
&= h \, \cW_2^{2}([Z^{n,t_o,\xi}_{{\sf f}_n, t_{n-1}}], [Z^{n,t_o,\xi}_{{\sf f}_n,T}]) \notag \\
&\le h \e |Z^{n,t_o,\xi}_{{\sf f}_n, t_{n-1}} -  Z^{n,t_o,\xi}_{{\sf f}_n,T}|^2 \notag \\
&\le 2 h \left( \e |Z^{n,t_o,\xi}_{{\sf f}_n, t_{n-1}}|^2  + \e|Z^{n,t_o,\xi}_{{\sf f}_n,T}|^2  \right).
\end{align}
First we show that $|Z^{n,t_o,\xi}_{{\sf f}_n,T}|$ is bounded:
For the solution to \eqref{eq:mainBSDEn} we use the scheme \eqref{z-fn-eqn}-\eqref{y-fn-eqn}.
Denoting by $\widetilde \E$  the expectation w.r.t. $\widetilde \zeta_n,$ independent of $B^{n,t_o,\xi}_{t_{n-1}}$,
by \eqref{z-fn-eqn} we have
\begin{align}
    | Z^{n,t_o,\xi}_{{\sf f}_n,T} |& = \left  |h^{-1/2}\,\e\left[Y^{n,t_o,\xi}_{{\sf f_n},t_n} \,\zeta_{n}\,|\,\m F^{n,t_o,\xi}_{t_{n-1}}\right]\right |
  = | h^{-1/2} \widetilde \e g(B^{n,t_o,\xi}_{t_{n-1}} + \sqrt{h}\widetilde\zeta_n)\tilde \zeta_n| \notag \\
&\le \frac{|g(B^{n,t_o,\xi}_{t_{n-1}} + \sqrt{h})-g(B^{n,t_o,\xi}_{t_{n-1}} -\sqrt{h})|}{2 \sqrt{h}} 
 \le \gho h^{\frac{\ep-1}{2}}. \label{eqn_Zn_upper_bound_g}
\end{align}
To estimate $\e|Z^{n,t_o,\xi}_{{\sf f}_n, t_{n-1}}|^2$ we first use \eqref{y-fn-eqn} and \eqref{z-fn-eqn} to get 
\begin{align*} 
    Z^{n,t_o,\xi}_{{\sf f}_n, t_{n-1}}
& = h^{-1/2} \e\left[ Y^{n,t_o,\xi}_{\f_n, t_{n-1}}\zeta_{n-1}\,|\,\m F^{n,t_o,\xi}_{t_{n-2}}\right]\\ 
& = h^{-1/2} \widetilde\e\left[g(B^{n,t_o,\xi}_{t_{n-2}} +\sqrt{h}\tilde\zeta_{n-1}+ \sqrt{h}\tilde\zeta_n)\widetilde \zeta_{n-1} \right]\\
&\quad+h^{1/2}\,\e\left[{\sf f}_n\left(t_n,B^{n,t_o,\xi}_{t_{n-1}},Y^{n,t_o,\xi}_{\f_n, t_{n-1}},
Z^{n,t_o,\xi}_{\f_n, t_{n}}\right)\zeta_{n-1}\,|\,\m F^{n,t_o,\xi}_{t_{n-2}}\right].
\end{align*}
Like above we have
\begin{align*}
\left |  h^{-1/2} \widetilde\e\left[g(B^{n,t_o,\xi}_{t_{n-2}} +\sqrt{h}\widetilde\zeta_{n-1}+ \sqrt{h}\widetilde\zeta_n)\tilde \zeta_{n-1} \right] \right | \le  \gho h^{\frac{\ep-1}{2}}.
\end{align*}
By Jensen's inequality we estimate the second term
\begin{align*}
&    \e  \left | h^{1/2}\,\e\left[{\f}_n\left(t_n,B^{n,t_o,\xi}_{t_{n-1}},Y^{n,t_o,\xi}_{\f_n, t_{n-1}},
                 Z^{n,t_o,\xi}_{\f_n, t_{n}}\right)\zeta_{n-1}\,|\,\m F^{n,t_o,\xi}_{t_{n-2}}\right]  \right |^2 \\
&\le h \e   \left |f\left(t_{n-1},B^{n,t_o,\xi}_{t_{n-1}},Y^{n,t_o,\xi}_{\f_n, t_{n-1}},
                           Z^{n,t_o,\xi}_{\f_n, t_{n}},[Y^{t_o,\xi}_{t_n}],  [Z^{t_o,\xi}_{t_{n-1}}]\right)\,  \right |^2 \\ 
&\le 2 h f_0^2 +  2h \e   \left |f\left(t_{n-1},B^{n,t_o,\xi}_{t_{n-1}},Y^{n,t_o,\xi}_{\f_n, t_{n-1}},
                           Z^{n,t_o,\xi}_{\f_n, t_{n}},[Y^{t_o,\xi}_{t_n}],  [Z^{t_o,\xi}_{t_{n-1}}]\right) -f_0 \right |^2 \\ 
&\le 2 h f_0^2 +  12h \e   \Bigg [   (\fhot)^2 \frac{ \left|t_{n-1}\right|^{2\alpha}}{T- t_{n-1}} 
                                    + \fhox^2 \E |B^{n,t_o,\xi}_{t_{n-1}}|^{2\ep} \\
&    \hspace{12em}                         + L_f^2  \left (   \E|Y^{n,t_o,\xi}_{\f_n, t_{n-1}}|^2
                                                   + \E |Z^{n,t_o,\xi}_{\f_n, t_{n}}|^2
                                                   + \| Y^{t_o,\xi}_{t_n} \|_{L^2}^2 
                                                   + \| Z^{t_o,\xi}_{t_{n-1}} \|_{L^2}^2 \right )\\
 & = 2 h f_0^2 +  12h \e   \Bigg [   (\fhot)^2 \frac{ \left|t_{n-1}\right|^{2\alpha}}{T- t_{n-1}} 
                                    + \fhox^2 \E |B^{n,t_o,\xi}_{t_{n-1}}|^{2\ep} 
        + L_f^2\left (\E |Z^{n,t_o,\xi}_{\f_n, t_{n}}|^2+\|g(B^{t_o,\xi}_T)\|_{L^2}^2 \right )  \\
&    \hspace{24em}                          + L_f^2  \left (   \E|Y^{n,t_o,\xi}_{\f_n, t_{n-1}}|^2
                                                + \| Z^{t_o,\xi}_{t_{n-1}} \|_{L^2}^2 \right ) \Bigg ].
\end{align*}
Also using \cref{eqn_Zn_upper_bound_g}, one has
\[ 2 h f_0^2 +  12h \left [ (\fhot)^2 \frac{ \left|t_{n-1}\right|^{2\alpha}}{T- t_{n-1}} 
  + \fhox^2 \E |B^{n,t_o,\xi}_{t_{n-1}}|^{2\ep} + L_f^2 \left ( \E |Z^{n,t_o,\xi}_{\f_n, t_{n}}|^2+\|g(B^{t_o,\xi}_T)\|_{L^2}^2 \right )
 \right ]  \le c^2_4 \]
with $c_4=c_4(T,\ep,\alpha,|g|_\ep,g(0),\fhot,\fhox,L_f,f_0,M)<\infty$. From \cref{lemma37} we get
\[     h \| Z^{t_o,\xi}_{t_{n-1}} \|_{L^2}^2 
   \le h \frac{c^2_\eqref{eqn:item:2:lemma37}}{|T-t_{n-1}|^{1-\ep}} 
    =  c^2_\eqref{eqn:item:2:lemma37}  h^\ep
   \le c^2_\eqref{eqn:item:2:lemma37}  T^\ep. \]
The estimate for $\e|Y^{n,t_o,\xi}_{\f_n, t_{n-1}}|^2 $  follows from
\cref{statement:a-priori-discrete}\eqref{item:1:statement:a-priori-discrete:general}
by setting $(\eta^1,f^1) := ( g(B^{n,t_o,\xi}_T), \fn)$  and $(\eta^2,f^2) := (0, 0)$ which gives
\begin{align*} 
     \e|Y^{n,t_o,\xi}_{\f_n, t_{n-1}}|^2 
&\le c_{\eqref{eqn:item:1:statement:a-priori-discrete:general}} \left(  \e |g(B^{n,t_o,\xi}_T)|^2 +  h \e    
     \left | f \left (t_{n-1},B^{n,t_o,\xi}_{t_{n-1}},0,0, [Y^{t_o,\xi}_{t_n}], [Z^{t_o,\xi}_{t_{n-1}}] \right ) \right |^2      \right)
 \le c_5^2
\end{align*}
with $c_5=c_5(T,\ep,\alpha,|g|_\ep,g(0),\fhot,\fhox,L_f,c_{f_0}^{(2)},f_0,M)<\infty$.
Summarizing, the term \eqref{remainder-term} is of order
$h^\ep$.
\medskip

\underline{Estimate of all terms:}
Collecting all terms,  yields to our first assertion as we get
\begin{align*}
 \E  \left| Y^{n,t_o,\xi}_{{\sf f}_n,t_{k-1}} -  Y^{n,t_o,\xi}_{t_{k-1}} \right|^2
+    \E \int_{]t_{k-1}, T[} \left| Z^{n,t_o,\xi}_{{\sf f}_n,s} -   Z^{n,t_o,\xi}_s \right|^2  d\langle \rwB \rangle_s   
\le  c   \frac{|\log (n+1)|^{2\beta}}{n^{2(\alpha \wedge \frac{ \ep}{2})}}.
\end{align*}
\bigskip

{\bf Inequality \cref{eqn:item:2:statement:disrcrete-rate:start-in-x}:}
We get from \cref{statement:a-priori-discrete}\eqref{item:2:statement:a-priori-discrete:special-1} for $\xi:= x$
that
\begin{align*}
&     \left|Y^{n,t_{k-1},x}_{{\sf f}_n,t_{k-1}} -   Y^{n,t_{k-1},x}_{t_{k-1}} \right|^2
      + \E \int_{]t_{k-1}, T[} \left| Z^{n,t_{k-1},x}_{{\sf f}_n,s} -   Z^{n,t_{k-1},x}_s \right|^2  d\langle \rwB \rangle_s   \\
&\le  c_{\eqref{eqn:item:2:statement:a-priori-discrete:special-1}}
      \left[ \int_{]t_{k-1}, T]} \left [ \cW_2^2([Y^{t_o,\xi}_s], [Y^{n,t_o,\xi}_s]) + \cW_2^{2}([Z^{t_o,\xi}_{s\wedge t_{n-1}}], [Z^{n,t_o,\xi}_s]) \right ]  d\langle \rwB \rangle_s \right] \\
&\le   c^2_{\eqref{eqn:item:2:statement:disrcrete-rate:start-in-x}} \frac{|\log (n+1)|^{2\beta}}{n^{2(\alpha \wedge \frac{ \ep}{2})}},
\end{align*}
where we use in the last step that
\begin{align*}
      \cW_2([Y^{n,t_o,\xi}_s], [Y^{t_o,\xi}_s])
& \le  \cW_2([Y^{n,t_o,\xi}_s], [Y^{n,t_o,\xi}_{{\sf f}_n,s}]) +\cW_2([Y^{n,t_o,\xi}_{{\sf f}_n,s}], [Y^{t_o,\xi}_s]), \\
      \cW_2([Z^{n,t_o,\xi}_s], [Z^{t_o,\xi}_{s\wedge t_{n-1}}] )
& \le \cW_2([Z^{n,t_o,\xi}_s], [Z^{n,t_o,\xi}_{{\sf f}_n,s}])  +\cW_2([Z^{n,t_o,\xi}_{{\sf f}_n,s}] ,[Z^{t_o,\xi}_{s\wedge t_{n-1}}]),
\end{align*} 
apply  \eqref{eqn:item:1:statement:disrcrete-rate:start-in-0}  and use the estimate we have done above for \eqref{W-Y-Y-n}.
Finally we use that
\[ Y_{t_{k-1}}^{n,t_{k-1},x} = U^n(t_{k-1},x)
   \sptext{.75}{and}{.75}
   Y_{\fn,t_{k-1}}^{n,t_{k-1},x} = U^n_\fn(t_{k-1},x)
   \sptext{.75}{for}{.75} k=j+1,\ldots,n. \]

{\bf Inequality \eqref{eqn:statement:Lexp_bound_difference_gradient_discrete_case}:}
Finally, we use \cref{statement:JN_discrete_time} to deduce from \eqref{eqn:item:2:statement:disrcrete-rate:start-in-x}
the relation \eqref{eqn:statement:Lexp_bound_difference_gradient_discrete_case}.
For the following we fix $k\in \{j+1,\ldots,n\}$, $x\in \R$, and set
$A_0:= 0$ and
\[     A_l 
   := \begin{cases} 
       0 &: l=0\\
       \int_{]t_{k-1},t_{k-1+l}]} \left| Z^{n,t_{k-1},x}_{{\sf f}_n,s} -   Z^{n,t_{k-1},x}_s \right|^2  d\langle \rwB \rangle_s 
         &: l=1,\ldots,n-k
       \end{cases}.
 \]
Define $\cG_l:= \sigma (\zeta_k,\ldots,\zeta_{k-1+l})$ for $l=1,\ldots,n-k$ and 
$\cG_0$ to be the trivial $\sigma$-algebra.
Then $(A_l)_{l=0}^{n-k}$ is $(\cG_l)_{l=0}^{n-k}$-adapted and for 
$l\in \{1,\ldots,n-k\}$ we get
\begin{align*} 
&    \E [A_{n-k} - A_{l-1}| \cG_l ] \\
& =   2^{-(n-k-l)} \sum_{ \zeta_{k+l}=\pm 1,\ldots,\zeta_{n-1}=\pm 1 }   \int_{]t_{k+l-2},t_{n-1}]} \left| Z^{n,t_{k-1},x}_{{\sf f}_n,s} -   Z^{n,t_{k-1},x}_s \right|^2  
     d\langle \rwB \rangle_s\\
& =     2^{-(n-k-l)} \sum_{ \zeta_{k+l}=\pm 1,\ldots,\zeta_{n-1}=\pm 1 }      \int_{]t_{k+l-2},t_{n-1}]} \left| Z^{n,t_{k+l-2},B_{t_{k+l-2}}^{n,t_{k-1},x}}_{{\sf f}_n,s} -  
       Z^{n,t_{k+l-2},B_{t_{k+l-2}}^{n,t_{k-1},x}}_{s} \right|^2  
     d\langle \rwB \rangle_s\\
&\le     2^{-(n-k-l)} \sum_{ \zeta_{k+l-1}=\pm 1,\ldots,\zeta_{n-1}=\pm 1 }    \int_{]t_{k+l-2},t_{n-1}]} \left| Z^{n,t_{k+l-2},B_{t_{k+l-2}}^{n,t_{k-1},x}}_{{\sf f}_n,s} -  
       Z^{n,t_{k+l-2},B_{t_{k+l-2}}^{n,t_{k-1},x}}_{s} \right|^2  
     d\langle \rwB \rangle_s  \\
&\le 2 c^2_{\eqref{eqn:item:2:statement:disrcrete-rate:start-in-x}}
      \frac{|\log (n+1)|^{2\beta}}{n^{2(\alpha \wedge   \frac{\ep}{2})}}. 
      \end{align*}
In the notation of \cref{discrBMO-def} this means that
\[\|(A_l)_{l=0}^{n-k}\|_{{\rm BMO}} \le 2 c^2_{\eqref{eqn:item:2:statement:disrcrete-rate:start-in-x}}
   \frac{|\log (n+1)|^{2\beta}}{n^{2(\alpha \wedge   \frac{\ep}{2})}}.    \]
Then  \cref{statement:JN_discrete_time} implies that
 \[     |A_{n-k}|_{M^0_{\varphi_1}} 
    \le 2 c_\eqref{eqn:statement:JN_discrete_time} 
        c^2_{\eqref{eqn:item:2:statement:disrcrete-rate:start-in-x}} \frac{|\log (n+1)|^{2\beta}}{n^{2(\alpha \wedge\frac{\ep}{2})}}.  \]
\qed

\subsection{Proof of \texorpdfstring{\cref{statement:estimate_exponential_norm_with_Psi}}{lemma91}}
\label{sec:proof:statement:estimate_exponential_norm_with_Psi}

For the proof we write $\eta=c \theta$ where
 \[ \sup_{p\in [1,\infty[} \frac{\| \theta \|_{L^p}}{p} \le 1.\]
We recall that $\ps r^2 =    r^{2 \ep} \left( \log \Big (\frac{1}{r^{2\ep} \wedge r_\beta} \Big ) \right)^{ 2\beta}$ and
use that it holds for $\gamma,a,b>0$ and $x\ge 1$ that
\[ \frac{1}{(ab)\wedge r_\beta} \le \frac{1}{a\wedge\sqrt{r_\beta}}\frac{1}{b\wedge\sqrt{r_\beta}}
   \sptext{1}{and}{1}
   \log x \le \frac{1}{\gamma} x^{\gamma}. \]
Let $\gamma =   \frac{1 }{2 \beta p}$. Then we get
\begin{align*}
&    \E \ps {c \theta} ^p  \\
&  = \E \left [|c \theta|^{\ep p}\left( \log \Big (\frac{1}{|c \theta|^{2 \ep}\wedge r_\beta} \Big ) \right)^{\beta p}\right ]\\
& \le 2^{(p \beta -1)^+} \left \{ \E  \left [ |c \theta|^{\ep p}\left( \log \Big (\frac{1}{c^{2\ep }\wedge \sqrt{r_\beta}} \Big ) \right)^{\beta p} \right ]
      + \E  \left [ |c \theta|^{\ep p}\left( \log \Big (\frac{1}{|\theta|^{2\ep}\wedge \sqrt{r_\beta}} \Big ) \right)^{\beta p} \right ] \right \} \\
& = c^{\ep p} 2^{(p \beta -1)^+} \left \{  \left| \log \Big (\frac{1}{c^{2\ep}\wedge \sqrt{r_\beta}} \Big ) \right |^{\beta p}  \E  \left [ |\theta|^{\ep p} \right ]
    +  \E  \left [ |\theta|^{\ep p}\left( \log \Big (\frac{1}{|\theta|^{2\ep}\wedge \sqrt{r_\beta}} \Big ) \right)^{\beta p} \right ] \right \}.
\end{align*}
Now, using $\beta p \gamma = \frac{1}{2}$,
\begin{align*}
 \E  \left [ |\theta|^{\ep p}\left( \log \Big (\frac{1}{|\theta|^{2\ep}\wedge \sqrt{r_\beta}} \Big ) \right)^{\beta p} \right ]
&\le \E \left [|\theta|^{\ep p} \frac{1}{\gamma^{\beta p} } \left | \frac{1}{|\theta|^{2 \ep}\wedge r_\beta} \right |^{\beta p\gamma}\right ]\\
&\le \E \left [|\theta|^{\ep p} \frac{1}{\gamma^{\beta p} } \left | \frac{1}{|\theta|^{2 \ep}} + \frac{1}{r_\beta}\right |^{\beta p\gamma}\right ]\\
&\le  \frac{1}{\gamma^{\beta p} } \E \left [|\theta|^{\ep p} \left [ \frac{1}{|\theta|^{{\ep}}} +  \frac{1}{\sqrt{r_\beta}}\right ]\right ]\\
& =  (2 \beta p )^{\beta p} \left [ \E |\theta|^{\ep (p-1)}  +   \E |\theta|^{\ep p} \frac{1}{\sqrt{r_\beta}}\right ].
\end{align*}
Finally, we observe for $\ep (p-1) \ge 1$ that
\[    \left ( (\E |\theta|^{\ep (p-1)})^\frac{1}{\ep (p-1)} \right )^\frac{\ep (p-1)}{p}
   \le (\ep (p-1))^{\frac{\ep (p-1)}{p}}
   \le (\ep p)^{\frac{\ep (p-1)}{p}}.\]
So we get for $p\ge 1 +\frac{1}{\ep}$ that
\begin{multline*}
    \left \| \ps {c \theta}  \right \|_{L^p}^p \\
    \le c^{\ep p} 2^{(p \beta -1)^+}  \left \{
           \left| \log \Big (\frac{1}{c^{2\ep}\wedge \sqrt{r_\beta}} \Big ) \right |^{\beta p}  (\ep p)^{\ep p}
        +  (2 \beta p )^{\beta p} \left [ (\ep (p-1))^{\ep (p-1)} + (\ep p)^{\ep p} \frac{1}{\sqrt{ r_\beta}}  \right ] \right \}.
\end{multline*}
Taking the $p$--th root and using
$\log \left (\frac{1}{c^{2\ep}\wedge \sqrt{r_\beta}}\right )\ge \beta > \frac{1}{2}$
the result follows.  The range $p\ge 1 +\frac{1}{\ep}$  one can extend
to $p\ge 1$ by using the monotonicity of the $L_p$-norm in $p$ and by adding the
factor $(1+\frac{1}{\ep})^{\ep+\beta}$ on the RHS.
\qed


\bibliographystyle{alpha}

\end{document}